\documentclass[a4paper,11pt,reqno]{amsart}
\usepackage{color}

\usepackage{graphicx}
\usepackage{enumerate}
\usepackage{amssymb, amsmath}

\theoremstyle{plain}
\newtheorem{theorem}{Theorem}[section]
\newtheorem{corollary}[theorem]{Corollary}
\newtheorem{lemma}[theorem]{Lemma}
\newtheorem{proposition}[theorem]{Proposition}
\newtheorem{definition}[theorem]{Definition}

\newtheorem*{definition*}{Definition}

\theoremstyle{remark}
\newtheorem{remark}[theorem]{Remark}
\newtheorem{example}[theorem]{Example}

\newtheorem*{claim*}{Claim}
\newtheorem*{remark*}{Remark}
\newtheorem*{example*}{Example}
\newtheorem*{notation*}{Notation}

\def\N{{\mathbb N}}

\def\R{{\mathbb R}}

\newcommand{\LL}{\mathrm{L}}
\newcommand{\RR}{\mathrm{R}}
\newcommand{\HH}{\mathrm{H}}
\newcommand{\Hess}{\mathrm{Hess}}
\newcommand{\Ent}{\mathrm{Ent}}

\newcommand{\Dom}{\mathit{Dom}}
\newcommand{\A}{\mathcal{A}}

\newcommand{\Ric}{\mathrm{Ric}}

\newcommand{\Cpl}{\mathrm{Cpl}}
\newcommand{\Pz}{\mathcal{P}}

\begin{document}

\title{Super-Ricci Flows for Metric Measure Spaces}

\author{ Karl-Theodor Sturm}

\thanks{
\begin{minipage}[l]{0.17\textwidth}
\includegraphics[width=2.4cm]{
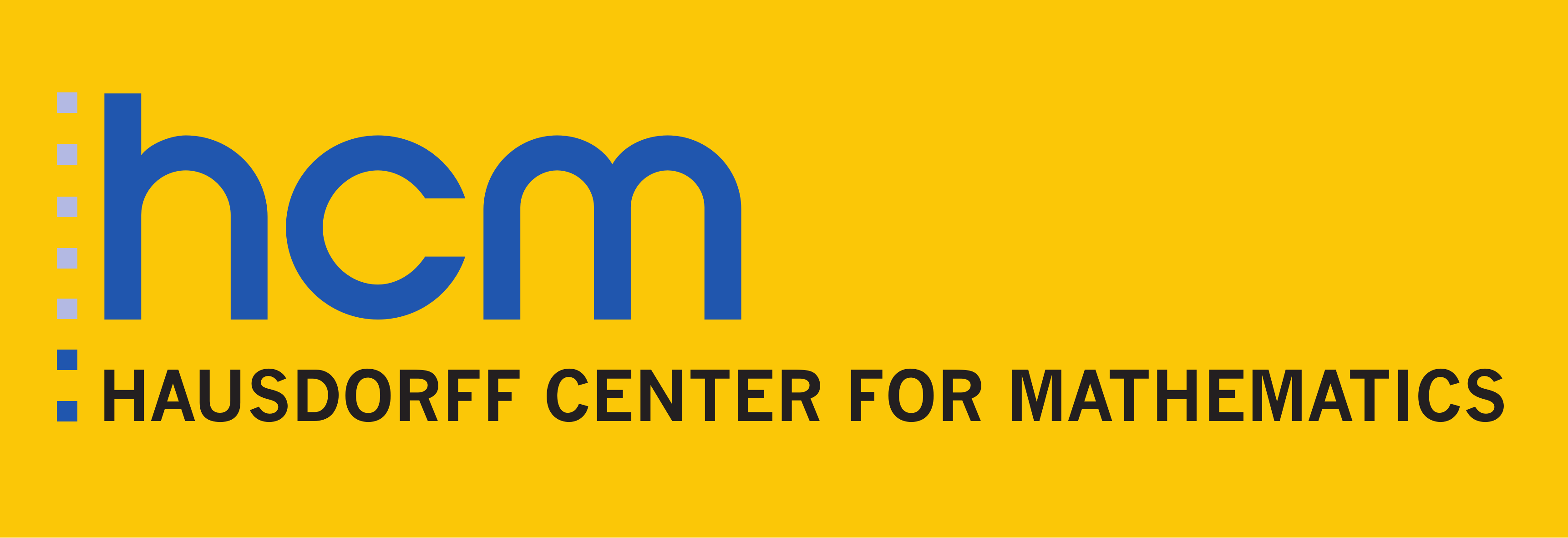}
\end{minipage}
\begin{minipage}[l]{0.65\textwidth}
The author gratefully acknowledges  support by the German Research Foundation through the Hausdorff Center for Mathematics and the Collaborative Research Center 1060 ``The Mathematics of Emergent Effects''.
\end{minipage}
\begin{minipage}[l]{0.14\textwidth}
\includegraphics[width=2.4cm]{
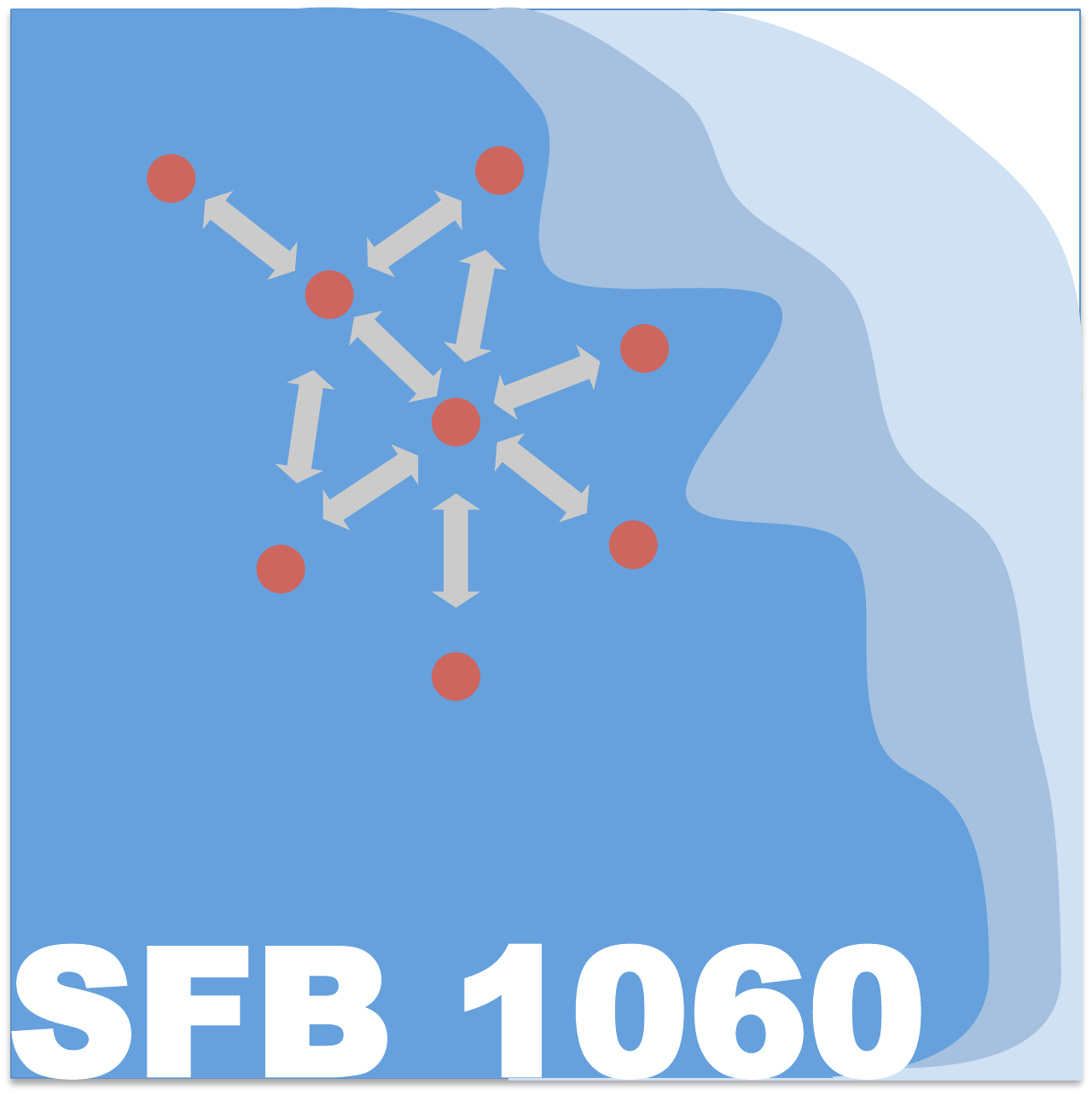}
\end{minipage}
}

\maketitle
\begin{abstract}
We introduce the notions of `super-Ricci flows' and `Ricci flows' for time-dependent families of metric measure spaces $(X,d_t,m_t)_{t\in I}$. 
The former property is proven to be stable under suitable space-time versions of mGH-convergence. Uniformly bounded families of super-Ricci flows are compact.
In the spirit of the synthetic lower Ricci bounds of Lott-Sturm-Villani for static metric measure spaces, the defining property for super-Ricci flows is the `dynamic convexity' of the Boltzmann entropy $\Ent(.|m_t)$ regarded as a functions on the time-dependent geodesic space $(\Pz(X),W_t)_{t\in I}$. For Ricci flows, in addition a nearly dynamic concavity of the Boltzmann entropy is requested.

Alternatively, super-Ricci flows will be studied in the framework of the $\Gamma$-calculus of Bakry-\'Emery-Ledoux and equivalence to gradient estimates will be derived.

For both notions of super-Ricci flows, also enforced versions involving an `upper dimension bound' $N$ will be presented.
\end{abstract}

\setcounter{tocdepth}{1}
\tableofcontents

\section*{Introduction}

\noindent
{\bf A.\ }
Given a manifold $M$ and a smooth 1-parameter family $(g_t)_{t\in I}$ of Riemannian tensors on $M$, we say that the `time-dependent Riemannian manifold' $(M,g_t)_{t\in I}$ evolves as a \emph{Ricci flow} if 
$\Ric_{g_t}=-\frac12\partial_t g_t$
for all $t\in I$. It is called \emph{super-Ricci flow} if instead only $$\Ric_{g_t}\ge-\frac12\partial_t g_t$$ holds true (regarded as inequalities between quadratic forms on the tangent bundle of $(M,g_t)$).
Due to the groundbreaking work of Perelman \cite{Pe1,Pe2,Pe3}, see also \cite{CZ,KL,MoTian}, Ricci flow is known to be  a powerful tool and inspiring source for many questions and new developments.

The objective of this paper are new characterizations in terms of optimal transport for both, Ricci flows and super-Ricci flows. More precisely, super-Ricci flows will be characterized by the \emph{dynamic convexity} of the Boltzmann entropy
$S: (t,\mu)\mapsto \Ent(\mu|\mathrm{vol}_t)$
regarded as a function on the time-dependent geodesic space $(\Pz(M),W_t)_{t\in I}$. This is an innovative concept of convexity in space-time settings. For instance, a smooth function $V$ on $I\times M$ is dynamically convex iff
$$\Hess_{g_t}\, V_t \ge -\frac12\partial_t g_t$$
for all $t\in I$. This equivalently can be re-formulated in terms of metric quantities as follows: 
\begin{equation}\partial_\tau V_t(\gamma^\tau)\Big|_{\tau=1}-\partial_\tau V_t(\gamma^\tau)\Big|_{\tau=0}\ge -\frac1{2}\,\big(\partial_td_t^2\big)(\gamma^0,\gamma^1)
\end{equation}
for all $t\in I$ and all  $g_t$-geodesic $(\gamma^\tau)_{\tau\in[0,1]}$ in $M$.
Here $d_t$ denotes the Riemannian distance induced by $g_t$ and $\partial_td_t$ its variation in time. 
 $\partial_\tau V_t(\gamma^\tau)$ is the slope of  $V_t$ along $\gamma$. Note that we strictly separate space and time derivatives: time derivatives are considered for fixed pairs of points in space and spatial differentiation is performed  along geodesics in fixed time-slices.

Recall that in the time-independent case, a function $V$ on $(M,g)$ is $K$-convex iff $\Hess_gV\ge K\cdot g$ or equivalently iff
$\partial_\tau V(\gamma^\tau)\big|_{\tau=1}-\partial_\tau V(\gamma^\tau)\big|_{\tau=0}\ge  K\cdot d^2(\gamma^0,\gamma^1)$
for all  geodesics $(\gamma^\tau)_{\tau\in[0,1]}$ in $M$.
As for the classical concept of $K$-convexity, the new concept of dynamic convexity easily extends to the setting of time-dependent geodesic spaces.

The innovative notion of dynamic convexity will allow us to study and analyze super-Ricci flows in great generality.

 Even more, in combination with a related concept of nearly dynamic concavity it allows to give a precise meaning of Ricci flows in a quite general framework.
 
 \bigskip
 
\noindent {\bf B. \ }
One of the main contributions of this paper is to present synthetic notions of super-Ricci flows and Ricci flows for time-dependent metric measure spaces 
$(X,d_t,m_t)_{t\in I}$. Here and in the sequel, $X$ is a Polish space equipped with
\begin{itemize}
\item a 1-parameter family of geodesic metrics $d_t$ which generate the topology of $X$
\item and a 1-parameter family of  Borel measures $m_t$, absolutely continuous w.r.t.\ each other, say $m_t=e^{-f_t}m$ for some probability measure $m$ and suitable `weight' functions $f_t$.
\end{itemize}
To simplify the presentation here, let us assume that all the measures $m_t$ are finite, that all the metrics $d_t$ are bounded and  that   $t\mapsto d_t(x,y)$ is  absolutely continuous uniformly in $x$ and $y$.

The basic quantity for the subsequent considerations will be the time-dependent Boltzmann entropy
$$S: (t,\mu)\mapsto \Ent(\mu|m_t)$$
regarded as a function on the time-dependent geodesic space $(\Pz(X),W_t)_{t\in I}$ where for each $t$ under consideration $W_t$ denotes the $L^2$-Wasserstein distance on $\Pz(X)$, the space of probability measures on $X$, induced by the metric $d_t$.

\begin{definition}
 We say that the time-dependent mm-space $\big(X,d_t,m_t\big)_{t\in I}$  is a \emph{super-Ricci flow} if the Boltzmann entropy $S$ is  \emph{strongly dynamical convex} in the following sense:
for   a.e.\ $t\in I$ and every $W_t$-geodesic $(\mu^\tau)_{\tau\in[0,1]}$ with  finite entropy at the endpoints, the function
$\tau\mapsto S_t(\mu^\tau)$ is absolutely continuous on $[0,1]$ and
\begin{eqnarray}\label{intro-dyn-conv-W}
\partial^+_\tau S_t(\mu^{\tau})\Big|_{\tau=1-}-\partial^-_\tau S_t(\mu^{\tau})\Big|_{\tau=0+}
&\ge&- \frac 12\partial_t^- W_{t-}^2(\mu^0,\mu^1).
\end{eqnarray}
Here 
$\partial_t^- W_{t-}^2(\mu^0,\mu^1)=\liminf\limits_{s\nearrow t}\frac1{t-s}\Big( W_t^2(\mu^0,\mu^1)- W_s^2(\mu^0,\mu^1)\big)$ and \\
$\partial^+_\tau S_t(\mu^{\tau})\big|_{\tau=1-}=\limsup\limits_{\tau\nearrow1}\frac1{\sigma-\tau}\big(
S_t(\mu^{\sigma})-S_t(\mu^{\tau})
\big), \quad \partial^-_\tau S_t(\mu^{\tau})\big|_{\tau=0+}=\liminf\limits_{\tau\searrow0}\frac1{\tau-\rho}\big(
S_t(\mu^{\tau})-S_t(\mu^{\rho})
\big)$.
\end{definition}

Among the many examples to which this concept applies, there are two prominent classes which are well studied:
\begin{itemize}
\item time-dependent weighted Riemannian manifolds
\item static metric measure spaces.
\end{itemize}
A time-dependent weighted Riemannian manifold is a family 
$\big(M,g_t,\tilde f_t\big)_{t\in I}$ where $M$ is a manifold equipped with a 1-parameter family 
$(g_t)_{t\in I}$ of smooth Riemannian metric tensors and a family of measurable functions $\tilde f: M\to \R$ (the `weights'). It induces canonically a 
time-dependent mm-space $\big(X,d_t,m_t\big)_{t\in I}$ with $X=M$, 
$d_t$ being the Riemannian distance associated with $g_t$, and $dm_t=e^{-\tilde f_t}d\mathrm{vol}_{g_t}$.

\begin{theorem}
The time-dependent mm-space $\big(X,d_t,m_t\big)_{t\in I}$ 
induced by a time-dependent weighted Riemannian manifold  $\big(M,g_t,\tilde f_t\big)_{t\in I}$
is a super-Ricci flow if and only if for all $t\in I$
\begin{equation}
\Ric_{g_t}+\Hess_{g_t}\tilde f_t\ge -\frac12 \partial_t g_t.
\end{equation}
\end{theorem}

The other important class consists of static spaces.
\begin{example}
A family $\big(X,d_t,m_t\big)_{t\in I}$ of mm-spaces with time-independent metrics $d_t=d$ and time-independent measures $m_t=m$ is a super-Ricci flow if and only if 
the static space $(X,d,m)$ has nonnegative Ricci curvature in the (`strong') sense of Lott-Sturm-Villani. 
\end{example}

\bigskip

\noindent {\bf C.\ }
Let us recall that more than a decade ago, the author \cite{St1,St2} -- and independently Lott \& Villani \cite{LV1} --  introduced a
 synthetic definition of generalized lower bounds for the Ricci curvature of mm-spaces. These definitions of lower Ricci bounds initiated a wave of investigations on analysis and geometry of mm-spaces and led to many new insights and results. Until now, however, no similar concept for upper bounds on the Ricci curvature was available. Surprisingly enough, such a definition can be given easily and very much in the same spirit as for the lower bounds. It will be briefly introduced in Appendix 1 (Chapter 4) of this paper.
The extension to  the dynamic setting will lead to a synthetic notion of sub-Ricci flows and thus also of Ricci flows.

\begin{definition}
A time-dependent mm-space $\big(X,d_t,m_t\big)_{t\in I}$ is called \emph{weak sub-Ricci flow} if for every $\epsilon>0$ there exists a partition $X=\bigcup_i X_i$ such that for all $i$, for any pair of nonempty open sets $U^0,U^1\subset X_i$ and for a.e.\ $t$    there exists
  a $W_t$-geodesic $(\mu^\tau)_{\tau\in[0,1]}$ 
  with 
  $\tau\mapsto S_t(\mu^\tau)$ being (finite and) absolutely continuous on $(0,1)$, 
  $\mathrm{supp}(\mu^\tau)\subset U^\tau$ for $\tau=0,1$ and 
\begin{eqnarray}\label{intro-sub-ricci-flow}
\partial_\tau^+ S_t(\mu^{\tau})\Big|_{\tau=\sigma-}-\partial^-_\tau S_t(\mu^{\tau})\Big|_{\tau=\rho+}
\le- \frac 1{2(\sigma-\rho)}\partial_t^+ W_{t+}^2(\mu^0,\mu^1)
+\epsilon
\end{eqnarray}
for all $0<\rho<\sigma<1$.

A time-dependent mm-space $\big(X,d_t,m_t\big)_{t\in I}$ is called \emph{weak Ricci flow} if it is a super-Ricci flow and and a weak sub-Ricci flow.
\end{definition}

\begin{theorem}
The mm-space $\big(X,d_t,m_t\big)_{t\in I}$ induced by a time-dependent weighted Riemannian manifold
$\big(M,g_t,\tilde f_t\big)_{t\in I}$  is a sub-Ricci flow if and only if for all $t\in I$
\begin{equation}\label{-introriem-sub-rf}
\Ric_{g_t}+\Hess_{g_t}\tilde f_t\le -\frac12 \partial_t g_t.
\end{equation}
\end{theorem}

\bigskip

 \noindent {\bf D.\ }
The concept of super-Ricci flows will allow for a huge class of examples. Occasionally, it might be desirable to focus on a more restrictive class controlled by an additional parameter $N\in [1,\infty]$ (playing the role of some upper bound on the dimension).
 
\begin{definition}
A time-dependent mm-space $\big(X,d_t,m_t\big)_{t\in I}$  is called \emph{super-$N$-Ricci flow} if for
a.e.\ $t\in I$
and every $W_t$-geodesic $(\mu^\tau)_{\tau\in[0,1]}$ with  finite entropy at the endpoints, the function
$\tau\mapsto S_t(\mu^\tau)$ is absolutely continuous on $[0,1]$ and
\begin{eqnarray}
\partial^+_\tau S_t(\mu^{\tau})\Big|_{\tau=1-}-\partial^-_\tau S_t(\mu^{\tau})\Big|_{\tau=0+}
&\ge&- \frac 12\partial_t^- W_{t-}^2(\mu^0,\mu^1)
+\frac1N \big| S_t(\mu^0)-S_t(\mu^1)\big|^2.
\end{eqnarray}
\end{definition}

A super-$N$-Ricci flow is also a super-$N'$-Ricci flow for any $N'>N$. In particular, it is a super-$\infty$-Ricci flow.
Super-Ricci flow is the same as super-$N$-Ricci flow for $N=\infty$.

\begin{theorem}
The mm-space $\big(X,d_t,m_t\big)_{t\in I}$ induced by a time-dependent weighted $n$-dimensional Riemannian manifold
$\big(M,g_t,\tilde f_t\big)_{t\in I}$   is a
 super-$N$-Ricci flow if and only if $N\ge n$ and for all $t\in I$
\begin{equation}
\Ric_{g_t}+\Hess_{g_t}\tilde f_t-\frac1{N-n}\nabla_t \tilde f_t \otimes \nabla_t \tilde f_t
\ge -\frac12 \partial_t g_t.
\end{equation}
In particular for $N=n$ this requires $\tilde f_t$ to be constant. That is, $m_t=C_t\cdot \mathrm{vol}_t$ for each $t$.
\end{theorem}


\bigskip

\noindent {\bf E. \ }
Under weak assumptions (being fulfilled in most cases of application) the defining property of super-Ricci flows allows for important simplifications and equivalent re-formulations:
\begin{itemize}
\item if the spaces $(X,d_t,m_t)$ are known to satisfy  RCD$(-\kappa_t,\infty)$-conditions then it suffices to check the dynamic convexity condition  for `sufficiently many' geodesics. More precisely, it suffices to check that for each pair of measures $\mu^0,\mu^1$ with finite entropy there exists a connecting geodesic for which  \eqref{intro-dyn-conv-W} is satisfied.
\item if the distances $t\mapsto d_t$ satisfy a certain lower exponential growth bound (with `control function' $\lambda\in L^1_{loc}(I)$) 
then \eqref{intro-dyn-conv-W} can be replaced by a  integrated version: integration w.r.t.\ $\tau$ allows to replace the spatial derivatives $\partial^\pm_\tau S_t(\mu^{\tau})$
by differences:
\begin{eqnarray}\label{intro-controlled-W}\nonumber
\frac1\tau\left(S_t(\mu^0))-S_t(\mu^\tau)\right)
&+&\frac1\tau\left(S_t(\mu^1)-S_t(\mu^{1-\tau})\right)\\
&\ge&- \frac 12\partial_t^- W_{t-}^2(\mu^0,\mu^1)- \lambda_t \tau \cdot W^2_t(\mu^0,\mu^1).
\end{eqnarray}
\item 
integration w.r.t.\ $t$ finally allows to replace the time derivative $\partial^-_tW^2_t$
by a difference. 
This 
 results in a weak formulation of super-Ricci flows which is stable under convergence.
\end{itemize}
 The relevant convergence concept here is a space-time version of the 
${\mathbb D}$-convergence introduced by the author in \cite{St1} and being closely related to measured Gromov-Hausdorff convergence.
The ${\mathbb D}_I$-distance between two time-dependent mm-spaces with common time-interval $I$ will be defined as
 \begin{eqnarray}
\lefteqn{ {\mathbb D}_I\Big( (X,d_t,m_t), (\tilde X,\tilde d_t,\tilde m_t)\Big)}\nonumber\\
&=&
 \inf\Big\{
 \left(\frac1{|I|}\int_I\int_{X\times \tilde X} \hat d_t(x,y)^2\, d\hat m(x,y)\,dt\right)^{1/2}
 +
\frac1{|I|}\int_I\int_{X\times \tilde X} | f_t(x)-\tilde f_t(y)|\, d\hat m(x,y)\,dt:\nonumber\\
 && \qquad \quad\hat d_t\in \mathrm{Cpl}(d_t, \tilde d_t) \ \mbox{for a.e. }t\in I, \  \hat m\in \mathrm{Cpl}(m, \tilde m)\Big\}.
 \end{eqnarray}
 Here $m=\frac1{m_T(X)}\,m_T$ and  $\tilde m=\frac1{\tilde m_T(\tilde X)}\,\tilde m_T$ for some fixed $T\in I$ are chosen as `reference' measures such that
 $m_t=e^{-f_t}m$ and $\tilde m_t=e^{-\tilde f_t}\tilde m$  with suitable `weights' $f_t,\tilde f_t$ (for a.e.\ $t\in I$).
 
 \begin{theorem}
The class of averaged
 super-Ricci flows 
 with uniform controls, uniform bounds on diameter and lower Ricci curvature  is  closed w.r.t.\ ${\mathbb D}_I$-convergence.
\end{theorem}

As in the static case, such a stability result together with some additional growth bounds will lead to a compactness result.

 \begin{theorem}
 The class of  averaged super-Ricci flows with uniform controls, uniform bounds on diameter, curvature-dimension and uniform modulus of continuity for $d$ and $f$ is compact w.r.t.\ ${\mathbb D}_I$-convergence.
\end{theorem}

\bigskip

\noindent {\bf F.\ }
In a forthcoming paper \cite{KoS}, accompanying this here, we will study the heat equation on time-dependent mm-spaces. Given $(X,d_t,m_t)_{t\in I}$, under suitable regularity assumptions we will prove existence, uniqueness and regularity for solutions to the heat equation. Among others, this will lead to  equivalent characterization of super-Ricci flows in terms of
\begin{itemize}
\item[$\triangleright$] dynamic convexity of the Boltzmann entropy
\item[$\triangleright$] monotonicity of Wasserstein distances for the dual/backward heat flow acting on probability measures
\item[$\triangleright$]
gradient estimates for the forward heat flow acting on functions
\item[$\triangleright$] a  Bochner inequality involving the time-derivative of the metric.
\end{itemize}
The latter two properties can also be formulated -- and proven to be equivalent --  in the framework of the $\Gamma$-calculus of Bakry-\'Emery-Ledoux. This setting  is partly more general (it also applies to non-reversible operators), partly more restrictive (it requires a dense `smooth' algebra of functions).
In Appendix 2 (Chapter 5) of the current paper, we extend the well-known $\Gamma$-calculus to time-dependent families of diffusion operators $(L_t)_{t\in I}$ defined on a common algebra $\A$. 

 For each $t$, the usual definitions yield square field operators $\Gamma_t$ and its iterate $\Gamma_{2,t}$. The family $(L_t)_{t\in I}$ is called super-Ricci flow if
 \begin{equation}\Gamma_{2,t}\ge\frac12 \partial_t\Gamma_t.\end{equation}
 This will lead to an alternative --  a priori seemingly unrelated -- approach to super-Ricci flows. In a follow-up   paper  \cite{KoS}, it will be linked with the concepts and results from previous chapters. The main result right now provides an important and elementary characterization for super-Ricci flows in terms of gradient estimates for the heat flow.
 
\begin{theorem}
Under appropriate regularity assumptions, a propagator $(P^s_t)_{s\le t}$ for the heat equation $L_t u=\partial_t u$
on $I\times X$ will exist. It will satisfy the gradient estimate
\begin{equation}\Gamma_t(P^s_{t}u)\le P^s_{t}(\Gamma_s(u))\qquad (\forall u\in\A)\end{equation}
if and only if  $(L_t)_{t\in I}$ is a super-Ricci flow.
\end{theorem}

\bigskip

\noindent {\bf G.\ }
From the very first days when Lott \& Villani and the author presented their approach to synthetic bounds on the Ricci curvature for mm-spaces, 
the ubiquitous question was
how this concepts relates to Ricci flow. Besides some isolated results \cite{MT, Lo1, To1,To2, GiM}, however, during all these years no deeper relation -- and in particular no extension to mm-spaces -- could be established.

The study of Ricci flows with irregular  or incomplete Riemannian manifolds as initial data is a broad and hot field of current research, see e.g.\
\cite{GiT, KoL, Lak, MRS, PSSW, Si1, Si2, Yin}.
The recent papers by
Kleiner/Lott \cite{KL1} and 
by Haslhofer/Naber \cite{HN} provide the first weak characterizations of Ricci flows -- the results, however, still being formulated for smooth Riemannian manifolds;
 the former, moreover, essentially being restricted to the 3-dimensional case.

 This paper will lay the foundations for a broad systematic study of (super-)Ricci flows in the context of mm-spaces with various  subsequent publications in preparation:
 \begin{itemize}
 \item
 construction and detailed analysis of the heat flow and Brownian motion on time-dependent mm-spaces \cite{KoS},  \cite{Ko2}; Jordan-Kinderlehrer-Otto gradient flow scheme for the entropy \cite{Ko1};
 \item   geometric functional inequalities on mm-spaces -- in particular, local Poincar\'e, logarithmic Sobolev and dimension-free Harnack inequalities --
  and characterization of super-Ricci flows in terms of them \cite{KoS2};
  \item synthetic approaches to upper Ricci bounds \cite{St5} and rigidity results for Ricci flat metric cones \cite{ErS}.
\end{itemize}

\bigskip

\emph{Preliminary remarks.}
Throughout this paper, we use $\partial_t$ as a short hand notation for $\frac{d}{dt}$.
Moreover, we put
$$\partial_t^+u(t)=\limsup_{s\to t}\frac1{t-s}(u(t)-u(s)),\quad
\partial_t^-u(t)=\liminf_{s\to t}\frac1{t-s}(u(t)-u(s)),$$
$$\partial_t u(t+)=\lim_{s\searrow t}\frac1{t-s}(u(t)-u(s)),\quad
\partial_t u(t-)=\lim_{s\nearrow t}\frac1{t-s}(u(t)-u(s)),$$
and combinations of them.
We use the abbreviations \emph{usc, lsc,} and \emph{ac} for \emph{upper semicontinuous, lower semicontinuous} and \emph{absolutely continuous}, resp.

We denote by
 $\chi^{\tau,\sigma}=\min\{ \sigma(1-\tau),\tau(1-\sigma)\}$  the Green function on $[0,1]$
and   put
\begin{equation*}
\Lambda^{\tau,\sigma}= \frac1\tau(\chi^{\tau,\sigma}+\chi^{1-\tau,\sigma})=
\left\{
\begin{array}{ll}
\frac \sigma\tau, & \sigma\in[0,\tau]\\
1, & \sigma\in [\tau,1-\tau]\\
\frac{1-\sigma}\tau, & \sigma\in[1-\tau,1]
\end{array}
\right.
\end{equation*}
for $\tau\in [0,\frac12]$ and $\sigma\in [0,1]$.
In the sequel, lower indices always indicate `time' parameters (typically $r,s,t$) whereas upper indices will denote  `curve' parameters (typically $\sigma,\tau,a,b$).

\section{Dynamic Convexity }

\subsection{Convexity}
This chapter is devoted to introducing a new concept of convexity (`dynamic convexity'). Before doing so, let us recapitulate some basic facts on convex functions.

A function $u: (\sigma,\tau)\to\R$ is called \emph{convex} iff 
\begin{equation}\label{convex}
u(b)\le \frac{c-b}{c-a} \cdot u(a) +\frac{b-a}{c-a}\cdot u(c)
\end{equation}
for all $\sigma<a<b<c<\tau$.
Every such convex function $u$ is absolutely continuous on $(\sigma,\tau)$ and \emph{differentiable} in a.e.\ point $b\in (\sigma,\tau)$ (more precisely, in all but at most countably many points). 
In particular,
\begin{equation*}
u(c)-u(a)=\int_a^c u'(b)\,db
\end{equation*}
for all $\sigma<a<b<c<\tau$.
In \emph{each} point $b\in (\sigma,\tau)$ both the right and left derivative exist. Each of them is a non-decreasing function of $b$.

Let us add the remarkable observation of \emph{Sierpinski} that $u: (\sigma,\tau)\to\R$ is convex if and only if it is {Lebesgue measurable} and if \eqref{convex} holds for all  $\sigma<a<c<\tau$ and $b=\frac{a+c}2$
(`midpoint convexity').

\begin{lemma} 
For any  function $u: (\sigma,\tau)\to\R$, the following are equivalent 
\begin{itemize}
\item[(i)] $u$ is convex
\item[(ii)] $u$ is absolutely continuous with  $u'$ being non-decreasing on its domain of definition
(or, equivalently,  with non-decreasing $b\mapsto \partial^+_bu(b)$ or with non-decreasing $b\mapsto \partial^-_bu(b)$).
\item[(iii)] $u$ is absolutely continuous with $u''\ge0$ in the sense of distributions
\end{itemize}

\begin{proof} ``(i) $\Rightarrow$ (ii)'' was already stated. Let us prove the implication ``(ii) $\Rightarrow$ (i)''.
Absolute continuity of $u$ and -- for the last inequality -- monotonicity of the derivative imply
\begin{eqnarray*}
\lefteqn{(c-b)u(a)+(b-a)u(c)-(c-a)u(b)}\\
&=&
-(c-b)\int_a^b u'(t)dt+(b-a)\int_b^c u'(t)dt\\
&=&(c-b)(b-a)\int_0^1 \left[ -u'\left( a+(b-a)t  \right)+u'\left(b+(c-b)t  \right)\right]dt\\
& \le&0.
\end{eqnarray*}
To apply the monotonicity assumption we used the fact that $a+(b-a)t\le b+(c-b)t$.

For ``(ii) $\Rightarrow$ (iii)'' it suffices to observe that the distributional derivative of a non-decreasing function is nonnegative.
For the converse, note that a nonnegative distribution (like $u''$) is a measure and thus - by integration - gives rise to a nondecreasing function (which in this case is $u'$).
\end{proof}
\end{lemma}

A function defined on a closed interval, say $u: [\sigma,\tau]\to\R$, is {convex} iff 
\eqref{convex} holds for all $\sigma\le a<b<c\le\tau$.
Equivalently, $u$ is convex if and only if it is \emph{upper semicontinuous} (`usc') on $[\sigma,\tau]$ and absolutely  continuous (`ac') on $(\sigma,\tau)$ with  $u'$ being non-decreasing on its domain of definition.

A straightforward extension of the concept of convexity is $K$-convexity with some number $K\in\R$.

\begin{lemma} \label{K-conv-lem}
For any  $K\in \R$ and any function $u: (\sigma,\tau)\to\R$, the following are equivalent 
\begin{itemize}
\item[(i)] 
For all $\sigma<a<b<c<\tau$
\begin{equation}\label{K-convex}
u(b)\le \frac{c-b}{c-a} \cdot u(a) +\frac{b-a}{c-a}\cdot u(c)-\frac K2(c-b)(b-a)
\end{equation}

\item[(ii)] $u$ is ac and for all $a<c$ for which $u'$ exists
\begin{equation*}
u'(a)\le u'(c)+ K(c-a)
\end{equation*}
\item[(iii)] $u$ is ac with $u''\ge K$ in the sense of distributions.
\end{itemize}
A function with any these properties is called $K$-convex.
\end{lemma}
The \emph{proof} is just a slight variation of the previous one and is left to the reader.

\begin{remark}
In \cite{EKS} the authors introduced the concept of $(K,N)$-convexity which also will play an important role throughout this presentation. It is a reinforcement of $K$-convexity involving an additional parameter $N\in\R_+$. A
 function $u: (\sigma,\tau)\to\R$ will be called $(K,N)$-convex iff it is absolutely continuous and if
 $$u''\ge K+\frac1N (u')^2$$
 in the sense of distributions.

\end{remark}

\bigskip

Let us now consider the concept of convex functions in the setting of geodesic spaces.
There exist several similar but slightly different definitions. The common basic idea behind these definitions is that convexity is requested `along geodesics'. That is $V:X\to\R$ will be called convex iff $u=V\circ \gamma: [0,1]\to\R$ is convex. The difference in the definitions essentially boils down to the question whether this is requested for all geodesics or for `sufficiently many' geodesics.
Another difference in the definitions arises from requesting the convexity inequality \eqref{convex} either for all triples $a,b,c$ of points within an interval, say $[0,1]$, or only for the triples $0,b,1$.

\medskip

A \emph{geodesic space} in the sequel always will mean a complete metric space $(X,d)$ 
such for each pair of points $x,y\in X$ there exists a curve $\gamma: [\sigma,\tau]\to X$ of length $d(x,y)$ which connects them (i.e.\ with $\gamma^\sigma=x$ and $\gamma^\tau=y$).
\emph{Geodesics} are curves of constant speed which (globally) minimize length.

 Given a function $V:  X\to (-\infty,\infty]$, its domain of finiteness will be denoted by
$\Dom(V)=\{x\in X: \ V(x)<\infty\}$.

\begin{definition}
Given a geodesic space $(X,d)$ and a number $K\in\R$, a function $V:  X\to (-\infty,\infty]$  is called
\begin{itemize}

\item
\emph{weakly $K$-convex} if for  every pair of points $x^0,x^1\in\Dom(V)$   there exists
  a geodesic $(\gamma^\tau)_{\tau\in[0,1]}$ connecting $x^0$ and $x^1$  such that for all $\tau\in[0,1]$
\begin{equation}
V(\gamma^\tau)\le (1-\tau)V(x^0)+\tau V(x^1)-K\tau(1-\tau)d^2(x^0,x^1)/2;
\end{equation}

\item
\emph{properly $K$-convex} if for  every pair of points $x^0,x^1\in \Dom(V)$   there exists
  a geodesic $(\gamma^\tau)_{\tau\in[0,1]}$ connecting $x^0$ and $x^1$  such that for all $\rho,\tau,\sigma\in[0,1]$ with $\rho<\tau<\sigma$
\begin{equation}\label{prop-conv}
V(\gamma^\tau)\le \frac{\sigma-\tau}{\sigma-\rho}V_t(\gamma^\rho)+\frac{\tau-\rho}{\sigma-\rho}V(\gamma^\sigma)-K{(\tau-\rho)(\sigma-\tau)}
d^2(x^0,x^1)/2;
\end{equation}

\item
\emph{strongly $K$-convex} if for  every geodesic $(\gamma^\tau)_{\tau\in[0,1]}$  and for every $\tau\in[0,1]$
\begin{equation}\label{stat-conv}
V(\gamma^\tau)\le (1-\tau)V(\gamma^0)+\tau V(\gamma^1)-K\tau(1-\tau)d^2(x^0,x^1)/2.
\end{equation}
\end{itemize}
In the case $K=0$, the function $V$ is simply called weakly convex or properly convex or strongly convex, resp.
\end{definition}

\begin{remark*}
One might study these concepts also in a more general setting where $d$ is just 
 a symmetric function on $X\times X$ with values in $[0,\infty]$ which satisfies the triangle inequality and which vanishes  on the diagonal.
Replacing $X$ by the quotient space $X':=X/d$ if necessary, one can achieve that $d$ only vanishes on the diagonal.
If $d$ is no longer assumed to be finite then in the above definitions the existence of connecting geodesics should be requested only for pairs of points with finite distance.
\end{remark*}

Obviously, strong $K$-convexity implies proper $K$-convexity which in turn implies weak convexity.
Surprisingly,  under weak regularity assumptions also the converse  implications hold true.

\begin{theorem}\label{mid-weak-str} Assume that $(X,d)$ is a locally compact geodesic space and $V:  X\to (-\infty,\infty]$ a lower semicontinuous function which in addition is upper semicontinuous along all geodesics with endpoints in  $\Dom(V)$.
Then the following are equivalent
\begin{itemize}
\item[(i)] $V$ is properly $K$-convex

\item[(ii)] $V$ is wekaly $K$-convex

\item[(iii)] $V$ is weakly midpoint $K$-convex in the sense that for  every pair of points $x^0,x^1\in \Dom(V)$   there exists
  a midpoint $x^{1/2}$ of $x^0$ and $x^1$  such that 
\begin{equation}\label{mp-conv}
V(\gamma^{1/2})\le \frac12 V(x^0)+\frac12V(x^1)-\frac K8d^2(x^0,x^1);
\end{equation}
\end{itemize}

\end{theorem}

This result and the subsequent proof is due to the author. It has been presented  in several private communications to colleagues, e.g. in May 2012  during the
ERC Summer School ``Analysis and Geometry in Metric Measure Spaces'' in Pisa. Since then, it has been used (and partly reproduced) in various publications.

\begin{remark*}The local compactness of $X$ can be replaced by the weaker requirement that for each $r,x$ and $\lambda$ the set
$\{V\le \lambda\}\cap \overline B_r(x)$
is compact.
\end{remark*}

\begin{proof} It remains to prove the implication ``(iii) $\Rightarrow$ (i)''.
Fix $x^0,x^1\in\Dom(V)$ and consider the set of midpoints 
$$Z=\{z\in X: d(x,z)=d(y,z)=\frac12 d(x,y)\}.$$
Among this bounded, closed set of midpoints we select one which minimizes $V$. More precisely, local compactness of $X$ implies compactness of $Z$. Together with lower semicontinuity of $V$ it guarantees the existence of a minimizer of $V$ in $Z$. Call this midpoint  $x^{1/2}$.
In the next step define analogously a midpoint of $x^0$ and $x^{1/2}$ with minimal value of $V$, call it $x^{1/4}$, and 
 a midpoint of  $x^{1/2}$ and $x^1$ with minimal value of $V$, call it $x^{3/4}$ .
 
 The crucial point of this approach is that this construction implies that the inequality \eqref{mp-conv}
 now also holds true for the triple $x^{1/4}, x^{1/2},x^{3/4}$ in the place of the triple $x^0,x^{1/2},x^1$.
 Indeed, weak midpoint $K$-convexity applied to the pair $x^{1/4}, x^{3/4}$ provides a midpoint of them, say $\tilde x^{1/2}$, such that
 $x^{1/4}, \tilde x^{1/2},x^{3/4}$ satisfies inequality \eqref{mp-conv}.
But the previously selected point $x^{1/2}$ is also a midpoint of $x^{1/4}, x^{3/4}$. Indeed, it is a minimizer of $V$ among all these midpoints.
Thus \eqref{mp-conv} is satisfied for the triple $x^{1/4}, x^{1/2},x^{3/4}$.

The midpoint construction --  and the previous argumentation -- will now be iterated to define a family of points $x^\tau, \tau=k2^{-n}$ for $k=0,1,\ldots,2^n$ and $n\in\N$. All these points will lie on a geodesic connecting $x^0$ and $x^1$. Taking the metric completion of this `dyadic curve' yields the requested geodesic $(x^\tau)_{\tau\in[0,1]}$. 
Iteration of the midpoint inequality \eqref{mp-conv} will yield the convexity inequality \eqref{prop-conv}
for dyadic $\rho,\tau,\sigma\in[0,1]$.
Continuity of $V$ along the geodesic $(x^\tau)_{\tau\in[0,1]}$ finally allows to extend this property to all real triples $\rho,\tau,\sigma\in[0,1]$.
\end{proof}

An alternative version which completely avoids any compactness assumption reads as follows.

\begin{proposition} Assume that $V:  X\to \R$ is continuous along geodesics and weakly midpoint $K$-convex.
Then it is properly $K'$-convex for each $K'<K$.
\end{proposition}

\begin{proof}
We essentially follow the previous argumentation. However, due to lack of compactness we cannot insist that $x^{1/2}$ is a minimizer of $V$ among the midpoints. Instead, we will choose a midpoint which minimizes $V$ up to $\epsilon/2$. In all the subsequent iterative steps we allow small errors: the point $x^{(2i+1)/2^n}$ will be chosen among the midpoints of $x^{(2i)/2^n}$  and $x^{(2i+2)/2^n}$ to minimize $V$ up to
$\epsilon 4^{-n}$.
It finally 
will yield the convexity inequality \eqref{prop-conv} with an extra term $\epsilon$. Choosing $\epsilon$ sufficiently small (depending on $(K-K') \,d^2(x^0,x^1)$) will prove the claim.
\end{proof}

\begin{proposition}\label{slope-conv-str}
For every $V:  X\to (-\infty,\infty]$  the following are equivalent
\begin{itemize}
\item[(i)] $V$ is strongly $K$-convex

\item[(ii)] 
for  every geodesic $(\gamma^\tau)_{\tau\in[0,1]}$ 
  with
$\gamma^0,\gamma^1\in\Dom(V)$ the function $\tau\mapsto V(\gamma^\tau)$ is usc on $[0,1]$, ac on $(0,1)$, and
  \begin{equation}
\partial^+_\tau
V(\gamma^{1-})\ge 
\partial^-_\tau
V(\gamma^{0+})+K
d^2(\gamma^0,\gamma^1).
\end{equation}
\end{itemize}
\end{proposition}

\begin{proof} 
Note that an application of (ii) to the restriction of $\gamma$ to the interval $[\rho,\sigma]$ yields
 \begin{equation*}
\partial^+_\tau
V(\gamma^{\tau})\big|_{\tau=\sigma-}\ge 
\partial^-_\tau
V(\gamma^{\tau})\big|_{\tau=\rho+}+\frac{K}{\sigma-\rho}
d^2(\gamma^0,\gamma^1).
\end{equation*}
Thus the claim easily follows from Lemma \ref{K-conv-lem}.
\end{proof}

\begin{proposition}\label{slope-conv-weak}
For every $V:  X\to (-\infty,\infty]$  and $K\in\R$ the implications 
  ``(i) $\Rightarrow$ (ii) $\Rightarrow$ (iii)''
hold true for the following assertions:
\begin{itemize}
\item[(i)] $V$ is weakly $K$-convex

\item[(ii)] there exists $\lambda\in\R$ such that
for  every $x^0,x^1\in\Dom(V)$ there exists a geodesic $(\gamma^\tau)_{\tau\in[0,1]}$ 
 connecting $x^0$ and $x^1$  with
\begin{equation}\label{w-dyn-conv}
V(x^0)+V(x^1)-
V(\gamma^\tau)-V(\gamma^{1-\tau})\ge \big(K\tau - \lambda\tau^2\big) d^2(x^0,x^1)
\end{equation}
for all $\tau\in[0,1/2]$
\item[(iii)] $V$ is weakly midpoint  $K'$-convex for each $K'<K$.
\end{itemize}
\end{proposition}

\begin{proof}  To see   ``(i) $\Rightarrow$ (ii)'' simply add up the convexity inequalities with parameters $\tau$ and $1-\tau$ to obtain
 $$
V(\gamma^\tau)+V(\gamma^{1-\tau})\le V(\gamma^0)+V(\gamma^1)-K\tau(1-\tau)d^2(x^0,x^1).$$
Choosing $\lambda=-K$ yields the claim.

 For ``(ii) $\Rightarrow$ (iii)'', fix $x^0,x^1$ and $n\in\N$. For $i=0,\ldots,n-1$, construct iteratively pairs $x^{(i+1)/(2n)}$, $x^{1-(i+1)/(2n)}$ of points on a connecting geodesic from $x^0$ to $x^1$ by applying  \eqref{w-dyn-conv} with $\tau=\frac1{2n-i)}$ to the previous pairs
 $x^{i/(2n)}$, $x^{1-i/(2n)}$. Then
 \begin{equation*}
V(x^{i/(2n)})+V(x^{1-i/(2n)})-
V(x^{(i+1)/(2n)})-V(x^{1-(i+1)/(2n)})\ge \Big( K\frac{n-i}{2n^2}- \lambda\frac1{4n^2}\Big) d^2(x^0,x^1).
\end{equation*}
 Adding up these estimates yields
  \begin{equation*}
V(x^0)+V(x^1)-2V(x^{1/2})\ge \Big( K\frac{n(n-1)}{4n^2}- \lambda\frac1{4n}\Big) d^2(x^0,x^1).
\end{equation*}
 Since the RHS converges to $\frac K4 d^2(x^0,x^1)$ as $n\to\infty$, this proves the claim.
\end{proof}

In many cases of applications which we have in mind, the functions under considerations are known to share regularity properties of semiconvex functions.

\begin{definition}\label{upp-reg}
 We say that a function $V:  X\to (-\infty,\infty]$ is \emph{upper regular} if for each geodesic $\gamma: [0,1]\to X$ the composed function $u=V\circ \gamma:[0,1]\to (0,\infty]$ is usc on $[0,1]$,
 ac on $(0,1)$, and satisfies
\begin{equation*}\label{reg-tau}
\partial^+_\tau u(\tau-)\le \partial^-_\tau u(\tau+)
\end{equation*}
for each $\tau\in(0,1)$ as well as
\begin{equation*}\label{reg-01}
\liminf_{\tau\nearrow 1}\partial^-_\tau u(\tau+)\le \partial^+_\tau u(1-),\qquad
\limsup_{\tau\searrow 0}\partial^+_\tau u(\tau-)\ge \partial^-_\tau u(0+).
\end{equation*}
\end{definition}
Actually, one of the latter two conditions is redundant: they imply each other by time reversal of $\gamma$. More explicitly, the first of them states that
\begin{equation*}
\liminf_{\tau\nearrow 1}\liminf_{\sigma\searrow \tau}\frac1{\sigma-\tau}\Big[u(\sigma)-u(\tau)\Big]\le
\limsup_{\sigma\nearrow 1}\frac1{1-\sigma}\Big[u(1)-u(\sigma)\Big]
\end{equation*}
whereas the explicit version of \eqref{reg-tau} is
\begin{equation*}
\limsup_{\sigma\nearrow \tau}\frac1{\tau-\sigma}\Big[u(\tau)-u(\sigma)\Big]\le
\liminf_{\sigma\searrow \tau}\frac1{\sigma-\tau}\Big[u(\sigma)-u(\tau)\Big].
\end{equation*}

\begin{remark*}
Every strongly $K$-convex function on a geodesic space is upper regular.
\end{remark*}

\begin{theorem}\label{slope-conv}
Assume that $V:  X\to (-\infty,\infty]$ an upper regular function on a geodesic space $(X,d)$.
Then the following are equivalent
\begin{itemize}
\item[(i)] $u$ is strongly $K$-convex

\item[(ii)] $u$ is weakly $K$-convex

\item[(iii)] 
for  every pair of points $x^0,x^1\in \Dom(V)$   there exists
  a geodesic $(\gamma^\tau)_{\tau\in[0,1]}$ connecting $x^0$ and $x^1$  such that
\begin{equation*}
\partial^+_\tau
V(\gamma^{1-})\ge 
\partial^-_\tau
V(\gamma^{0+})+K
d^2(x^0,x^1)
\end{equation*}

\item[(iv)] 
for  every geodesic $(\gamma^\tau)_{\tau\in[0,1]}$ 
  with
$\gamma^0,\gamma^1\in\Dom(V)$
  \begin{equation*}
\partial^+_\tau
V(\gamma^{1-})\ge 
\partial^-_\tau
V(\gamma^{0+})+K
d^2(\gamma^0,\gamma^1).
\end{equation*}
\end{itemize}

\end{theorem}

\begin{proof} The
 implication ``(i) $\Rightarrow$ (ii)'' is trivial. To see   ``(ii) $\Rightarrow$ (iii)'' simply add up the convexity inequalities with parameters $\tau$ and $1-\tau$ to obtain
 $$\frac1\tau\Big[V(\gamma^1)-V(\gamma^{1-\tau})\Big]\ge
 \frac1\tau\Big[V(\gamma^\tau)-V(\gamma^{0})\Big]+
K(1-\tau)d^2(x^0,x^1).$$
Passing to the limit $\tau\to0$ yields the claim.

 For   ``(iii) $\Rightarrow$ (iv)'' let an arbitrary geodesic $\gamma$ be given (with finite values of $V$ at the endpoints).
 For suitable sequences $\sigma_n\searrow0$ and $\tau_n\nearrow1$ choose geodesics $\tilde\gamma_n: \to X$ connecting $\gamma^{\sigma_n}$ to
 $\gamma^{\tau_n}$ for which property (iii) (with appropriate rescaling) is satisfied:
   \begin{equation*}
\partial^+_\tau
V(\tilde\gamma^{\tau_n-})\ge 
\partial^-_\tau
V(\tilde\gamma^{\sigma_n+})+K (\tau_n-\sigma_n)
d^2(\gamma^0,\gamma^1)
\end{equation*}
 
Let $\gamma_n: [0,1]\to X$ denote the geodesic which on $[\sigma_n,\tau_n]$ coincides with $\tilde\gamma_n$ and on the rest with $\gamma$. (It is easy to see that this indeed is a geodesic.)
Since $V$ is upper regular along $\gamma$ we firstly conclude that
 \begin{equation*}
\partial^-_\tau
V(\gamma^{\tau_n+})\ge 
\partial^+_\tau
V(\gamma^{\sigma_n-})+K (\tau_n-\sigma_n)
d^2(\gamma^0,\gamma^1)
\end{equation*}
and then secondly -- by considering the limit $n\to\infty$ -- that
 \begin{equation*}
\partial^+_\tau
V(\gamma^{1-})\ge 
\partial^-_\tau
V(\gamma^{0+})+K
d^2(\gamma^0,\gamma^1).
\end{equation*}
The
 implication ``(iv) $\Rightarrow$ (i)'' follows from Lemma \ref{K-conv-lem}.
\end{proof}


\subsection{Time-dependent Geodesic Spaces}

 For the subsequent discussion, our basic setting will be a
space $X$ equipped with a 1-parameter family of geodesic metrics
$(d_t)_{ t\in I}$ where $I\subset\R$ is a fixed interval which is assumed to be open from the left.
(More generally, one might allow $d_t$ to be pseudo metrics  where the existence of connecting geodesics is only requested for pairs $x,y\in X$ with $d_t(x,y)<\infty$.)

Given a function
 $\gamma: [\sigma,\rho]\to X, \tau\mapsto \gamma^\tau$,   its \emph{action} 
 (w.r.t.\ $d_t$) is given by
 \begin{equation}\label{action}{\frak a}_t(\gamma)=\sup\Big\{\sum_{i=1}^{k}\frac1{\tau_{i}-\tau_{i-1}}d_t^2\left(\gamma^{\tau_{i-1}}, \gamma^{\tau_i}\right):
 \ \sigma\le \tau_0<\tau_1<\ldots<\tau_k\le\rho, k\in\N\Big\}.
\end{equation}
(Recall that this sum $\sum_{i=1}^{k}\ldots$ is increasing if the partition $\vec\tau$ is getting finer.)
If the action is finite, then $\gamma$ is $d_t$-continuous (i.e. it is a `curve') and 
  the \emph{infinitesimal action} 
\begin{equation*}\label{inf-action}
{\frak g}_t^\tau(\gamma)= \lim_{\tau'\to \tau}\left|\frac{d_t(\gamma^\tau,\gamma^{\tau'})}{\tau-\tau'}\right|^2
\end{equation*}
exists for a.e.\ $\tau\in [\sigma,\rho]$. It is the square of the metric derivative w.r.t.\ $d_t$.
Moreover, 
$
{\frak a}_t(\gamma)=
\int_\sigma^\rho {\frak g}_t^\tau(\gamma)d\tau.
$
The curve $\gamma$ is called \emph{$d_t$-geodesic} if ${\frak a}_t(\gamma)=d_t^2(\gamma^0,\gamma^1)$. Note that in this case the  sum in \eqref{action} does not depend on the partition $\vec\tau=\{\tau_0,\ldots, \tau_k\}$.


For fixed $t$ and for a fixed $d_t$-geodesic $\gamma$ let us now consider the quantity
\begin{equation*}
\sum_{i=1}^{k}\frac1{\tau_{i}-\tau_{i-1}}\frac1{t-s}\left[d_t^2\left(\gamma^{\tau_{i-1}}, \gamma^{\tau_i}\right)
-d_s^2\left(\gamma^{\tau_{i-1}}, \gamma^{\tau_i}\right)\right]
\end{equation*}
for any partition $\vec\tau=\{\tau_0,\ldots, \tau_k\}$ of $[\sigma,\rho]$.
For any $s<t$ it is decreasing with $\vec\tau$ getting finer. (For $s>t$ it is increasing in $\vec\tau$.)
Thus also the quantity ${\frak b}_t^{\vec\tau}(\gamma)$ is decreasing with increasing $\vec\tau$ where we put
${\frak b}_t^{\vec\tau}(\gamma)=-\infty$ if $\partial_t^-d_{t-}^2(\gamma^{\tau_{i-1}}, \gamma^{\tau_i})=-\infty$ for some $i$ and 
\begin{equation*}
{\frak b}_t^{\vec\tau}(\gamma)=\sum_{i=1}^{k}\frac1{\tau_{i}-\tau_{i-1}}\partial_t^-d_{t-}^2\left(\gamma^{\tau_{i-1}}, \gamma^{\tau_i}\right)
\end{equation*}
otherwise.
 Let us illustrate this in the case $\{\tau^0,\tau^2\}\subset \{\tau^0,\tau^1,\tau^2\}$:
\begin{eqnarray*}
&&\sum_{i=1}^{2}\liminf_{s\nearrow t}\frac1{t-s}\frac1{\tau_{i}-\tau_{i-1}}\left[d_t^2\left(\gamma^{\tau_{i-1}}, \gamma^{\tau_i}\right)
-d_s^2\left(\gamma^{\tau_{i-1}}, \gamma^{\tau_i}\right)\right]\\
&&\le
\liminf_{s\nearrow t}\frac1{t-s}\sum_{i=1}^{2}\frac1{\tau_{i}-\tau_{i-1}}\left[d_t^2\left(\gamma^{\tau_{i-1}}, \gamma^{\tau_i}\right)
-d_s^2\left(\gamma^{\tau_{i-1}}, \gamma^{\tau_i}\right)\right]\\
&&\le\liminf_{s\nearrow t}\frac1{t-s}\frac1{\tau_{2}-\tau_{0}}\left[d_t^2\left(\gamma^{\tau_{0}}, \gamma^{\tau_2}\right)
-d_s^2\left(\gamma^{\tau_{0}}, \gamma^{\tau_2}\right)\right].
\end{eqnarray*}

We define the {\it  strain} of the geodesic $\gamma$ by 
\begin{equation}
{\frak b}_t(\gamma)=\inf_{\vec\tau}
\sum_{i=1}^{k}\frac1{\tau_{i}-\tau_{i-1}}\partial_t^-d_{t-}^2\left(\gamma^{\tau_{i-1}}, \gamma^{\tau_i}\right).
\end{equation}
In particular,  ${\frak b}_t(\gamma)\le\frac1{\rho-\sigma}
\partial_t^-d_{t-}^2(\gamma^\sigma,\gamma^\rho)$.

If $\gamma: [0,1]\to X$ and $\tilde\gamma: [\sigma,\rho]\to X$ are linear time changes of each other -- i.e.\ $\gamma(\tau)=\tilde\gamma(\sigma+\tau(\rho-\sigma))$ for all $\tau$ -- then
$${\frak a}_t(\gamma)={\frak a}_t(\tilde\gamma)\cdot(\rho-\sigma)\quad\mbox{and}\quad
{\frak b}_t(\gamma)={\frak b}_t(\tilde\gamma)\cdot(\rho-\sigma).$$

\medskip

From time to time, we assume that  there exist
 nonnegative functions $\kappa\in L^1_{loc}(I)$ or $\lambda\in L^1_{loc}(I)$ such that 
 for all $s<t$ and all $x,y$
\begin{equation}\label{upp-lip}
 \log d_t(x,y)-\log d_s(x,y)\le \int_s^t \kappa_rdr
\end{equation}
or
\begin{equation}\label{low-lip}
 \log d_t(x,y)-\log d_s(x,y)\ge -\int_s^t \lambda_rdr,
\end{equation}
resp.
The  functions $\kappa$ and $\lambda$ are called (upper or lower, resp.) \emph{log-Lipschitz controls} for $d$.
Obviously, \eqref{upp-lip} is equivalent to upper absolute continuity of $t\mapsto d_t(x,y)$ and
$$\partial_t^+ d_t(x,y)\le \kappa_t\cdot d_t(x,y)\qquad\quad\mbox{for a.e.} t\in I$$
whereas  \eqref{low-lip}
is equivalent to lower absolute continuity and
$$\partial_t^- d_t(x,y)\ge -\lambda_t\cdot d_t(x,y)\qquad\quad\mbox{for a.e.\ } t\in I.$$
These  log-Lipschitz controls on $d_t$ immediately imply
$ {\frak b}_t(\gamma) \le 2\kappa_t\cdot  {\frak a}_t(\gamma)$ and
$ {\frak b}_t(\gamma) \ge -2\lambda_t\cdot  {\frak a}_t(\gamma)$, resp.

\subsection{Dynamic Convexity}

Let a time-dependent geodesic space $(X,d_t)_{ t\in I}$ be given and a function  $V: I\times X\to (-\infty,\infty]$. We always assume that $I$ is open from the left.
For $t\in I$, put $\Dom(V_t)=\{x\in X: \ V_t(x)<\infty\}$.

\begin{definition}\label{dyn}
(i) We say that
the  function $V: I\times X\to (-\infty,\infty]$  is \emph{strongly dynamically convex} if for a.e.\ $t\in I$  and for every $d_t$-geodesic $(\gamma^\tau)_{\tau\in[0,1]}$ with $\gamma^0,\gamma^1\in\Dom(V_t)$
the function $\tau\mapsto V_t(\gamma^\tau)$ is usc on $[0,1]$, ac on $(0,1)$, and
  \begin{equation}\label{dyn-conv-0}
\partial^+_\tau V_t(\gamma^{1-})- \partial^-_\tau V_t(\gamma^{0+})\ge -\frac 12\partial_t^- d_{t-}^2(\gamma^0,\gamma^1).
\end{equation}
For well-definedness of the LHS here -- and in similar expressions henceforth -- we make the convention that $(+\infty)-(+\infty)=+\infty$ as well as $(-\infty)-(-\infty)=+\infty$.

(ii)
Given $\lambda\in L^1_{loc}(I)$, we say $V$  is \emph{$\lambda$-moderate  dynamically convex} 
for a.e.\ $t$ and  
 $\forall x^0,x^1\in \Dom(V_t)$: \  $\exists$
  $d_t$-geodesic $(x^\tau)_{\tau\in[0,1]}$ connecting $x^0,x^1$  s.t. for all $\tau\in[0,1/2]$
\begin{eqnarray}\nonumber
\frac1\tau\Big[V_t(x^0)-V_t(x^\tau)
-V_t(x^{1-\tau})+V_t(x^1)\big]
\ge -\frac12\partial_t^- d_{t-}^2(x^0,x^1)
- \tau\lambda_t d_t^2(x^0,x^1).
\end{eqnarray}
\end{definition}

\begin{remark*} In most cases of application, for every $t\in I$ the function $V_t$ will be lower semicontinuous on $X$. In this case, the requirement that $\tau\mapsto V_t(\gamma^\tau)$ is usc on $[0,1]$ and ac on $(0,1)$ is equivalent to the condition that it is ac on $[0,1]$.
\end{remark*}

\begin{proposition}\label{stat-dyn-k} Given  a number $K\in\R$ and a non-empty interval $I\subset \{t\in\R: 2Kt<1\}$. Then for any  function $V$ on a static metric space $(X,d)$, the following are equivalent
\begin{itemize}

\item[(i)] $V$ is strongly $K$-convex on  $(X,d)$
\item[(ii)] $V$ is strongly dynamically convex on $(X,d_t)_{t\in I}$ where
$d_t^2=d^2\cdot(1-2Kt)$.
\end{itemize}
\end{proposition}

\begin{proof} Recall that a function which is absolutely continuous along all geodesics is strongly
$K$-convex if and only  for every $d$-geodesic
\begin{eqnarray*}
\partial^+_\tau V_t(\gamma_t^{1-})- \partial^-_\tau V_t(\gamma_t^{0+})\ge K\cdot d^2(x^0,x^1).
\end{eqnarray*}
Moreover, by construction we have 
$-2K\cdot d^2(x^0,x^1)=\partial_t d_t^2(x^0,x^1)$. This proves the claim.
\end{proof}

One can also introduce the more general concept of \emph{dynamically $K$-convex} functions (or, more precisely, dynamically $(K,\infty)$-convex functions). However, as illustrated in the subsequent result, this does not really enlarge the scope.

\begin{proposition}\label{dyn-dyn-k} Given a number $K\ne0$ and an interval interval $I\subset \{t\in\R: 2Kt<1\}$.
Assume that $t\mapsto d_t(x,y)$ is left continuous (for all $x,y$).
Then for each function $V: I\times X\to (-\infty,\infty]$  the following are equivalent
\begin{itemize}

\item[(i)]
 $V$  is strongly
dynamically $K$-convex on $(X,d_t)_{t\in I}$ in the sense that for a.e.\ $t\in I$ and for every $d_t$-geodesic $(\gamma^\tau)_{\tau\in[0,1]}$ with $\gamma^0,\gamma^1\in\Dom(V_t)$
the function $\tau\mapsto V_t(\gamma^\tau)$ is usc on $[0,1]$, ac on $(0,1)$, and
  \begin{equation}\label{dyn-conv2}
\partial^+_\tau V_t(\gamma^{1-})- \partial^-_\tau V_t(\gamma^{0+})\ge -\frac 12\partial_t^- d_{t-}^2(\gamma^0,\gamma^1)+K\cdot d^2_t(\gamma^0,\gamma^1).
\end{equation}

\item[(ii)]
$V$  is 
strongly dynamically convex on $(X,\tilde d_t)_{t\in I}$ where
$$\tilde d_t^2=(1-2Kt)\cdot d^2_{s(t)}, \quad s(t)=-\frac1{2K}\ln(1-2Kt).$$

\end{itemize}
\end{proposition}

\begin{proof} By construction
\begin{eqnarray*}
\partial^-_t \tilde d_{t-}^2&=&-2K\cdot d_{s(t)}^2+
(1-2Kt)\cdot \partial^-_t d^2_{s(t)-}\\
&=&-2K\cdot d_{s(t)}^2+(1-2Kt)\cdot
(\partial^-_s d^2_{s-})\big|_{s=s(t)}\cdot \partial_t s(t)\\
&=&\left[-2K\cdot d_{s}^2+
(\partial^-_s d^2_{s-})\right]\Big|_{s=s(t)}.
\end{eqnarray*}
Since each $d_t$-geodesic is a $\tilde d_t$-geodesic (and vice versa) this proves the claim.
\end{proof}

\begin{proposition}\label{log-ric}
Assume that $\kappa$ is an upper log-Lipschitz control for the time-dependent metric space $(X,d_t)_{t\in I}$ and that $V$ is strongly dynamically convex on $I\times X$. Then for a.e.\ $t\in I$  the function $V_t$   is strongly $(-\kappa_t)$-convex on $X$.
\end{proposition}

\begin{proof}
Using first the strong dynamic convexity of $V$ and then the log-Lipschitz bound for $d_t$ we conclude  that for all  $d_t$-geodesics 
\begin{eqnarray*}
\partial_\tau V_t(\gamma^{1-})- \partial_\tau V_t(\gamma^{0+})
&\ge& -\partial^-_t d^2_{t-}(\gamma^0,\gamma^1)\\
&\ge& -\kappa_t\cdot d_t^2(\gamma^0,\gamma^1).
\end{eqnarray*}
This proves the strong $(-\kappa_t)$-convexity.
\end{proof}

\begin{remark*}
For static spaces, $\lambda$-moderate  dynamic convexity implies weak midpoint $K$-convexity for each $K<0$ (see Proposition \ref{slope-conv-weak}).
This in turn implies proper convexity provided $X$ is locally compact and $V$ is lsc on $X$  as well as usc along geodesics (Theorem \ref{mid-weak-str}).
\end{remark*}

Under quite moderate assumptions,
the concepts of strong dynamic convexity, weak dynamic convexity, and $\lambda$-moderate  dynamic convexity will coincide.

We say that $V$ is upper regular if for a.e.\ $t\in I$ the function  $x\mapsto V_t(x)$ is upper regular on $(X,d_t)$ in the sense of Definition \ref{upp-reg}.

\begin{theorem}\label{prop-dyn}
Assume that $\lambda$ is a lower log-Lipschitz control for the time-dependent metric  $(d_t)_{t\in I}$ and that $V$ is an upper regular function  on  $I\times X$. 
Then for a.e.\ $t\in I$ the following are equivalent:
\begin{itemize}

\item[(i)]
$\forall x^0,x^1\in \Dom(V_t)$:    $\exists$
   $d_t$-geodesic $(x^\tau)_{\tau\in[0,1]}$ connecting $x^0$ and $x^1$  s.t.
\begin{equation*}
\partial^+_\tau V_t(x^{1-})- \partial^-_\tau V_t(x^{0+})\ge -\frac 12\partial_t^- d_{t-}^2(x^0,x^1)
\end{equation*}

\item[(ii)]
$\forall $
   $d_t$-geodesic $(\gamma^\tau)_{\tau\in[0,1]}$ with $\gamma^0,\gamma^1\in \Dom(V_t)$
\begin{equation*}
\partial^+_\tau V_t(\gamma^{1-})- \partial^-_\tau V_t(\gamma^{0+})\ge -\frac 12 {\frak b}_t(\gamma_t)
\end{equation*}

\item[(iii)]
$\forall $
   $d_t$-geodesic $(\gamma^\tau)_{\tau\in[0,1]}$ with 
   $\gamma^0,\gamma^1\in \Dom(V_t)$
    and all $\tau\le \frac12$
\begin{equation}\label{sieben999} 
 V_t(\gamma^0) +V_t(\gamma^1)- V_t(\gamma^\tau)-V_t(\gamma^{1-\tau}) \ge -\frac12 \int_0^\tau \frac1{1-2\sigma} \partial_t^- d_{t-}^2(\gamma^\sigma,\gamma^{1-\sigma})d\sigma
\end{equation}

\item[(iv)]
$\forall x^0,x^1\in \Dom(V_t): \, \exists
  d_t$-geodesic $(x^\tau)_{\tau\in[0,1]}$ connecting $x^0$ and $x^1$  s.t. for all $\tau\in[0,1/2]$
\begin{equation}\label{dyn-conv}
V_t(x^0)+V_t(x^1)-
V_t(x^\tau)-V_t(x^{1-\tau})\ge -\frac \tau2\partial_t^- d_{t-}^2(x^0,x^1)- \lambda_t \tau^2 d^2_t(x^0,x^1).
\end{equation}
\end{itemize}
\end{theorem}

\begin{remark*}
The fact that $ \int_0^\tau \frac1{1-2\sigma} \partial_t^- d_{t-}^2(\gamma^\sigma,\gamma^{1-\sigma})d\sigma$ is well-defined -- as an extended integral with values in $(-\infty,+\infty]$ -- should be regarded as part of assertion (iii) of the previous Theorem.

This remark will apply in analogous ways to similar integrals in Proposition \ref{triple-prop} and in Theorem \ref{prop-N-dyn}.
\end{remark*}

\begin{proof}
(i)$\Leftrightarrow$(ii): Since $\partial_t^-d^2_{t-}(\gamma^0,\gamma^1)\ge
{\frak b}_t(\gamma)$,
obviously (ii) implies (i). For the converse, let a $d_t$-geodesic $\gamma$ as well as a partition $\vec\tau=\{\tau_0,\ldots,\tau_k\}$ of $[0,1]$ be given and apply (i) for each $i=1,\ldots, k$ to the `endpoints' 
$\gamma^{\tau_{i-1}}$ and  $\gamma^{\tau_i}$. It yields the existence of a connecting $d_t$-geodesic
 $\tilde \gamma: [\tau_{i-1},\tau_i]\to X$ such that 
\begin{equation*}\label{sieben4}
\partial_\tau V_t(\tilde\gamma^{{\tau_i}-})- \partial_\tau V_t(\tilde\gamma^{\tau_{i-1}+})\ge
- \frac 12\partial_t^- d_{t-}^2(\gamma^{\tau_{i-1}},\gamma^{\tau_{i}})
\cdot\frac1{\tau_i-\tau_{i-1}}.
\end{equation*}
Patching together all these pieces, it yields a $d_t$-geodesic
$\tilde\gamma:[0,1|\to X$.
Due to the upper regularity of $V$
\begin{equation*}\label{ssss}
\partial_\tau V_t(\tilde\gamma^{\tau_i-})\le \partial_\tau V_t(\tilde\gamma^{\tau_i+})
\end{equation*}
for all $i=1,\ldots, k-1$. Thus adding up all the previous terms yields
\begin{eqnarray*}\label{sss9} 
-\partial_\tau V_t(\tilde\gamma^{0+}) +\partial_\tau V_t(\tilde\gamma^{1-})
 &\ge&
 \sum_{i=1}^{k} 
\partial_\tau V_t(\tilde\gamma^{\tau_i-})- \partial_\tau V_t(\tilde\gamma^{\tau_{i-1}+})\\
&\ge&
-\frac 12  \sum_{i=1}^{k} \frac1{\tau_i-\tau_{i-1}} \cdot \partial_t^- d_{t-}^2(\gamma^{\tau_i},\gamma^{\tau_{i-1}}).
\end{eqnarray*}
 In the limit of finer and finer partitions, the RHS approaches $ -\frac 12 {\frak b}_t(\gamma)$.
 Still some work is requested since on the LHS the slope of $V$ along (some suitable interpolating geodesic) $\tilde\gamma$ appears - however, it should be the slope along the originally given $\gamma$.
 To overcome this, we make the summation only over $i=2,\ldots, k-1$ (instead of $i=1,\ldots,k$) in order to obtain
 \begin{eqnarray}\label{k-1} 
-\partial_\tau V_t(\tilde\gamma^{\tau_1+}) +\partial_\tau V_t(\tilde\gamma^{\tau_{k-1}-})
 &\ge&
-\frac 12  \sum_{i=2}^{k-1} \frac1{\tau_i-\tau_{i-1}} \cdot \partial_t^- d_{t-}^2(\gamma^{\tau_i},\gamma^{\tau_{i-1}}).
\end{eqnarray}
By assumption $V_t$ is upper regular along each $d_t$-geodesic, thus in particular along the geodesic given by 
 $\gamma$ on $[0,\tau_1]\cup [\tau_{k-1},1]$ and by $\tilde\gamma$ on $[\tau_1,\tau_{k-1}]$.
 Hence, the LHS of the above inequality \eqref{k-1} is bounded from above by
 \begin{equation*}
-\partial_\tau V_t(\gamma^{\tau_1-}) +\partial_\tau V_t(\gamma^{\tau_{k-1}+}).
\end{equation*}
 Again by upper regularity of $V_t$ this approaches $-\partial_\tau V_t(\gamma^{0+}) +\partial_\tau V_t(\gamma^{1-})$
 as $\tau_1\to0$ and $\tau_{k-1}\to1$.
  Moreover, due to the lower log-Lipschitz control for $d_t$, the RHS of \eqref{k-1} approaches or exceeds $ -\frac 12 {\frak b}_t(\gamma)$ if the partition $\vec\tau$ gets sufficiently fine. Thus the claim is proven.
   
    $(ii)\Rightarrow(iii)$: 
     For $\sigma\in [0,\tau]$ consider geodesics $(\gamma_\sigma^a)_{a\in  [0,1]}$ connecting $\gamma^\sigma$ and $\gamma^{1-\sigma}$ defined by
     $$\gamma_\sigma^a=\gamma^{\sigma+a(1-2\sigma)}.$$
     Then
     \begin{eqnarray*}
     V_t(\gamma^\tau)-V_t(\gamma^0)&=&\int_0^\tau\partial_\sigma V_t(\gamma^\sigma)d\sigma
     =\int_0^\tau \frac1{1-2\sigma} \partial_a V_t(\gamma^a_\sigma)\Big|_{a=0}d\sigma
     \end{eqnarray*}
     and similarly
     $V_t(\gamma^1)-V_t(\gamma^{1-\tau})=\int_0^\tau \frac1{1-2\sigma} \partial_a V_t(\gamma^a_\sigma)\Big|_{a=1}d\sigma$.
   Adding up these two identities and applying (ii) to the geodesics  $(\gamma_\sigma^a)_{a\in  [0,1]}$ yields
      \begin{eqnarray*}
       V_t(\gamma^1)-V_t(\gamma^{1-\tau})
     - V_t(\gamma^\tau)+V_t(\gamma^0)&=&
     \int_0^\tau \frac1{1-2\sigma} \left[
      \partial_a V_t(\gamma^a_\sigma)\big|_{a=1}-
     \partial_a V_t(\gamma^a_\sigma)\big|_{a=0}
     \right]d\sigma\\
     &\ge&-\frac12
     \int_0^\tau \frac1{1-2\sigma} \partial_t^- d_{t-}^2(\gamma^\sigma,\gamma^{1-\sigma})
     d\sigma.
       \end{eqnarray*}
      
$(iv)\Rightarrow(v)$: 
Using the log-Lipschitz bound for $d_t$, the  claim follows from the fact that for each $d_t$-geodesic $\gamma$
$$\partial_t^- d_{t-}^2(\gamma^0,\gamma^{1})\ge
\frac1{1-2\sigma} \partial_t^- d_{t-}^2(\gamma^\sigma,\gamma^{1-\sigma})
+\frac1\sigma  \partial_t^- d_{t-}^2(\gamma^0,\gamma^{\sigma})+
\frac1\sigma  \partial_t^- d_{t-}^2(\gamma^{1-\sigma},\gamma^{1}).$$
Indeed, this implies
\begin{eqnarray*}
\int_0^\tau \frac1{1-2\sigma} \partial_t^- d_{t-}^2(\gamma^\sigma,\gamma^{1-\sigma})
     d\sigma
     &\le&
     \int_0^\tau\Big[
 \partial_t^- d_{t-}^2(\gamma^0,\gamma^{1})
-\frac1\sigma  \partial_t^- d_{t-}^2(\gamma^0,\gamma^{\sigma})-
\frac1\sigma  \partial_t^- d_{t-}^2(\gamma^{1-\sigma},\gamma^{1})\Big]d\sigma\\
&\le& \tau\cdot\partial^-_td^2_{t-}(\gamma^0,\gamma^1)+2\lambda_t \tau^2 \cdot d^2_t(\gamma^0,\gamma^1).
\end{eqnarray*}

 $(iv)\Rightarrow(i)$:  Dividing by $\tau$ and letting $\tau\to0$.
\end{proof}

\begin{remark*}
a) The implications $(ii)\Rightarrow(iii)\Rightarrow (iv)\Rightarrow (i)$ hold for each geodesic without any semiconvexity assumption on $\tau\mapsto V_t(\gamma^\tau_t)$.

b) For the equivalence $(i)\Leftrightarrow(ii)\Leftrightarrow (iii)$, the
 log-Lipschitz bound for the distance can be replaced by the significantly weaker condition that 
 \begin{equation}\label{d-regu}\liminf_{\sigma\searrow 0}\partial_t^-d_{t-}(\gamma^0,\gamma^\sigma)\ge0\end{equation}
 for a.e.\ $t\in I$ and every $d_t$-geodesic $\gamma$.
 
 c) The RHS in \eqref{sieben999}  can be equivalently replaced by
 \begin{equation*}
  -\frac12 \int_0^\tau {\frak b}_t(\gamma|_{[\sigma,1-\sigma]})d\sigma.
 \end{equation*}
\end{remark*}

The concept of (strong) dynamic convexity also allows for a re-formulation which is closer to the standard triple condition for convexity.

\begin{proposition}\label{triple-prop}
Assume that  $(d_t)_{t\in I}$ satisfies the regularity condition \eqref{d-regu}
and that for a.e.\ $t\in I$ the function $V_t$ is absolutely continuous along all  $d_t$-geodesics $(\gamma^\tau)_{\tau\in[0,1]}$. Then 
 the following are equivalent:
\begin{itemize}

\item[(i)] $V$ is strongly dynamically convex
\item[(ii)] for a.e.\ $t\in I$, 
for every 
   $d_t$-geodesic $(\gamma^\tau)_{\tau\in[0,1]}$ with 
   $\gamma^0,\gamma^1\in \Dom(V_t)$ 
    and  for all $\tau\in [0,1]$
\begin{equation}\label{triple-dyn} 
V_t(\gamma^\tau)- (1-\tau) V_t(\gamma^0) -\tau V_t(\gamma^1) \le \frac{\tau(1-\tau)}2 \int_0^1 \frac1{1-\sigma} \partial_t^- d_{t-}^2(\gamma^{\sigma\tau},\gamma^{1-\sigma+\sigma\tau})d\sigma
\end{equation}
\item[(iii)] 
for a.e.\ $t\in I$ and
for every 
   $d_t$-geodesic $(\gamma^\tau)_{\tau\in[0,1]}$ with 
   $\gamma^0,\gamma^1\in \Dom(V_t)$ 
\begin{equation}\label{single-slope} 
\partial^-_\tau
V_t(\gamma^\tau)\big|_{\tau=0+}\le V_t(\gamma^1) -V_t(\gamma^0) +\frac12 \int_0^1 \frac1{\sigma} \partial_t^- d_{t-}^2(\gamma^{0},\gamma^{\sigma})d\sigma.
\end{equation}
\end{itemize}
\end{proposition}

\begin{remark*}
The RHS in \eqref{triple-dyn} can equivalently be replaced by 
$$\frac{\tau(1-\tau)}2\int_0^1 {\frak b}_t\big(\gamma\big|_{[\sigma\tau, 1-\sigma+\sigma\tau]}\big)d\sigma$$ and the RHS in \eqref{single-slope} by
$$ V_t(\gamma^1) -V_t(\gamma^0) +\frac12 \int_0^1 
{\frak b}_t\big(\gamma\big|_{[0,\sigma]}\big)
d\sigma.$$
\end{remark*}

\begin{proof}
``(i)$\Rightarrow$(ii)'':
 Given a $d_t$-geodesic $(\gamma^\tau)_{\tau\in[0,1]}$ and $\sigma\in[0,1]$, define a geodesic $(\gamma^a_\sigma)_{a\in[0,1]}$ connecting $\gamma^{\sigma\tau}$ and $\gamma^{1-\sigma(1-\tau)}$ by
$$\gamma_\sigma^a=\gamma^{\sigma\tau+(1-\sigma)a}.$$
Then
\begin{eqnarray*}
\frac1\tau\big[V_t(\gamma^\tau)-V_t(\gamma^0)\big]=
\frac1\tau\int_0^1\partial_\sigma V_t(\gamma^{\sigma\tau})d\sigma
=\int_0^1 \frac1{1-\sigma}\partial_a V_t(\gamma_\sigma^a)\big|_{a=0}d\sigma
\end{eqnarray*}
and
\begin{eqnarray*}
\frac1{1-\tau}\big[V_t(\gamma^1)-V_t(\gamma^\tau)\big]=
-\frac1{1-\tau}\int_0^1\partial_\sigma V_t(\gamma^{1-\sigma(1-\tau)})d\sigma
=\int_0^1 \frac1{1-\sigma}\partial_a V_t(\gamma_\sigma^a)\big|_{a=1}d\sigma.
\end{eqnarray*}
Adding up these two identities and applying the defining property of dynamic convexity to the geodesic $(\gamma^a_\sigma)_{a\in[0,1]}$ we obtain
\begin{eqnarray*}
V_t(\gamma^\tau)- (1-\tau) V_t(\gamma^0) -\tau V_t(\gamma^1) &=& \tau(1-\tau)
\int_0^1
\frac1{1-\sigma}\Big[\partial_a V_t(\gamma_\sigma^a)\big|_{a=0}d\sigma-
\partial_a V_t(\gamma_\sigma^a)\big|_{a=1}\Big]d\sigma\\
&\le& \frac{\tau(1-\tau)}2
\int_0^1
\frac1{1-\sigma}
 \partial_t^- d_{t-}^2(\gamma^{\sigma\tau},\gamma^{1-\sigma(1-\tau)})
d\sigma.
\end{eqnarray*}
This proves the claim. 

``(ii)$\Rightarrow$(iii)'': Reshuffling the terms in \eqref{triple-dyn} and passing to the limit $\tau\searrow0$ yields
$$\partial^-_\tau
V_t(\gamma^\tau)\big|_{\tau=0+}-\Big[ V_t(\gamma^1) -V_t(\gamma^0)\Big] \le
\frac12\liminf_{\tau\searrow0} \int_0^1 \frac{1-\tau}{1-\sigma} \partial_t^- d_{t-}^2(\gamma^{\sigma\tau},\gamma^{1-\sigma+\sigma\tau})d\sigma.
$$
It remains to prove that the RHS is dominated by 
$$
\frac12 \int_0^1 \frac{1}{1-\sigma} \partial_t^- d_{t-}^2(\gamma^{0},\gamma^{1-\sigma})d\sigma.
$$

To do so, first note that 
$$\frac{1}{1-\sigma} \partial_t^- d_{t-}^2(\gamma^{\sigma\tau},\gamma^{1-\sigma+\sigma\tau})
\le \frac{1}{1-\sigma+\sigma\tau}\partial_t^- d_{t-}^2(\gamma^{0},\gamma^{1-\sigma+\sigma\tau})-
\frac{1}{\sigma\tau} \partial_t^- d_{t-}^2(\gamma^{0},\gamma^{\sigma\tau})
$$
and that 
$$\liminf_{\tau\searrow0} \int_0^1\frac{1}{\sigma\tau} \partial_t^- d_{t-}^2(\gamma^{0},\gamma^{\sigma\tau})d\sigma\ge0$$ due to our assumption \eqref{d-regu}.
Moreover, note that
\begin{eqnarray*}
\liminf_{\tau\searrow0} \int_0^1\frac{1-\tau}{1-\sigma+\sigma\tau} \partial_t^- 
d_{t-}^2(\gamma^{0},\gamma^{1-\sigma+\sigma\tau})
d\sigma&=&\liminf_{\tau\searrow0}
\int_\tau^1\frac1\rho \partial_t^- 
d_{t-}^2(\gamma^{0},\gamma^\rho)d\rho\\
&\le&
\int_0^1\frac1\rho \partial_t^- 
d_{t-}^2(\gamma^{0},\gamma^\rho)d\rho
\end{eqnarray*}
where for the last inequality we again used \eqref{d-regu}.

``(iii)$\Rightarrow$(i)'': Applying (iii) to the time-reversed geodesic $\tilde\gamma^\tau=\gamma^{1-\tau}$ yields
$$-\partial^-_\tau
V_t(\gamma^\tau)\big|_{\tau=1-}-\Big[ V_t(\gamma^0) -V_t(\gamma^1)\Big] \le
\frac12 \int_0^1 \frac{1}{1-\sigma} \partial_t^- d_{t-}^2(\gamma^{\sigma},\gamma^{1})d\sigma.
$$ 
Together with \eqref{single-slope} we end up with
\begin{eqnarray*}
\partial^-_\tau
V_t(\gamma^\tau)\big|_{\tau=0+}-\partial^-_\tau
V_t(\gamma^\tau)\big|_{\tau=1-}&\le&
\frac12 \int_0^1 \Big[
\frac1{\sigma}
 \partial_t^- d_{t-}^2(\gamma^{0},\gamma^{\sigma})+
\frac{1}{1-\sigma} \partial_t^- d_{t-}^2(\gamma^{\sigma},\gamma^{1})\Big]d\sigma\\
&\le&
\frac12 \partial_t^- d_{t-}^2(\gamma^{0},\gamma^{1})
\end{eqnarray*}
which is the claim.
\end{proof}

\subsection{Dynamic $N$-Convexity}

Besides dynamic convexity, also an enforced version will be of interest involving a parameter $N\in [1,\infty]$. In applications, this parameter will play the role of an upper bound for the dimension. In the case $N=\infty$ this new concept of dynamic $N$-convexity will coincide with the previously defined dynamic convexity. For finite $N$, it will be more restrictive.

\begin{definition}\label{dyn-Nconv}
We say that
the  function $V: I\times X\to (-\infty,\infty]$  is \emph{strongly dynamically $N$-convex} if for a.e.\ $t\in I$ and for every $d_t$-geodesic $(\gamma^a)_{a\in[0,1]}$ with $\gamma^0,\gamma^1\in\Dom(V_t)$ the function $a\mapsto V_t(\gamma^a)$ is ac on $(0,1)$, usc on $[0,1]$,  and 
  \begin{equation}\label{N-dyn-conv2}
\partial^+_a V_t(\gamma_t^{1-})- \partial^-_a V_t(\gamma_t^{0+})\ge -\frac 12\partial_t^- d_{t-}^2(x^0,x^1)+\frac1N \left| V_t(x^0)-V_t(x^1)\right|^2.
\end{equation}
\end{definition}

\begin{theorem}\label{prop-N-dyn}
Assume that $\lambda$ is a lower log-Lipschitz control for the time-dependent metric space $(X,d_t)_{t\in I}$ and that $V$ is an upper regular function  on  $I\times X$. 
Then for a.e.\ $t\in I$ the following are equivalent:
\begin{itemize}

\item[(i)]
$\forall x^0,x^1\in \Dom(V_t)$:    $\exists$
   $d_t$-geodesic $(x^\tau)_{\tau\in[0,1]}$ connecting $x^0$ and $x^1$  s.t.
\begin{equation*}
\partial^+_\tau V_t(x^{1-})- \partial^-_\tau V_t(x^{0+})\ge -\frac 12\partial_t^- d_{t-}^2(x^0,x^1)+\frac1N \left| V_t(x^0)-V_t(x^1)\right|^2
\end{equation*}

\item[(ii)]
$\forall $
   $d_t$-geodesic $(\gamma^\tau)_{\tau\in[0,1]}$ with $\gamma^0,\gamma^1\in \Dom(V_t)$
\begin{equation*}
\partial^+_\tau V_t(\gamma^{1-})- \partial^-_\tau V_t(\gamma^{0+})\ge -\frac 12 {\frak b}_t(\gamma)+
\frac1N\int_0^1  \big| \partial_\sigma V_t(\gamma^\sigma)\big|^2d\sigma
\end{equation*}

\item[(iii)] 
 $\forall $
   $d_t$-geodesic $(\gamma^\tau)_{\tau\in[0,1]}$ with 
   $\gamma^0,\gamma^1\in \Dom(V_t)$ 
\begin{equation}
\partial^-_\tau
V_t(\gamma^\tau)\big|_{\tau=0+}\le V_t(\gamma^1) -V_t(\gamma^0) + \int_0^1 \Big[
\frac12{\frak b}_t\big(\gamma\big|_{[0,\sigma]}\big)-\frac{1-\sigma} N\cdot \big| \partial_\sigma V_t(\gamma^\sigma)\big|^2\Big]
d\sigma
\end{equation}

\item[(iv)] 
 $\forall $
   $d_t$-geodesic $(\gamma^\tau)_{\tau\in[0,1]}$ with 
   $\gamma^0,\gamma^1\in \Dom(V_t)$ 
    and  for all $\tau\in [0,1]$
\begin{eqnarray*}
V_t(\gamma^\tau)- (1-\tau) V_t(\gamma^0) -\tau V_t(\gamma^1) &\le& \frac{\tau(1-\tau)}2 \int_0^1 \frac1{1-\sigma} \partial_t^- d_{t-}^2(\gamma^{\sigma\tau},\gamma^{1-\sigma+\sigma\tau})d\sigma\\
&-&\frac1N\int_0^1\chi^{\tau,\sigma}\cdot 
\big| \partial_\sigma V_t(\gamma^\sigma)\big|^2
d\sigma
\end{eqnarray*}

\item[(v)]
$\forall $
   $d_t$-geodesic $(\gamma^\tau)_{\tau\in[0,1]}$ with $\gamma^0,\gamma^1\in \Dom(V_t)$
   and all $\tau\le \frac12$
\begin{eqnarray*} 
 V_t(\gamma^0) +V_t(\gamma^1)- V_t(\gamma^\tau)-V_t(\gamma^{1-\tau})& \ge& -
 \frac12 \int_0^\tau \frac1{1-2\sigma} \partial_t^- d_{t-}^2(\gamma^\sigma,\gamma^{1-\sigma})d\sigma\\
&&+\frac\tau N\int_0^1 \Lambda^{\tau,\sigma} \cdot\left| \partial_\sigma V_t(\gamma^\sigma)\right|^2d\sigma
\end{eqnarray*}

\item[(vi)]
$\forall x^0,x^1\in \Dom(V_t)$: \  $\exists$
  $d_t$-geodesic $(x^\tau)_{\tau\in[0,1]}$ connecting $x^0$ and $x^1$  s.t. for all $\tau\in[0,1/2]$
and   all $N'\in [N,\infty]$ with $N'\ge 2\tau[ |V_t(x^0)-V_t(x^1)|+ \lambda_t d_t^2(x^0,x^1)/2]$
\begin{eqnarray}\nonumber
\lefteqn{
\frac1\tau\Phi_{N'}\left(V_t(x^0))-V_t(x^\tau)\right)
+\frac1\tau\Phi_{N'}\left(V_t(x^1)-V_t(x^{1-\tau})\right)}\\
&\ge& -\frac12\partial_t^- d_{t-}^2(x^0,x^1)
- \lambda_t \tau d^2_t(x^0,x^1)+\frac1{N'} \left| V_t(x^0)-V_t(x^1)\right|^2
\end{eqnarray}
where $\Phi_{N'}(u)= u + \frac1{N'} u^2$.
\end{itemize}
\end{theorem}

\begin{proof} The implication  $(vi)\Rightarrow (i)$ follows simply by passing to the limit $\tau\to0$.
For the proofs of the implications $(i)\Rightarrow (ii)\Leftrightarrow (iii)\Leftrightarrow (iv)\Leftrightarrow (v)$ we follow
 the argumentation in the proof of the previous Theorem and Proposition.
 
 For the implication
 $(v)\Rightarrow (vi)$ note that (v) implies that
 for all $\tau\le \frac12$ and all $N'\ge N$
\begin{eqnarray*}
\lefteqn{
 \frac1\tau\left[V_t(\gamma^0) +V_t(\gamma^1)- V_t(\gamma^\tau)-V_t(\gamma^{1-\tau})\right] + 
  \frac12 \int_0^\tau \frac1{1-2\sigma} \partial_t^- d_{t-}^2(\gamma^\sigma,\gamma^{1-\sigma})d\sigma}\\
 &\ge&
  \frac 1{N} \int_0^1 \Lambda^{\tau,\sigma}
   |\partial_\sigma V_t(\gamma^\sigma)|^2
  d\sigma\ge
   \frac 1N \int_\tau^{1-\tau} 
   |\partial_\sigma V_t(\gamma^\sigma)|^2
  d\sigma\\
   &\ge&
  \frac 1{N(1-2\tau)} \left|\int_\tau^{1-\tau} 
  \partial_\sigma V_t(\gamma^\sigma)
  d\sigma\right|^2
   =
  \frac 1{N(1-2\tau)} \left|V_t(\gamma^\tau)-
 V_t(\gamma^{1-\tau})
\right|^2\\  
   &\ge&
  \frac 1{N'} \left[
  \left|V_t(\gamma^0)-V_t(\gamma^{1})\right|^2-\frac1\tau  \left|V_t(\gamma^0)- V_t(\gamma^{\tau})\right|^2
  -\frac1\tau  \left|V_t(\gamma^{1-\tau})- V_t(\gamma^{1})\right|^2
    \right].
\end{eqnarray*}
\end{proof}

For later application of the stability argument, we need the following monotonicity property.

\begin{lemma}\label{mono-conv}
Assume that
$V_t$ is $(-\kappa_t)$-convex,
Then the function $u\mapsto \Phi(u)$ is increasing in $u$ for all $u\ge u^*:=-\frac{N'}2$ and 
$$V_t(x^0)-V_t(x^\tau)\ge u^*, \qquad V_t(x^1)-V_t(x^{1-\tau})\ge u^*$$
for all $\tau\le \frac12$.
\end{lemma}

\begin{proof}
 $(-\kappa_t)$-convexity of $V_t$ implies
\begin{eqnarray*}
V_t(x^0)-V_t(x^\tau)\ge \tau \left[V_t(x^0)-V_t(x^1)\right]-\kappa_t\frac{\tau(1-\tau)}2d_t^2(x^0,x^1)
\ge-\tau\cdot \frac{N'}{2\tau}=u^*.
\end{eqnarray*}
\end{proof}

\begin{definition}\label{def-controlled} Given a time-dependent metric space $(X,d_t)_{t\in I}$
and $\lambda\in L^1_{loc}(I)$, we say that a 
 function $V: I\times X\to (-\infty,\infty]$  is \emph{$\lambda$-moderate  dynamically $N$-convex} if 
 $\forall x^0,x^1\in \Dom(V_t)$: \  $\exists$
  $d_t$-geodesic $(x^\tau)_{\tau\in[0,1]}$ connecting $x^0$ and $x^1$  s.t. for all $\tau\in[0,1/2]$
and   all $N'\in [N,\infty]$ with $N'\ge 2\tau[ |V_t(x^0)-V_t(x^1)|+ \lambda_t d_t^2(x^0,x^1)/2]$
\begin{eqnarray}\label{N-dyn-conv}\nonumber
\lefteqn{
\frac1\tau\Phi_{N'}\left(V_t(x^0))-V_t(x^\tau)\right)
+\frac1\tau\Phi_{N'}\left(V_t(x^1)-V_t(x^{1-\tau})\right)}\\
&\ge& -\frac12\partial_t^- d_{t-}^2(x^0,x^1)
- \lambda_t \tau d^2_t(x^0,x^1)+\frac1{N'} \left| V_t(x^0)-V_t(x^1)\right|^2
\end{eqnarray}
where $\Phi_{N'}(u)= u + \frac1{N'} u^2$.
 
\end{definition}

\subsection{EVI Gradient Flows}

In the static case, the concept of \emph{evolution-variation inequality} as introduced and utilized e.g. in \cite{AGS-Calc, AGS-Mms, AGS-BE} turned out to be extremely powerful. We present here an analogous concept for the dynamic case. However, this seems to be less powerful for applications. We will not use it in subsequent chapters.

Given a  function $V: I\times X\to (-\infty,\infty]$, a curve $(x_t)_{t\in(r,T)}$ with $(r,T)\subset I$ will be called \emph{upward
 $(EVI)_{dyn}$-gradient flow} for $V$ if $s\mapsto d_t(x_s,y)$ is absolutely continuous, uniformly in $t$ and $y$, and if for a.e.\ $t\in(r,T)$, all $z\in X$ with $V_t(z)<\infty$, and each $d_t$-geodesic 
  $(\gamma^a)_{a\in[0,1]}$  connecting $x_t$ to $z$
  \begin{align}\label{evi-dyn}
     {\frac12\partial_s^- d^2_t(x_s,z)}\Big|_{s=t-}+
      \frac12 
     {\frak b}_t^0 (\gamma)
       ~\ge V_t(x_t)-V_t(z)\;.
  \end{align}
  where 
  ${\frak b}_t^0 (\gamma)=\int_0^1 {\frak b}_t(\gamma|_{[0,\sigma]})d\sigma.$
  
\begin{proposition} Assume that for each $x'\in X$ and $(r,T)\subset I$ there exists an  upward $(EVI)_{dyn}$-gradient flow $(x_t)_{t\in(r,T)}$ for $V$ terminating in $(T,x')$.
Then $V$ is strongly dynamically convex.
\end{proposition}

\begin{proof}
Fix a `non-exceptional' $t\in (r,T)$,  an arbitrary $d_t$-geodesic $(\gamma^b)_{b\in[0,1]}$ and $a\in[0,1]$. 
Let 
$(x_s)_{r< s\le T}$ be the gradient flow for $V$ terminating in $x_T=\gamma^a$.
The (EVI)$_{dyn}$-property applied to $z=\gamma^0$ and $z=\gamma^1$, resp, yields
 \begin{align*}
  V_t(\gamma^a)\le V_t(\gamma^0)+   {\frac12\partial_s^- d^2_t(x_s,\gamma^0)}\Big|_{s=t-}  +
  \frac12 {\frak b}_t^0 \left((\gamma^{a(1-b)})_{b\in[0,1]}\right)
  \end{align*}
and
\begin{align*}
  V_t(\gamma^a)\le V_t(\gamma^1)+   {\frac12\partial_s^- d^2_t(x_s,\gamma^1)}\Big|_{s=t-}  +
   \frac12 {\frak b}_t^0 \left((\gamma^{a+b(1-a)})_{b\in[0,1]}\right) 
  \end{align*}
Multiplying these inequalities by $(1-a)$ or $a$, resp., and adding them up leads to
\begin{eqnarray*}
  V_t(\gamma^a)&\le& (1-a)V_t(\gamma^0) +aV_t(\gamma^1)\\
&+&   \frac12\Big((1-a)\partial_s^- d^2_t(x_s,\gamma^0)+a\partial_s^- d^2_t(x_s,\gamma^1)\Big)\Big|_{s=t-} \\
& +& \frac12 \Big((1-a){\frak b}_t^0 
 \left((\gamma^{a(1-b)})_{b\in[0,1]}\right)
 +a {\frak b}_t^0
  \left((\gamma^{a+b(1-a)})_{b\in[0,1]}\right) 
\Big)\\
&\le& (1-a)V_t(\gamma^0) +aV_t(\gamma^1)+
\frac{a(1-a)}2\int_0^1 {\frak b}_t\big(\gamma\big|_{[\sigma a, 1-\sigma+\sigma a]}\big)d\sigma
  \end{eqnarray*}
which is the claim. The last inequality here is due to the fact that
$$(1-a)d_t^2(x_s,\gamma^0)+a d_t^2(x_s,\gamma^1)\ge a(1-a) d_t^2(\gamma_0,\gamma_1)\qquad\mbox{for all }s<t$$ whereas
$(1-a)d_t^2(x_t,\gamma^0)+a d_t^2(x_t,\gamma^1)= a(1-a) d_t^2(\gamma_0,\gamma_1)$. 
\end{proof}

\begin{remark}
(i) A curve  $(x_t)$ satisfies the dynamic evolution-variation inequality $(EVI_{(0,N)})_{dyn}$ for $V$ if
 \begin{align}\label{evi-N}
     {\frac12\partial_s^- d^2_t(x_s,z)}\Big|_{s=t-}+
      \frac12 {\frak b}_t^0 \left(\gamma\right) 
      ~\ge V_t(x_t)-V_t(z)
+\frac1N\int_0^1 (1-b)\Big( \partial_b V_t(\gamma^b)\Big)^2db
\;
  \end{align} 
for all $z\in M$ and  a.e.\ $t\in(r,T)$. Here $(\gamma^b)_{b\in[0,1]}$ is a geodesic from $x_t$ to $z$.

 If for each terminating point $(T,x')$ there exists a curve $(x_t)_{r<t\le T}$ satisfying this EVI-inequality then $V$ is dynamically $(0,N)$-convex.

(ii)
We say that the function $V:(t,x)\mapsto V_t(x)$ on $[0,T)\times X$  is \emph{backward dynamically convex} iff the function $\tilde V_s(x):=V_{T-s}(x)$ on $(0,T]\times X$ is dynamically convex w.r.t.\ the family of metrics $(\tilde d_s)_{s\in[0,T)}$ where $\tilde d_s=d_{T-s}$. That is, $V$  is backward dynamically convex iff for all $t$, all $d_t$-geodesics $(\gamma^a)_{a\in[0,1]}$ and all $a\in[0,1]$
\begin{equation}\label{sieben}
\partial_aV_t(\gamma^{1-})-
\partial_aV_t(\gamma^{0+})
\ge \frac12\partial_t^+ d_{t+}^2(\gamma^0,\gamma^1).
\end{equation}
A smooth function on a smooth, time-dependent Riemannian manifold is backward dynamically convex iff
$\Hess_t V_t\ge  \frac12\partial_t g_t.$

(iii)
Looking backward in time allows to re-interpret upward gradient flows as gradient flows in the usual sense (i.e.\ as downward gradient flows).
More precisely, any upward gradient flow terminating in $(T,x')$ can be regarded as a downward gradient flow emanating in $(T,x')$ and running backward in time.

\end{remark}

\subsection{The Smooth Riemannian Setting}

Convexity allows for an alternative, more explicit characterization in the Riemannian setting. It is well-known that a smooth function $V: M\to\R$  is {convex} if and only if
$\Hess \, V\ge 0$
and, more generally, it is $K$-convex if and only if
$\Hess\, V\ge K\cdot g$. 

The time-dependent extension of these results reads as follows.

\begin{proposition}
Given  a  manifold $M$ with a smooth 1-parameter family of Riemannian metric tensors  $(g_t)_{t\in I}$, then a smooth function $V: I\times M\to\R$  is {dynamically convex} if and only if
\begin{equation}
\Hess_t V_t\ge -\frac12\partial_t g_t.
\end{equation}
\end{proposition}


\begin{proof}
The equivalence follows from two basic facts. Firstly,  on each (static) Riemannian manifold $(M,g_t)$
\begin{eqnarray*}
-V_t(\gamma^a)+ (1-a)V_t(\gamma^0)+a V_t(\gamma^1)&=&
\int_0^1 \chi^{a,b} \cdot\Hess_t V_t(\dot\gamma^b,\dot\gamma^b)\,db\\
&=&a(1-a)\int_0^1\int_0^1 \Hess_t V_t(\dot\gamma^{ab+(1-b)\sigma},\dot\gamma^{ab+(1-b)\sigma})\,d\sigma\,db
\end{eqnarray*}
for any $g_t$-geodesic $\gamma$.
Secondly. 
\begin{eqnarray*}
\int_0^1\int_0^1 \partial_t g_t(\dot\gamma^{ab+(1-b)\sigma},\dot\gamma^{ab+(1-b)\sigma})\,d\sigma\,db
=
\int_0^1 \frac1{1-b} \partial_t d^2_t(\gamma^{ab},\gamma^{ab+(1-b)})\,db.
\end{eqnarray*}
Combing both,  yields \eqref{triple-dyn} and thus proves the forward implication. The backward implication follows by applying these identities for fixed $x,t$ and tangent vector $\xi$ to  $d_t$-geodesics $(\gamma^a_{t})_{a\in[-\epsilon/2,\epsilon/2]}$ with $\gamma^0=x$ and $\dot\gamma^0=\xi$  for $\epsilon\searrow 0$. Then \eqref{triple-dyn}  implies
\begin{eqnarray*}
\partial^2_a V_t(\gamma_t^a)\ge-\frac12\partial_t g_t(\xi,\xi).
\end{eqnarray*}
\end{proof}

\begin{lemma}
 Let $M$ be  a compact manifold $M$ equipped with a smooth 1-parameter family of Riemannian metric tensors $g_t,t\in I$, with $I$ being a left open interval and let $V: I\times M\to\R$ be a smooth function. Then for each $(T,x')\in I\times M$ there exists $r\in I, r<t$ and a smooth curve
  $x:[r, T]\to M$ terminating in $x'$ (i.e.\ $x_{T}=z$) which solves  the gradient flow equation
  \begin{align}\label{eq:smoothgradflow}
   \dot x_t=\nabla_t V_t(x_t) \qquad \forall t\in [r,T].
  \end{align}
\end{lemma}
\begin{proof}
Generalized Picard iteration or limit of steepest descend scheme.
\end{proof} 

\begin{proposition}\label{lem:gradflowevi}
  Assume that $V$ is  dynamically convex  on a manifold $M$ with a smooth 1-parameter family of Riemannian metric tensors. A smooth curve
  $x:[r, T)\to M$ is a solution to the gradient flow equation
\eqref{eq:smoothgradflow}
 if and only if the following  \emph{dynamic Evolution Variation Inequality}
  $(EVI)_{dyn}$ holds
  \begin{align}\label{eq:smoothevikn}
     {\frac12\partial^-_s d^2_t(x_s,z)}\Big|_{s=t-} + \frac12 
     {\frak b}_t^0 (\gamma_t)
   ~\ge V_t(x_t)-V_t(z)\;
  \end{align}
for all $z\in M$ and  all $t\in(r,T)$ where $(\gamma_t^a)_{a\in[0,1]}$ denotes any $d_t$-geodesic connecting $x_t$ to $z$.
\end{proposition}

\begin{proof}
  To prove the only if part, fix $t>r$, $z\in M$ and a constant
  speed $d_t$-geodesic $\gamma:[0,1]\to M$ connecting $x_t$ to $z$. 
The dynamic convexity of $V$ implies
\begin{eqnarray*}
\partial_a V_t(\gamma^a)\Big|_{a=0+}&\le& V_t(\gamma^1)-V_t(\gamma^0)+\frac12
{\frak b}_t^0 (\gamma).
\end{eqnarray*}
On the other hand, the first variation formula  (in the static Riemannian manifold $(M,g_t)$) yields
  \begin{align*}
\partial_a V_t(\gamma^a)\Big|_{a=0}
~&=\langle{\nabla_t V_t(x_t),\dot\gamma^0}\rangle_t~=\langle{\dot x_t,\dot\gamma^0}\rangle_t
              =-\partial_s \frac12 d_t^2(x_s,z)\Big|_{s=t}\;.
  \end{align*}
Combing both assertions gives the claim.

For the if part, fix $t>t_0$ and a constant speed $d_t$-geodesic
  $\gamma:[0,1]\to M$ with $\gamma^0=x_t$ and arbitrary  $\xi=\dot\gamma^0$. Using the dynamic Evolution
  Variation inequality 
  (applied to the rescaled curve $\tilde\gamma^a=\gamma^{\epsilon a}$)
 with
  $z=\gamma^\epsilon$ for some $\epsilon>0$
and the fact that
$\partial_s \frac12 d_t^2(x_s,z)\Big|_{s=t}=\langle{\dot x_t,\dot\gamma^0}\rangle_t$
 we obtain
  \begin{eqnarray*}
    -V_t(\gamma^\epsilon)+V_t(\gamma^0)
    &\ge&
    -\langle{\dot x_t,\dot\tilde\gamma^0}\rangle_t
+
 \frac12 
     {\frak b}_t^0 (\tilde\gamma_t)\\
      &\ge&
    -\langle{\dot x_t,\dot\tilde\gamma^0}\rangle_t
- \lambda_t \int_0^1(1-\tau)
     {\frak g}_t (\tilde\gamma_t)d\tau\\
     &=&
        -\epsilon\langle{\dot x_t,\dot\gamma^0}\rangle_t
- \lambda_t\epsilon^2\cdot d_t^2(\gamma^0,\gamma^1)/2.
  \end{eqnarray*}
 Dividing by $\epsilon$ and letting $\epsilon\to0$ yields 
 \begin{align}\label{grad-ident}
- \langle \nabla_t  V_t(x_t),\xi\rangle_t~\ge~\
-\langle{\dot x_t,\xi}\rangle_t.
  \end{align}
The fact that \eqref{grad-ident} holds true for all $\xi$ in the $g_t$-tangent space of $x_t$ implies \eqref{eq:smoothgradflow}.
\end{proof}

\begin{proposition}\label{riem-evi} Let $(x_t)_{t\in[r,T)}$  and $(y_t)_{t\in[r,T)}$ be two  upward $(EVI)_{dyn}$-gradient flow curves for a given function $V$, both flows depending smoothly on $t$. Then
for all $r\le s \le t<T$
 \begin{align}
d_t(x_t,y_t)\ge d_s(x_s,y_s).
\end{align}
\end{proposition}

\begin{proof}
Applying \eqref{evi-dyn} to the curve $(x_t)_t$ and the point $z=y_t$ yields
 \begin{align*}
 {\frac12\partial_r d^2_t(x_r,y_t)}\Big|_{r=t} +
 \frac12 
     {\frak b}_t^0 (\gamma_t)
      ~\ge V_t(x_t)-V_t(y_t)\;.
  \end{align*}
Interchanging the roles of $x_t$ and $y_t$ yields
 \begin{align*}
 {\frac12\partial_r d^2_t(x_t,y_r)} \Big|_{r=t}  + 
  \frac12 
     {\frak b}_t^1 (\gamma_t)
     ~\ge V_t(y_t)-V_t(x_t)\;.
  \end{align*}
  Here 
  $(\gamma_t^a)_{a\in[0,1]}$  denotes any  $d_t$-geodesic connecting $x_t$ to $y_t$.
  Note that
$$ \frac12\partial_r d^2_r(x_t,y_t)\Big|_{r=t} \ge \frac12 {\frak b}_t^0 (\gamma_t)+ \frac12 {\frak b}_t^1 (\gamma_t).$$
Adding up these three estimates leads to
 \begin{align*}
 \frac12\partial_t d^2_t(x_t,y_t)\ge \Big( {\frac12\partial_r d^2_t(x_r,y_t)}+\frac12\partial_r d^2_t(x_t,y_r) + \frac12\partial_r{ d^2_r(x_t,y_t)}\Big) \Big|_{r=t}    ~\ge 0\;
  \end{align*}
which proves the claim. 
\end{proof}

\begin{corollary}
For each $(T,x')$ there exists at most one smooth upward $(EVI)_{dyn}$-gradient flow for $V$ terminating in $(T,x')$.
\end{corollary}

\section{Super-Ricci Flows for MM-Spaces}

\subsection{Time-dependent Wasserstein Metrics}
Let $X$ be a Polish space equipped with a 1-parameter family $(d_t)_{t\in I}$ of geodesic metrics on $X$. We assume that each of them generates the topology of $X$.

In terms of the  metric $d_t$ for given $t$, we define the \emph{$L^2$-Wasserstein  metric} 
$W_t$ on the space of probability measures on $X$:
$$W_t(\mu,\nu)=\inf\left\{\int_{X\times X}d^2_t(x,y)\,dq(x,y):\ q\in\Cpl(\mu,\nu)\right\}^{1/2}$$
where $\Cpl(\mu,\nu)$ as usual denotes the set of all probability measures on $X\times X$ with marginals $\mu$ and $\nu$.
In general, $W_t$  might take the value $+\infty$, i.e.\ it 
will only be a pseudo metric. However, this will cause no troubles since all discussions will be restricted to the subspace
$$\Pz_t=\Big\{\mu\in \Pz: W_t(\mu,\delta_z)<\infty\quad\mbox{for some/all }z\in X\Big\}.$$

\begin{remark*}
More generally, one might assume that $d_t$  is not really a metric but just a pseudo metric on $X$, possibly attaining values 0 and $\infty$ off the diagonal.
Then also $W_t$ will be a pseudo metric on $\Pz$ which canonically induces to a metric on  the quotient space $\Pz_t/W_t$.
\end{remark*}

\begin{lemma}
For each $s,t\in I$ and each $C_1,C_2$ the following are equivalent
\begin{itemize}
\item[(i)] for all points $x,y\in X$
\begin{equation*}
d^2_t(x,y)\le C_1+C_2d^2_s(x,y);
\end{equation*}
\item[(ii)]
for  all probability measures $\mu,\nu$ on $X$
\begin{equation*}
W^2_t(\mu,\nu)\le C_1+C_2W^2_s(\mu,\nu).
\end{equation*}
\end{itemize}
\end{lemma}

\begin{proof}
The backward implication follows from the isometric embedding $x\mapsto \delta_x$ of $X$ into the space of probability measures on $X$.
For the forward implication, note
that for all $s<t$ and each $\epsilon>0$
\begin{eqnarray*}
 W_t^2(\mu^0,\mu^1)&\le&
\int_{X\times X}  d_t^2(x^0,x^1)dq_{s}(x^0,x^1)\\
&\le& C_1+ C_2\cdot 
\int_{X\times X} d^2_{s}(x^0,x^1)\,dq_{s}(x^0,x^1)=C_1+C_2\cdot \left(W^2_{s}(\mu^0,\mu^1)+\epsilon\right)
\end{eqnarray*}
where $q_{s}$ denotes some `almost' $d_{s}$-optimal coupling of $\mu^0$ and $\mu^1$.
\end{proof}

\begin{corollary}\label{w-control} The upper log-Lipschitz bound \eqref{upp-lip} implies that for all $s<t$ and all $\mu,\nu\in\Pz$
\begin{equation}\label{w-upp-lip}
 \log W_t(\mu,\nu)-\log W_s(\mu,\nu)\le \int_s^t \kappa_rdr.
\end{equation}
whereas the lower log-Lipschitz bound \eqref{low-lip} implies 
\begin{equation}\label{w-low-lip}
 \log W_t(\mu,\nu)-\log W_s(\mu,\nu)\ge -\int_s^t \lambda_rdr.
\end{equation}
\end{corollary}

We denote the \emph{infinitesimal action} of a $W_t$-Lipschitz curve $\mu=(\mu^\tau)_{\tau\in[0,1]}$ in $\Pz_t$ by
\begin{equation*}{\frak G}_t^\tau(\mu)=\lim_{\sigma\to \tau}\left|\frac{W_t(\mu^\tau,\mu^\sigma)}{\tau-\sigma}\right|^2
\end{equation*}
and  the \emph{action} itself by
\begin{equation*}
{\frak A}_t(\mu)=
\sup_{\vec\tau}\sum_{i=1}^{k}\frac1{\tau_{i}-\tau_{i-1}}W_t^2\left(\mu^{\tau_{i-1}}, \mu^{\tau_i}\right)
=\int_0^1 {\frak G}^\tau_t(\mu)d\tau.
\end{equation*}
Of particular interest is the \emph{strain}
\begin{equation*}
{\frak B}_t(\mu)=
\inf_{\vec\tau}\sum_{i=1}^{k}\frac1{\tau_{i}-\tau_{i-1}}
\partial_t^-
W_{t-}^2\left(\mu^{\tau_{i-1}}, \mu^{\tau_i}\right)
\end{equation*}
defined for any $W_t$-geodesic $(\mu^\tau)_{\tau\in[0,1]}$ in $\Pz_t$.


The bounds \eqref{w-upp-lip} and \eqref{w-low-lip}
 imply
 $$-\lambda_t\cdot
W_t(\mu,\nu)\le
 \partial_t^- W_{t}(\mu,\nu)=
 \partial_t^+ W_{t}(\mu,\nu)
\le \kappa_t\cdot
W_t(\mu,\nu)$$
for all $\mu,\nu\in\Pz$ as well as
$$ -2\lambda_t\cdot
{\frak A}_t(\mu)\le {\frak B}_t(\mu)\le 2\kappa_t\cdot
{\frak A}_t(\mu)$$
for every $W_t$-geodesic $(\mu^\tau)_{\tau\in[0,1]}$ in $\Pz_t$.

\bigskip

\subsubsection*{\bf Time-dependent Metric Measure Spaces.} \ 
A time-dependent metric measure space is a family $\big(X,d_t,m_t\big)_{t\in I}$  
where 
$X$ is a Polish space. 
The parameter set $I\subset \R$ will be a left open interval.
For each $t$ under consideration, $m_t$ will be a Borel measure on $X$ and
$d_t$ will be a geodesic metric on $X$ which generates the given topology.
We also request that $\int\exp(-C_t d_t^2(x,z_t))dm_t(x)<\infty$ for some $z_t\in X$ and $C_t\in\R$.
Occasionally, we will assume that $d$ 
satisfies a lower log-Lipschitz bound \eqref{low-lip} or an upper log-Lipschitz bound \eqref{upp-lip}  with control function $\lambda\in L^1_{loc}(I)$ or $\kappa\in L^1_{loc}(I)$, resp.

The basic quantity for the subsequent considerations will be the time-dependent Boltzmann entropy
$$S: \quad I\times \Pz\to (-\infty,\infty], \quad (t,\mu)\mapsto \Ent(\mu|m_t)
.$$
For each $t\in I$, the function $S_t:\mu\mapsto \Ent(\mu|m_t)$ is 
lower semicontinuous on $\Pz_t$. Put $\Dom(S_t)=\{\mu\in\Pz_t: S_t(\mu)<\infty\}$.
More far reaching regularity will follow from synthetic curvature bounds.

\begin{lemma}\label{rcd-upp-reg}
If for a.e.\ $t\in I$ the static space $(X,d_t,m_t)$ satisfies a Riemannian curvature dimension-condition RCD$(-\kappa_t,\infty)$ for some $\kappa_t\in\R$ (in the sense of \cite{AGS-BE, EKS}) then $S$ is upper regular.
\end{lemma}

\begin{proof} RCD$(-\kappa_t,\infty)$ implies that $S_t$ is $-\kappa_t$-convex along every $W_t$-geodesic in $\Pz_t$. Since semiconvexity implies upper regularity this proves the claim.
\end{proof}

\subsection{Super-Ricci Flow via  Dynamic Convexity}

\begin{definition}
(i) We say that the time-dependent mm-space $\big(X,d_t,m_t\big)_{t\in I}$  is a \emph{super-Ricci flow} if the Boltzmann entropy $S$ is  \emph{strongly dynamical convex} on  $I \times \Pz$
in the sense of Definition \ref{dyn}: \  
for a.e.\ $t\in I$ and every $\mu^0,\mu^1\in \Pz_t$  there exists
  a $W_t$-geodesic $(\mu^\tau)_{\tau\in[0,1]}$ connecting $\mu^0$ and $\mu^1$  such that
  $\tau\mapsto S_t(\mu^\tau)$ is absolutely continuous on $[0,1]$ and
\begin{eqnarray}\label{dyn-conv-W}
\partial^+_\tau S_t(\mu^{1-})-\partial^-_\tau S_t(\mu^{0+})
&\ge&- \frac 12\partial_t^- W_{t-}^2(\mu^0,\mu^1).
\end{eqnarray}

(ii) Given a function $\lambda\in L^1_{loc}(I)$,  we say that the time-dependent mm-space $\big(X,d_t,m_t\big)_{t\in I}$ is a  \emph{$\lambda$-moderate  super-Ricci flow} if 
for a.e.\ $t\in I$ and every $\mu^0,\mu^1\in \Pz_t$  there exists
  a $W_t$-geodesic $(\mu^\tau)_{\tau\in[0,1]}$ connecting $\mu^0$ and $\mu^1$  such that for all $\tau\in(0,\frac12]$
\begin{eqnarray}\label{controlled-W}
\frac1\tau\left[S_t(\mu^0))-S_t(\mu^\tau)\right]
&+&\frac1\tau\left[S_t(\mu^1)-S_t(\mu^{1-\tau})\right]\\
&\ge&- \frac 12\partial_t^- W_{t-}^2(\mu^0,\mu^1)- \lambda_t \tau \cdot W^2_t(\mu^0,\mu^1).
\nonumber
\end{eqnarray}
\end{definition}


\begin{corollary}
(i) Every upper regular, $\lambda$-moderate  super-Ricci flow is a super-Ricci flow.

(ii) Every super-Ricci flow  satisfying a lower log-Lipschitz bound  \eqref{low-lip}   with control function $\lambda\in L^1_{loc}(I)$   is a $\lambda$-moderate  super-Ricci flow.
\end{corollary}

\begin{proof}
Theorem \ref{prop-dyn} (and Corollary \ref{w-control}).
\end{proof}

\begin{theorem}
Assume   that for a.e.\ $t\in I$ the static space $(X,d_t,m_t)$ is infinitesimally Hilbertian and that an
upper log-Lipschitz bound  \eqref{upp-lip}   with control function $\kappa\in L^1_{loc}(I)$ 
holds.

(i) If $\big(X,d_t,m_t\big)_{t\in I}$  is a $\lambda$-moderate  super-Ricci flow or a super-Ricci flow then  $S$ is upper regular.

(ii) If $\big(X,d_t,m_t\big)_{t\in I}$  is a  super-Ricci flow then for 
a.e.\ $t$ the static space 
$(X,d_t,m_t)$ satisfies  RCD$(-\kappa_t,\infty)$. 
\end{theorem}

\begin{proof}
(i) From \eqref{controlled-W}
with $\tau=\frac12$ we deduce that $S_t$ is $K_t$-midpoint convex (in the sense of Theorem \ref{mid-weak-str} (iii)) with $K_t=-(2\kappa_t+\lambda_t)$.
Selecting iteratively midpoints yields the $K_t$-convexity inequality 
\eqref{stat-conv} for each dyadic $\tau\in[0,1]$. Lower semicontinuity of $S_t$ allows to deduce this inequality for all $\tau\in[0,1]$. In other words, for the static space $(X,d_t,m_t)$ we have deduced the curvature dimension condition CD$(K_t,\infty)$. 
Together with the assumption of being infinitesimally Hilbertian this yields the RCD$(K_t,\infty)$-condition. (The curvature bound will be improved in part (ii) below.)
Thus in both cases, Lemma \ref{rcd-upp-reg} applies and states that $S$ is upper regular.  

(ii) The claim follows from Proposition \ref{log-ric}. For convenience, let us repeat directly the argument:
Thanks to upper regularity, the previous Corollary applies and proves \eqref{dyn-conv-W}. Combined with the log-Lipschitz bound it yields
\begin{eqnarray*}
\partial^+_\tau S_t(\mu_t^{1-})-\partial^-_\tau S_t(\mu_t^{0+})
\ge- \kappa_t W_{t}^2(\mu^0,\mu^1).
\end{eqnarray*}
The claim thus follows with Theorem \ref{slope-conv}.
\end{proof}

As already illustrated in the previous result, there are close links between the defining property of super-Ricci flows and synthetic Ricci bounds in the sense of Lott-Sturm-Villani.
Other results of this type will follow from Propositions \ref{stat-dyn-k} and \ref{dyn-dyn-k}. We re-formulate the easier of them.

\begin{proposition} Given an infinitesimally Hilbertian mm-space $(X,d,m)$,  a number $K\in\R$ and a non-empty interval $I\subset \{t\in\R: 2Kt<1\}$. Then the following are equivalent
\begin{itemize}

\item[(i)] $(X,d,m)$ satisfies the curvature-dimension condition RCD$(K,\infty)$;
\item[(ii)]  $\big(X,d_t,m_t\big)_{t\in I}$  is a  super-Ricci flow  where
$d_t^2=d^2\cdot(1-2Kt)$ and $m_t=m$ for all $t$ under consideration.
\end{itemize}
\end{proposition}

\begin{example}\label{neg-rim}
Let $X=[-1,1]\times S^1$ equipped with a metric $g$ of constant curvature $-1$ and glued together along the boundaries (i.e.\ identifying $(-1,\alpha)$ and $(+1,\alpha)$). Put
$g_t=(1+2t)g$. Let $d_t$ and $m_t$ be the associated Riemannian distance and Riemannian volume, resp. Then  $\big(X,d_t,m_t\big)_{t\in (-1/2,\infty)}$ is a super-Ricci flow.

It should be added that this certainly will not be a `minimal super-Ricci flow' (whatever definition for the latter one might choose). For example, at any time $t_0>-1/2$ one can start  to smoothen out the positive curvature concentrated up to time $t_0$ on the rim  and then to evolve according to the classical Ricci flow equation for smooth manifolds which  for $t\to\infty$ eventually  will approach a flat torus.
\end{example}

For smooth Riemannian manifolds, super-Ricci flows are nothing but super-solutions to the Ricci flow equation. (One should take into account, however, that these are not scalar quantities. Super-solutions have to be understood in the sense of inequalities between quadratic forms.)

A time-dependent weighted Riemannian manifold is a family 
$\big(M,g_t,\tilde f_t\big)_{t\in I}$ where $M$ is a manifold equipped with a 1-parameter family 
$(g_t)_{t\in I}$ of smooth Riemannian metric tensors and a family of measurable functions $\tilde f: M\to \R$ (the `weights'). It induces canonically a 
time-dependent mm-space $\big(X,d_t,m_t\big)_{t\in I}$ where $X=M$ and 
 $d_t$ is the Riemannian distance associated with $g_t$ and
 \begin{equation*}
 dm_t=e^{-\tilde f_t}d\mathrm{vol}_{g_t}
 \end{equation*}
  for all $t\in I$.

\begin{theorem}\label{thm-riem-srf}
Let $\big(M,g_t,\tilde f_t\big)_{t\in I}$ be a time-dependent Riemannian manifold with a smooth
family $(g_t)_{t\in I}$ of  metric tensors and a smooth family $(\tilde f_t)_{t\in I}$  of weights. Then the induced time-dependent mm-space $\big(X,d_t,m_t\big)_{t\in I}$ is a super-Ricci flow if and only if for all $t\in I$
\begin{equation}\label{riem-srf}
\Ric_{g_t}+\Hess_{g_t}\tilde f_t\ge -\frac12 \partial_t g_t
\end{equation}
in the sense of inequalities between quadratic forms on the tangent bundle of $(M,g_t)$.
\end{theorem}

\begin{proof} For the `if'-implication we have to verify that the Boltzmann entropy $S_t(.)=\Ent(.|m_t)$ is dynamically convex on $(\Pz,W_t)_{t\in I}$. Fix $t\in I$ (and drop it from the notation as far as possible). Given a $W_t$-geodesic $(\mu^\tau)_{\tau\in[0,1]}$ let
$F^\tau: X\to X$ for each $\tau$ be the unique optimal map with the push forward property
$(F^\tau)_\sharp \mu^0=\mu^\tau$. From seminal work of Otto, Otto/Villani, McCann, Cordero-Erausquin/McCann/Schmuckenschl\"ager, von Renesse/Sturm and others \cite{Ot,OV,CMS,vRS, St0} we know quite well that one can estimate the second derivative of the entropy (along geodesics in the Wasserstein space) in terms of the Ricci curvature of the underlying space together with the Hessian of the weight function.
Combined with the assumption \eqref{riem-srf} it immediately yields the estimate for the distributional second derivative the lower bound
\begin{eqnarray*}
\partial_\tau^2 S_t(\mu^\tau)&\ge&
\int_X \Big(\Ric_{g_t}+\Hess_{g_t}\tilde f_t\Big)\big(\dot F^\tau,\dot F^\tau\big)d\mu^0\\
&\ge&-\frac12
\int_X\partial_t g_t\big(\dot F^\tau,\dot F^\tau\big)d\mu^0.
\end{eqnarray*}
Integrating this estimate w.r.t.\ $\tau\in [0,1]$ we obtain
\begin{eqnarray*}
\partial_\tau^- S_t(\mu^{1-})-
\partial_\tau^+ S_t(\mu^{0+})
&\ge&
\int_0^1
\int_X \Big(\Ric_{g_t}+\Hess_{g_t}\tilde f_t\Big)\big(\dot F^\tau,\dot F^\tau\big)d\mu^0\,d\tau\\
&\ge&-\frac12 \int_0^1
\int_X\partial_t g_t\big(\dot F^\tau,\dot F^\tau\big)d\mu^0\,d\tau\\
&=&-\frac12 \partial_t\int_0^1
\int_X g_t\big(\dot F^\tau,\dot F^\tau\big)d\mu^0\,d\tau\\
&\ge&-\frac12 \partial_t^- W^2_{t-}(\mu^0,\mu^1).
\end{eqnarray*}
For the latter estimate, note that $W_t^2(\mu^0,\mu^1)=\int_0^1
\int_X g_t\big(\dot F^\tau,\dot F^\tau\big)d\mu^0\,d\tau$ whereas 
$W_s^2(\mu^0,\mu^1)\le \int_0^1
\int_X g_s\big(\dot F^\tau,\dot F^\tau\big)d\mu^0\,d\tau$ for all $s\not= t$.
This proves the forward implication.

\medskip

For the converse -- the `only if'-implication --  we assume that \eqref{riem-srf} is not satisfied. Then for some $\varepsilon>0$, some $s\in I$ and some $v\in TM$
\begin{equation*}
\Big(\Ric_{g_t}+\Hess_{g_t}\tilde f_t+\frac12 \partial g_t\Big)(w,w)\le -\varepsilon g_t(w,w)
\end{equation*}
for all points $(t,w)$ in a neighborhood of $(s,v)\in I \times TM$.
Now we can closely follow the argumentation in \cite{St0}, Thm. 5.2, to see that this contradicts the dynamic convexity of $S_t$.
\end{proof}

\begin{remark}
Assume for simplicity that $0\in I$ and $m_0=\mathrm{vol}_{g_0}$ and put $m_t=e^{-f_t}m_0$ as well as $\mathrm{vol}_{g_t}=e^{-\hat f_t}\mathrm{vol}_{g_0}$. Then
$f_t=\tilde f_t+\hat f_t$ and
$$\partial_t \hat f_t=-\frac12 \mathrm{tr} \partial_t g_t.$$
In particular, for a super-Ricci flow
$$\partial_t \hat f_t\le \mathrm{R}_{g_t}+\Delta_{g_t}\tilde f_t$$
where $ \mathrm{R}_{g_t}$ denotes the scalar curvature.
For an `un-weighted' manifold this simplifies to 
$\partial_t \hat f_t\le\mathrm{R}_{g_t}$ in the case of super-Ricci flows -- and with equality in the case of Ricci flows.
\end{remark}

\subsection{Super $N$-Ricci Flow}

The concept of super-Ricci flows will allow for a huge class of examples. Occasionally, it might be desirable to focus on a more restrictive class.
 
\begin{definition}\label{super-N-mms}
(i) We say that the time-dependent mm-space $\big(X,d_t,m_t\big)_{t\in I}$  is a \emph{super-$N$-Ricci flow} if the Boltzmann entropy $S$ is  \emph{strongly dynamical $N$-convex} on  $I \times \Pz$
in the sense of Definition \ref{dyn-Nconv}: \  
for a.e.\ $t\in I$ and every $\mu^0,\mu^1\in \Pz_t$  there exists
  a $W_t$-geodesic $(\mu^\tau)_{\tau\in[0,1]}$ connecting $\mu^0$ and $\mu^1$  such that
  $\tau\mapsto S_t(\mu^\tau)$ is absolutely continuous and 
\begin{eqnarray}\label{dyn-Nconv-W}
\partial^+_\tau S_t(\mu^{1-})-\partial^-_\tau S_t(\mu^{0+})
&\ge&- \frac 12\partial_t^- W_{t-}^2(\mu^0,\mu^1)
+\frac1N \big| S_t(\mu^0)-S_t(\mu^1)\big|^2.
\end{eqnarray}

(ii) Given a function $\lambda\in L^1_{loc}(I)$,  we say that the time-dependent mm-space $\big(X,d_t,m_t\big)_{t\in I}$ is a  \emph{$\lambda$-moderate  super-$N$-Ricci flow} if 
for a.e.\ $t\in I$ and every $\mu^0,\mu^1\in \Pz_t$  there exists
  a $W_t$-geodesic $(\mu^\tau)_{\tau\in[0,1]}$ connecting $\mu^0$ and $\mu^1$  such that for all $\tau\in(0,\frac12]$
  and all $N'\in[N,\infty]$ with $N'\ge 2\tau
   \left[ |S_t(\mu^0)-S_t(\mu^1)|+ \lambda_t W_t^2(\mu^0,\mu^1)/2\right]$
\begin{eqnarray}\label{N-dyn-conv-W}
\frac1\tau\Phi_{N'}\left(S_t(\mu^0))-S_t(\mu^\tau)\right)
&+&\frac1\tau\Phi_{N'}\left(S_t(\mu^1)-S_t(\mu^{1-\tau})\right)\\
&\ge&- \frac 12\partial_t^- W_t^2(\mu^0,\mu^1)- \lambda_t \tau \cdot W^2_t(\mu^0,\mu^1)
+\frac1{N'} \left| S_t(\mu^0)-S_t(\mu^1)\right|^2\nonumber
\end{eqnarray}
where $\Phi_{N'}(u)= u + \frac1{N'} u^2$.
\end{definition}

A super-$N$-Ricci flow is also a super-$N'$-Ricci flow for any $N'>N$. In particular, it is a super-$\infty$-Ricci flow.
Super-Ricci flow is the same as super-$N$-Ricci flow for $N=\infty$.

\begin{corollary}\label{2.13}
(i) Every upper regular, $\lambda$-moderate  super-$N$-Ricci flow  is a super-$N$-Ricci flow.

(ii) Every super-$N$-Ricci flow  satisfying a lower log-Lipschitz bound  \eqref{low-lip}   with control function $\lambda\in L^1_{loc}(I)$   is a $\lambda$-moderate  super-$N$-Ricci flow.
\end{corollary}

\begin{proposition}
Assume that  $\big(X,d_t,m_t\big)_{t\in I}$  is a  super-$N$-Ricci flow satisfying an
upper log-Lipschitz bound  \eqref{upp-lip}   with control function $\kappa\in L^1_{loc}(I)$ 
and that for a.e.\ $t\in I$ the static space $(X,d_t,m_t)$ is infinitesimally Hilbertian. Then for a.e.\ $t$ the static space
$(X,d_t,m_t)$ satisfies  RCD$^*(-\kappa_t,N)$. 
\end{proposition}

\begin{proof} The defining property of a super-$N$-Ricci flow combined with
 the log-Lipschitz bound for the distance yields
\begin{eqnarray*}
\partial^+_\tau S_t(\mu_t^{1-})-\partial^-_\tau S_t(\mu_t^{0+})
\ge- \kappa_t W_{t}^2(\mu^0,\mu^1)+ \frac1N
\big|S_t(\mu^{1})-S_t(\mu_t^{0})\big|^2
\end{eqnarray*}
for every $W_t$-geodesic in $\Pz$. This easily is seen (cf. proof of Theorem  \ref{slope-conv}) to be equivalent to the RCD$^*(-\kappa_t,N)$-condition as presented in \cite{EKS,AGS-BE}.
\end{proof}

In the spirit of Theorem \ref{thm-riem-srf} one may formulate and prove (following the lines of argumentation in the proof of that result as well as in \cite{St0} and in \cite{EKS}) an extension and tightening depending on the extra parameter $N$.

\begin{theorem}
Let $\big(M,g_t,\tilde f_t\big)_{t\in I}$ be a time-dependent Riemannian manifold of dimension $n$ with a smooth
family $(g_t)_{t\in I}$ of  metric tensors and a smooth family $(\tilde f_t)_{t\in I}$  of weights. Then the induced time-dependent mm-space $\big(X,d_t,m_t\big)_{t\in I}$ is a super-$N$-Ricci flow if and only if $N\ge n$ and for all $t\in I$
\begin{equation}\label{riem-N-srf}
\Ric_{g_t}+\Hess_{g_t}\tilde f_t-\frac1{N-n}\nabla_t \tilde f_t \otimes \nabla_t \tilde f_t
\ge -\frac12 \partial_t g_t
\end{equation}
in the sense of inequalities between quadratic forms on the tangent bundle of $(M,g_t)$. 

In particular for $N=n$ this requires $\tilde f_t$ to be constant. That is,
the measure $m_t$ has to be -- up to renormalization -- the Riemannian volume measure $\mathrm{vol}_t$, more precisely, $\exists C_t$ s.t. $m_t=C_t\cdot \mathrm{vol}_t$ for each $t$.
\end{theorem}

\begin{remark}
Similar characterizations will be possible for super-Ricci flows of Finsler spaces. Here the relevant quantity will be the \emph{flag Ricci curvature}, see \cite{OS1} for the static case.
\end{remark}

\subsection{Ricci Flows}

The previous characterization of super-Ricci flows is a natural extension of the synthetic lower bounds on the Ricci curvature to the time-dependent setting.
The concept of synthetic lower bounds for the Ricci curvature, formulated in terms of optimal transports, has led to plenty of results and new insights.
Despite the fact that during the last decade only lower bounds for the Ricci curvature have been studied,
there is also a surprisingly elementary, synthetic characterization of upper bounds for the Ricci curvature in terms of optimal transports.
It will be introduced and discussed in Appendix 1.
The time-dependent version will provide us with a natural definition of Ricci flows.

\begin{definition}
A time-dependent mm-space $\big(X,d_t,m_t\big)_{t\in I}$ is called \emph{weak sub-Ricci flow} if for every $\epsilon>0$ and for a.e.\ $t$ there exists a partition $X=\bigcup_i X_i$ such that for all $i$ and   for any pair of nonempty open sets $U^0,U^1\subset X_i$    there exists
  a $W_t$-geodesic $(\mu^\tau)_{\tau\in[0,1]}$ 
  such that
  $\tau\mapsto S_t(\mu^\tau)$ is finite and absolutely continuous on $(0,1)$, 
  $\mathrm{supp}(\mu^\sigma)\subset U^\sigma$ for $\sigma=0,1$ and  for all $0<\sigma<\rho<1$
\begin{eqnarray}\label{sub-ricci-flow}
\partial_\tau^+ S_t(\mu^{\tau})\big|_{\tau=\rho-}-\partial^-_\tau S_t(\mu^{\tau})\big|_{\tau=\sigma-}
\le- \frac 1{2(\rho-\sigma)}\partial_t^+ W_{t+}^2(\mu^\sigma,\mu^\rho)
+\epsilon.
\end{eqnarray}
A time-dependent mm-space $\big(X,d_t,m_t\big)_{t\in I}$ is called \emph{weak Ricci flow} if it is a super-Ricci flow and and a weak sub-Ricci flow.
\end{definition}

\begin{theorem}\label{thm-riem-sub-rf}
Let $\big(M,g_t,\tilde f_t\big)_{t\in I}$ be a time-dependent Riemannian manifold with a smooth
family $(g_t)_{t\in I}$ of  metric tensors and a smooth family $(\tilde f_t)_{t\in I}$  of weights. Then the induced time-dependent mm-space $\big(X,d_t,m_t\big)_{t\in I}$ is a weak sub-Ricci flow if and only if for all $t\in I$
\begin{equation}\label{riem-sub-rf}
\Ric_{g_t}+\Hess_{g_t}\tilde f_t\le -\frac12 \partial_t g_t
\end{equation}
in the sense of inequalities between quadratic forms on the tangent bundle of $(M,g_t)$.
\end{theorem}

\begin{proof}
For the `if'-implication we follow the argumentation from the proof of Theorem \ref{upper-Ric} in Appendix 1. Given open sets $U_0,U_1$ and $\varepsilon>0$, for each $t$ one can  construct a $W_t$-geodesic emanating in $U_0$ and ending in $U_1$ such that along this geodesic the second derivative of the entropy is almost given by the weighted Ricci curvature. More precisely,
\begin{eqnarray*}
\partial_\tau^2 S_t(\mu^\tau)&\le&
\int_X \Big(\Ric_{g_t}+\Hess_{g_t}\tilde f_t\Big)\big(\dot F^\tau,\dot F^\tau\big)d\mu^0+\varepsilon.
\end{eqnarray*}
Integrating this estimate w.r.t.\ $\tau\in [0,1]$ and combining it with the sub-Ricci flow property we obtain
\begin{eqnarray*}
\partial_\tau^- S_t(\mu^{1-})-
\partial_\tau^+ S_t(\mu^{0+})
&\le&
\int_0^1
\int_X \Big(\Ric_{g_t}+\Hess_{g_t}\tilde f_t\Big)\big(\dot F^\tau,\dot F^\tau\big)d\mu^0\,d\tau
+\varepsilon
\\
&\le&-\frac12 \int_0^1
\int_X\partial_t g_t\big(\dot F^\tau,\dot F^\tau\big)d\mu^0\,d\tau+\varepsilon\\
&=&-\frac12 \partial_t\int_0^1
\int_X g_t\big(\dot F^\tau,\dot F^\tau\big)d\mu^0\,d\tau+\varepsilon\\
&\le&-\frac12 \partial_t^+ W^2_{t+}(\mu^0,\mu^1)+\varepsilon.
\end{eqnarray*}

The `only if'-implication is easy. Assume that \eqref{riem-sub-rf} is not satisfied. Then for some $\varepsilon>0$ in a suitable neighborhood of some $(s,v)\in I\times TM$ the inequality 
\begin{equation*}
\Ric_{g_t}+\Hess_{g_t}\tilde f_t\ge -\frac12 \partial g_t+\epsilon \cdot g_t
\end{equation*}
holds true. Thus all transports with velocity fields in this neighborhood will lead to dynamic $\varepsilon$-convexity of the entropy -- being in contrast to the requested `nearly' dynamic concavity. 
\end{proof}

\begin{corollary} In the previous setting,
$\big(X,d_t,m_t\big)_{t\in I}$ is a weak Ricci flow if and only if for all $t\in I$
\begin{equation}\label{riem-rf}
\Ric_{g_t}+\Hess_{g_t}\tilde f_t= -\frac12 \partial_t g_t.
\end{equation}
\end{corollary}

\begin{example}

\

\begin{itemize}
\item[(i)] Let $X$ be the doubling of the unit disc $\overline B_1(0)\subset \R^n$ (= gluing of two copies along their boundaries) in $n\ge 2$. Then the static space $(X,d,m)$ is \emph{not} a weak Ricci flow. 

If $\mu^0$ is supported in one of the copies of $B_1(0)$ and $\mu^1$ in the other, then optimal transports on $X$ have to pass the rim and will descry
the focusing effect of the positive curvature concentrated on the rim.

\item[(ii)] 
Let $X$ be a cone over a circle of length $\alpha<2\pi$. Then the static space $(X,d,m)$ is  a weak Ricci flow. 

In this case, the positive curvature (concentrated in the vertex) will not be detected since optimal transports never pass through the vertex \cite{BaS}.

In the forthcoming second part to this paper, we will present a more restrictive notion of `Ricci flows' which will rule out this example.
\item[(iii)] 
Let $X$ be a cone over $S^2(1/\sqrt3)\times S^2(1/\sqrt3)$. Then the static space $(X,d,m)$ is  a weak Ricci flow.  

Same argument as for (ii) -- now, however, we indeed expect that this space is a `Ricci flow' in a more restrictive sense.
\item[(iv)] 
The time-dependent mm-space from Example \ref{neg-rim} is not a weak Ricci flow. The positive curvature concentrated in the glued boundaries will be detected by optimal transports.
(Same argument as for (i).)
\end{itemize}
\end{example}

\subsection{Averaged Super-Ricci Flows}

We started our discussion of super-Ricci flows for mm-spaces with the most comprehensive definition which involves derivatives w.r.t.\ space and time. 
In a first step, we already got rid of the spatial derivatives by passing to $\lambda$-moderate  super-Ricci flows. In the next step now we will illustrate how to get rid of the time-derivatives.

For simplicity, here and in the sequel we assume that $d$ is uniformly bounded on $I\times X$ 
and lower absolutely continuos in $t$, uniformly in $x,y$, in the following sense:
there exists a nonnegative function $\eta\in L^1_{loc}(I)$  such that for all $s<t$
\begin{equation*}
\label{low-abs-cont}
d_s^2(x,y)\le d_t^2(x,y)+\int_s^t \eta_r\,dr.
\end{equation*}
Boundedness of $d$ implies that  $\Pz_t=\Pz$ for all $t$  and the lower absolute continuity
 immediately carries over to the 
Wasserstein distances: 
\begin{equation}
W_s^2(x,y)\le W_t^2(x,y)+\int_s^t \eta_r\,dr.
\end{equation}

We will henceforth also assume that all the measures $m_t$ for $t\in I$ are absolutely continuous w.r.t.\ each other with uniformly bounded densities and that they measurably depend on $t$. In other words, there exist a 
bounded measurable function $f: I\times X\to \R$ and a Borel measure $m$ such that
 $m_t=e^{-f_t}m$ for all $t$.
Note that then 
$$S_t(\mu)=\Ent(\mu|m)+\int f_t\,d\mu$$
for all $\mu\in\Pz$.
Thus the condition $S_t(\mu)<\infty$ will be independent of $t\in I$ and  also equivalent to $S_J(\mu)<\infty$ for each interval $J\subset I$. 
Let $\Dom(S)$ denote the set of these $\mu\in\Pz$.

To proceed, we need some more notation. Given any subinterval $J=(r,s]\subset I$ we put
$S_J(\mu)=\frac1{s-r}\int_r^s S_t(\mu)dt$ for $\mu\in\Pz$ and 
$S_J(\mu_J)=\frac1{s-r}\int_r^s S_t(\mu_t)dt$ for a curve $\mu_J=(\mu_t)_{t\in J}$ in $\Pz$.
Moreover, we put
\begin{equation*}
W_J(\mu_J,\nu_J)=\left[\frac1{s-r}\int_r^s W_t^2(\mu_t,\nu_t)dt\right]^{1/2},\quad
W_{\lambda,J}(\mu_J,\nu_J)=\left[\frac1{s-r}\int_r^s W_t^2(\mu_t,\nu_t)\lambda_t dt\right]^{1/2}
\end{equation*}
for curves $\mu_J=(\mu_t)_{t\in J}$ and $\nu_J=(\nu_t)_{t\in J}$ in $\Pz$
(or, in other words, for points $\mu_J$ and $\nu_J$ in $\Pz^J$).
Note that a curve $(\mu_J^a)_{a\in [0,1]}$ in $\Pz^J$ is a geodesic w.r.t.\ $W_J$ if and only if for a.e.\ $t\in J$ the curve 
$(\mu_t^a)_{a\in [0,1]}$ in $\Pz$ is a geodesic w.r.t.\ $W_t$.

\begin{definition}
$\big(X,d_t,m_t\big)_{t\in I}$ is called \emph{$\lambda$-averaged  super-$N$-Ricci flow} if for every $J=(r,s]\subset I$
and for any pair of points $\mu^0,\mu^1\in \Dom(S)$   there exists
  a $W_J$-geodesic $(\mu_J^a)_{a\in[0,1]}$ connecting $\mu^0$ and $\mu^1$  (i.e.\
  $\mu^a_t=\mu^a$ for $a\in\{0,1\}$ and a.e.\ $t\in J$) such that for all $a\in(0,\frac12]$
  and all $N'\in[N,\infty]$ with $N'\ge 2a
   \left[ |S_J(\mu^0)-S_J(\mu^1)|+  W_{\lambda,J}^2(\mu^0,\mu^1)\right]$
\begin{eqnarray}\label{N-int-conv-W}\nonumber
\frac1a\Phi_{N'}\left(S_J(\mu^0))-S_J(\mu_J^a)\right)
&+&\frac1a\Phi_{N'}\left(S_J(\mu^1)-S_J(\mu_J^{1-a})\right)\\
&\ge&- \frac 1{2(s-r)}\left[W_s^2(\mu^0,\mu^1)-W_r^2(\mu^0,\mu^1)\right]
\\
&&-  a \cdot W^2_{\lambda,J}(\mu^0,\mu^1)
+\frac1{N'} \left| S_J(\mu^0)-S_J(\mu^1)\right|^2\nonumber.
\end{eqnarray}
averaged $\lambda$-moderate  super-$N$-Ricci flows in the case $N=\infty$ will be called \emph{averaged $\lambda$-moderate  super-Ricci flows}.
\end{definition}

\begin{proposition}
$\big(X,d_t,m_t\big)_{t\in I}$ is a $\lambda$-averaged  super-Ricci flow if and only it is a $\lambda$-moderate  super-Ricci flow.
\end{proposition}

\begin{proof}
Let us first note that  $$W^2_s(\mu^0,\mu^1)-W^2_r(\mu^0,\mu^1)\ge\int_r^s\left( \partial_t W^2_t\right)(\mu^0,\mu^1)dt$$ since
 $t \mapsto W^2_t(\mu^0,\mu^1)$ by assumption is lower absolutely continuous. 

Now let us assume that $\big(X,d_t,m_t\big)_{t\in I}$ is a $\lambda$-moderate  super-Ricci flow.
We can find a measurable selection of the connecting geodesics (i.e.\ measurable w.r.t.\ $t$) such that we can integrate the inequality
\begin{eqnarray}\nonumber
\frac1a\left(S_t(\mu^0))-S_t(\mu_t^a)\right)
&+&\frac1a\left(S_t(\mu^1)-S_t(\mu_t^{1-a})\right)\\
&\ge&- \frac 1{2}\partial_t^-W_{t-}^2(\mu^0,\mu^1)
- \lambda_t a \cdot W^2_t(\mu^0,\mu^1)
\end{eqnarray}
w.r.t.\ $t$. This already proves that $\big(X,d_t,m_t\big)_{t\in I}$ is a $\lambda$-averaged  super-Ricci flow.

Conversely,  the a.e.-version follows from the integrated version  by applying Lebesgue's density theorem. 
To do so, note that in the integrated version one may choose a common family of connecting geodesics  $(\mu^a_t)_{a\in[0,1]}$, $t\in(r,s]$, for a dense countable set of interval endpoints $r,s$.
\end{proof}

\begin{proposition}\label{aver-ae-point}
 Assume that the metrics $(d_t)_{t\in I}$ satisfy a lower log-Lipschitz bound with control function $\lambda\in L^1_{loc}(I)$ and that for a.e.\ $t$ the static mm-space $(X,d_t,m_t)$ satisfies a RCD$(-\kappa_t,\infty)$-condition for some $\kappa_t\in\R$.
 Then 
$\big(X,d_t,m_t\big)_{t\in I}$ is a  $\lambda$-moderate  super-$N$-Ricci flow if and only if it is a $\lambda$-averaged  super-$N$-Ricci flow.
\end{proposition}

\begin{proof}
The case $N=\infty$ is already covered by the previous Proposition. The case $N<\infty$ requires some more care.
Assume that $\big(X,d_t,m_t\big)_{t\in I}$ is a $\lambda$-moderate  super-$N$-Ricci flow and that a.e.\ of the static spaces is a RCD$(-\kappa_t,\infty)$-space. The latter implies that the Boltzmann entropy is upper regular. Hence, by Theorem \ref{prop-N-dyn} for a.e.\ $t$, every $W_t$-geodesic 
$(\mu_t^a)_{a\in[0,1]}$  and
for all $a\in [0,1/2]$
\begin{eqnarray*}
S_t(\mu_t^0)+S_t(\mu_t^1)&-& S_t(\mu_t^a)-S_t(\mu_t^{1-a})\\
&\ge& -\frac12\int_0^a \frac1{1-2b}
\big(\partial^-_t W^2_{t-}\big)(\mu_t^{b},\mu_t^{1-b})\,db
+\frac aN \int_0^1\Lambda^{a,b} |\partial_b S_t(\mu_t^b)|^2\,db.
\end{eqnarray*}
Choosing a measurable selection of interpolating geodesics, this can be integrated w.r.t.\ $t$ which then  yields
\begin{eqnarray}
\label{s-J} \nonumber
\lefteqn{S_J(\mu_J^0)+S_J(\mu_J^1)- S_J(\mu_J^a)-S_J(\mu_J^{1-a})}\\
&\ge& -\frac1{2|J|}\int_J \int_0^a \frac1{1-2b}
\big(\partial^-_t W^2_{t-}\big)(\mu_t^{b},\mu_t^{1-b})\,db\,dt
+\frac a{N}\int_0^1\Lambda^{a,b}\frac1{|J|}\int_J |\partial_b S_t(\mu_t^b)|^2\,dt\,db.
\end{eqnarray}
The last term allows to apply Cauchy-Schwarz inequality and then can be estimated as in the proof of Theorem \ref{prop-N-dyn}
as follows
\begin{eqnarray*}
\lefteqn{\frac a{N}\int_0^1\Lambda^{a,b}\frac1{|J|}\int_J |\partial_b S_t(\mu_t^b)|^2\,dt\,db}\\
&\ge&
\frac a{N}\int_0^1\Lambda^{a,b} |\partial_b S_J(\mu_J^b)|^2\,db\\
&\ge&\frac a{N(1-2a)}\big| S_J(\mu_J^a)-S_J(\mu_J^{1-a})\big|^2\\
&\ge&
\frac1{N'}\Big[a\big| S_J(\mu_J^0)-S_J(\mu_J^{1})\big|^2-\big| S_J(\mu_J^0)-S_J(\mu_J^{a})\big|^2-\big| S_J(\mu_J^{1-a})-S_J(\mu_J^{1})\big|^2\Big]
\end{eqnarray*}
for any $N'\ge N$
whereas the first term on the RHS of \eqref{s-J} will be estimated as in the proof of Theorem \ref{prop-dyn} in the following way
\begin{eqnarray*}
\lefteqn{-\frac1{2|J|}\int_J \int_0^a \frac1{1-2b}
\big(\partial^-_t W^2_{t-}\big)(\mu_t^{b},\mu_t^{1-b})\,db\,dt}\\
&\ge&
\frac1{2|J|}\int_J \int_0^a\Big[ 
-\big(\partial^-_t W^2_{t-}\big)(\mu_t^{0},\mu_t^{1})
+\frac1 b \big(\partial^-_t W^2_{t-}\big)(\mu_t^{0},\mu_t^{b})
+\frac1b \big(\partial^-_t W^2_{t-}\big)(\mu_t^{1-b},\mu_t^{1})
\Big]
\,db\,dt
\end{eqnarray*}
Requiring now that all geodesics under consideration will start and end in the same points $\mu_0$ and $\mu_1$, independent of $t$, the first term on the RHS of the last 
inequality can be expressed more explicitly as
$$\frac1{2|J|}\int_J \int_0^a
-\big(\partial^-_t W^2_{t-}\big)(\mu^{0},\mu^{1})
\,db\,dt=\frac a{2(s-r)}\big[ W^2_r(\mu^0,\mu^1)-W^2_s(\mu^0,\mu^1)
\big].$$
Moreover, the requested log-Lipschitz bound of the metric allows to estimate the two remaining terms on the RHS of the previous inequality from below as follows
\begin{eqnarray*}
\lefteqn{\frac1{2|J|}\int_J \int_0^a\Big[ \frac1 b \big(\partial^-_t W^2_{t-}\big)(\mu_t^{0},\mu_t^{b})
+\frac1b \big(\partial^-_t W^2_{t-}\big)(\mu_t^{1-b},\mu_t^{1})
\Big]
\,db\,dt}\\
&=&-\frac1{|J|}\int_J \int_0^a\lambda_t\cdot W_t^2(\mu^0,\mu^1)\, 2b\,db\,dt\\
&=&-a^2\cdot W^2_{\lambda,J}(\mu^0,\mu^1).
\end{eqnarray*}
This proves the claim.

\medskip

For the converse, note that the Cauchy-Schwarz inequality allows to estimate the left hand side of \eqref{N-int-conv-W} from above
by
$$
\frac1a\frac1{s-r}\int_r^s \Phi_{N'}\left(S_t(\mu^0))-S_t(\mu_t^a)\right)
+\Phi_{N'}\left(S_t(\mu^1)-S_t(\mu_t^{1-a})\right)dt.$$
Moreover, for each $a\in\{0,1\}$ boundedness and measurability of 
 $t\mapsto F_t=\int f_t(x)d\mu^a$ implies that for a.e.\ $t$
\begin{eqnarray*}
S_t(\mu^a)-S_{J}(\mu^a)= F_t-\frac1{s-r}\int_{r}^s F_u\,du\to 0
\end{eqnarray*}
whenever $r\to t$ and $s\to t$. This allows to handle the last term on the right hand side of 
 \eqref{N-int-conv-W}. Hence, we may apply Lebesgue's density theorem as in the proof of the previous Proposition to deduce the  claim for a.e.\ $t$. 
\end{proof}

Under appropriate assumptions, the averaged formulation of super-Ricci flows even allows to deduce that the defining property holds true pointwise for \emph{every} $t$.

\begin{proposition}\label{aver-point}
 Assume that  $X)$ is compact and that $f$ is continuous in $x$ and  $t$.
Then every  $\lambda$-averaged super-$N$-Ricci flow
$\big(X,d_t,m_t\big)_{t\in(0,T]}$ satisfies the defining inequality \eqref{N-dyn-conv-W} for every $t\in I$.
\end{proposition}

\begin{proof}
For fixed $s\in I$ consider the averaged super-$N$-Ricci flow property \eqref{N-int-conv-W} on time intervals $J_n=[r_n,s]$ for a sequence of numbers $r_n\nearrow s$. 
 Put $\mu^a_{J_n}=\frac1{s-r_n}\int_{r_n}^s \mu_t^a dt$. Compactness of $X$ implies that for each $a\in [0,1]$ -- after passing to a suitable subsequence -- $\mu_{J_n}^a\to \mu_s^a$ weakly as $n\to\infty$ for some $W_s$-geodesic $(\mu^a_s)_{s\in [0,1]}$.
 (Indeed, this first will be proven for all rational $a$.)
 The claimed pointwise super-$N$-Ricci flow property at the given time $s$ will follow from the averaged versions if we know that
 \begin{itemize}
 \item for $a\in\{0,1\}$
 $$\limsup_{n\to\infty}  S_{J_n}(\mu^a)\le S_s(\mu^a)$$
 \item for all $a\in [0,1]$
 $$\liminf_{n\to\infty}  S_{J_n}(\mu_{J_n}^a)\ge S_s(\mu_s^a).$$
 \end{itemize}
The former follows (as in the final step of the previous proof) from continuity of $t\mapsto f_t(x)$. Recall that $S_t$ depends on $t$ only via $f_t$. More precisely,
$S_t(\mu)=S_\diamond(\mu)+\int f_td\mu$ with $S_\diamond(\mu):=\Ent(\mu|m)$.

 For the latter, note that Jensen's inequality and lower semicontinuity of $\mu\mapsto S_\diamond (\mu)$ imply
 $$\liminf_{n\to\infty}  \frac1{s-r_n}\int_{r_n}^sS_{\diamond}(\mu_{t}^a)dt\ge \liminf_{n\to\infty}S_\diamond(\mu_{J_n}^a)\ge S_\diamond(\mu_s^a)$$
 for all $a\in [0,1]$. Moreover,  continuity of $f$  in $t$  and  $x$ (in fact, uniform continuity, due to compactness) implies
 $$\liminf_{n\to\infty}  \frac1{s-r_n}\int_{r_n}^s\int_X f_t \,d\mu_{t}^a\, dt\ge 
 \liminf_{n\to\infty}  \frac1{s-r_n}\int_{r_n}^s \int_X f_s \,d\mu_{t}^a\, dt\ge 
  \int_X  f_s \,d\mu_{s}^a.$$
  Since 
  $$S_{J_n}(\mu^a_{J_n})= \frac1{s-r_n}\int_{r_n}^sS_{\diamond}(\mu_{t}^a)dt
  + \frac1{s-r_n}\int_{r_n}^s\int_X f_t \,d\mu_{t}^a\, dt$$ and 
  $S_s(\mu^a_s)=S_\diamond(\mu^a_s)+\int f_s\,d\mu^a_s$ this proves the claim.
\end{proof}

\section{Stability and Compactness of Super-Ricci Flows}

One of the most important property of the synthetic Ricci bounds  in the sense of Lott-Sturm-Villani for `static' mm-spaces is their stability under convergence -- and as a consequence a far reaching compactness result which extends Gromov's famous pre-compactness result. Convergence in this context is most naturally formulated in terms of optimal transports. It essentially coincides with measured Gromov-Hausdorff convergence (but leads to slightly stronger results).
Let us now try to carry over these concepts and results to the time-dependent setting.

\subsection{The Distance ${\mathbb D}_I$ between Time-dependent MM-Spaces}

For the sequel, we fix a bounded interval $I$ of length $|I|$.
The mm-spaces $(X,d_t,m_t)_{t\in I}$ to be considered 
will always be requested to have finite measures $m_t$ which are equivalent to each other with bounded densities. Thus we always may represent them as
$$m_t=e^{-f_t}m$$
with a family of bounded measurable functions $(f_t)_{t\in I}$ and some `reference' probability measure $m$. 
Instead of $(X,d_t,m_t)_{t\in I}$ we then will write $(X,d_t,f_t,m)_{t\in I}$.
The choice of the probability measure $m$, a priori, is not unique. 
For well-definedness, we assume that it is the normalized $m_T$ for $T$ being the midpoint of $I$, i.e.\ $m=\frac1{m_T(X)}\, m_T$ with $T=\frac12(T_0+T_1)$ for $I=(T_0,T_1]$.
(Any other choice of $T\in I$  also would be fine.)

In the spirit of the ${\mathbb D}$-distance between equivalence classes of mm-spaces introduced by the author in \cite{St1},
 the $L^{2,1}$-transportation distance between equivalence classes of time-dependent mm-spaces will be defined by
 
  \begin{eqnarray*}
\lefteqn{ {\mathbb D}_I\Big( (X,d_t,f_t,m), (\tilde X,\tilde d_t,\tilde f_t,\tilde m)\Big)}\nonumber\\
&=&
 \inf\left\{
 \left(\frac1{|I|}\int_I\int_{X\times \tilde X} \hat d_t(x,y)^2\, d\hat m(x,y)\,dt\right)^{1/2}
 +
\frac1{|I|}\int_I\int_{X\times \tilde X} | f_t(x)-\tilde f_t(y)|\, d\hat m(x,y)\,dt\right\}
 \end{eqnarray*}
 where the imfimum is taken over all  $\hat m\in \mathrm{Cpl}(m, \tilde m)$ and over all $\hat d_t\in \mathrm{Cpl}(d_t, \tilde d_t)$ for a.e.\ $t\in I$.
 Here $\mathrm{Cpl}(m, \tilde m)$ denotes the well-known set of all couplings of the probability measures $m, \tilde m$ (i.e.\ measures on the product space $X\times \tilde X$ with the given measures as marginals). Analogously,  $\mathrm{Cpl}(d_t, \tilde d_t)$ denotes the set of couplings of the metrics $d_t$ and $\tilde d_t$. These are pseudo-metrics on the disjoint union of $X$ and $\tilde X$ which on the respective subsets coincide with the given metrics, see \cite{St1} for the precise definition and a detailed discussion.
 
Time-dependent mm-spaces  $(X,d_t,f_t,m)_{t\in I}$ and  $(\tilde X,\tilde d_t,\tilde f_t, \tilde m)_{t\in I}$ will be considered as  equivalent if their ${\mathbb D}_I$-distance is 0.
 
 More general $L^{p,q}$-transportation distances could be considered which minimize 
 $$ \left(\int_I\int_{X\times \tilde X} \hat d_t(x,y)^p\, d\hat m(x,y)\,dt\right)^{1/p}
 +
 \left(\int_I\int_{X\times \tilde X} | f_t(x)-\tilde f_t(y)|^q\, d\hat m(x,y)\,dt\right)^{1/q}$$ instead of the previous 
  expressions. In most of the subsequent applications, however, both the distances $d_t$ and the weights $f_t$ will be uniformly bounded. Then all these transportation distances will be equivalent to each other and thus for simplicity we restrict to the case $p=2, q=1$.
 
 \begin{lemma}
 ${\mathbb D}_I$ is a metric on the space of equivalence classes of time-dependent mm-spaces.
 \end{lemma}
 
 \begin{proof}
 The triangle inequality follows with the standard argument based on the gluing lemma.
 \end{proof}

Given a bounded interval $I$, a continuous increasing function (`modulus of continuity')  $\Phi:\R_+\to\R_+$ with $\Phi(0)=0$, and constants $K, L\in(0,\infty)$, let ${\mathbb X}_I(K, L.\Phi)$ denote the space of equivalence classes of time-dependent mm-spaces $(X,d_t,f_t,m)_{t\in I}$ 
 \begin{itemize}
 \item with compact separable $X$, probability measure $m$,
  \item with  bounded measurable $(t,x)\mapsto f_t(x)$, lower semicontinuous in $x$,
  \item with geodesic metrics $d_t$ on $X$ each of which generates the topology of $X$, uniformly bounded and uniformly continuous in $t$
  \begin{equation}
  \label{dt-lip}
 d_t(x,y)\le L, \qquad |d_t(x,y)-d_s(x,y)|\le \Phi( |s-t|)
 \end{equation}
 \item and with uniform lower bounds on the Ricci curvature in the sense of Lott-Sturm-Villani, i.e.\  each static space $(X,d_t,m_t)$ satisfies the curvature-dimension condition CD$(-K,\infty)$.
 \end{itemize}

\begin{lemma}\label{DI-Dt}
For time-dependent mm-spaces 
$(X,d_t,f_t,m)_{t\in I}$ and $(\tilde X,\tilde d_t,\tilde f_t,\tilde m)_{t\in I}$ within the class ${\mathbb X}_I(K,L,\Phi)$ and each $s\in I$
$$
 {\mathbb D}\Big( \big(X,d_s,m\big), \big(\tilde X,\tilde d_s,\tilde m\big)\Big)\le \Phi_1\Big(|I|^{1/3}\cdot 
{\mathbb D}_I\Big( \big(X,d_t,f_t,m\big)_{t\in I}, \big(\tilde X,\tilde d_t,\tilde f_t,\tilde m\big)_{t\in I}\Big)^{2/3}\Big)$$
where $\Phi_1(r)=\Phi(r)+r$.
\end{lemma}

\begin{remark*}
 Calculations with transport distances require quite some care: for instance, equivalence of distances (which is a `multiplicative estimate') does not allow for any estimate on the correponding transport distances (which is an `additive estimate').
 \end{remark*}
 
\begin{proof}
Choose $\epsilon>
{\mathbb D}_I\Big( \big(X,d_t,f_t,m\big)_{t\in I}, \big(\tilde X,\tilde d_t,\tilde f_t,\tilde m\big)_{t\in I}\Big)$. Then there exist couplings $\hat d_t$ of $d_t$ and $\tilde d_t$ (for a.e.\ $t$) as well as a coupling $\hat m$ of $m$ and $\tilde m$ such that
$$
\int_I\int_{X\times \tilde X} \hat d_t(x,y)^2\, d\hat m(x,y)\,dt
\le\epsilon^2 
\cdot {|I|}.
$$
For $s\in J\subset I$ define $d_J(x,y)=\left(\frac1{|J|}\int_J d_t^2(x,y)dt\right)^{1/2}$ and similarly $\tilde d_J(x,y)$.
Then $$\hat d_J(x,y)=\left(\frac1{|J|}\int_J \hat d_t^2(x,y)dt\right)^{1/2}$$
is a coupling of $d_J$ and $\tilde d_J$. Thus
\begin{equation}
 {\mathbb D}^2\Big( \big(X,d_J,m\big), \big(\tilde X,\tilde d_J,\tilde m\big)\Big)\le 
\frac1{|J|}\int_J\int_{X\times \tilde X} \hat d_t(x,y)^2\, d\hat m(x,y)\,dt\le \frac{|I|}{|J|}\cdot 
\epsilon^2.
\end{equation}
The bound \eqref{dt-lip}  implies
$|d_t(x,y)-d_s(x,y)|\le \Phi( |J|)$ for all $t\in J$ and thus
$$|d_J(x,y)-d_s(x,y)|\le \Phi( |J|).$$
Hence, the distances $d_J$ and $d_s$ -- being regarded as distances on two disjoint copies of $X$ -- can be coupled by defining the distance of corresponding points by $\Phi(|J|)/2$. In other words,
$$
 {\mathbb D}\Big( \big(X,d_s,m\big), \big(X,d_J, m\big)\Big)\le \Phi( |J|)/2.$$
 And of course with the same argument
 $
 {\mathbb D}\Big( \big(\tilde X,\tilde d_s,\tilde m\big), \big(\tilde X,\tilde d_J, \tilde m\big)\Big)\le \Phi( |J|)/2$.
 Thus
 \begin{eqnarray*}
  {\mathbb D}\Big( \big(X,d_s,m\big), \big(\tilde X,\tilde d_s,\tilde m\big)\Big)&\le& 
   {\mathbb D}\Big( \big(X,d_s,m\big), \big(X,d_J, m\big)\Big)
   + {\mathbb D}\Big( \big(X,d_J,m\big), \big(\tilde X,\tilde d_J,\tilde m\big)\Big)\\
 &&+   {\mathbb D}\Big( \big(\tilde X,\tilde d_s,\tilde m\big), \big(\tilde X,\tilde d_J, \tilde m\big)\Big)\\
 &\le& \Phi( |J|)+\sqrt{\frac{|I|}{|J|}}\cdot 
\epsilon.
 \end{eqnarray*}
 Choosing $J={|I|}^{1/3}\cdot \epsilon^{2/3}$
 yields the claim.
\end{proof}

\subsection{The Stability Result for  Super-Ricci Flows}

 \begin{theorem}[Stability]\label{3.3}
 For each $N\in[1,\infty]$, fixed $K,L, \Phi$ and $\lambda\in L^1_{loc}(I)$, the class of $\lambda$-averaged  super-$N$-Ricci flows within ${\mathbb X}_I(K, L,\Phi)$ is closed w.r.t.\ ${\mathbb D}_I$-convergence.
 \end{theorem}
 That is, `uniformly bounded' ${\mathbb D}_I$-limits of $\lambda$-averaged  super-$N$-Ricci flows are $\lambda$-averaged  super-$N$-Ricci flows.

\begin{proof}
(i)
We follow the strategy from \cite{St1} for the proof of stability for synthetic Ricci bounds in the static case.
Fix $K, L\in(0,\infty)$ and $N\in[1,\infty]$ as well as a sequence of time-dependent mm-spaces
 $(X^n,d_t^n,f^n_t,m^n)_{t\in I}$ which converges to a mm-space
  $(X,d_t,f_t,m)_{t\in I}$.
For each $n$ choose an optimal (or almost optimal) coupling $q^n(dx,dy)$ of $m(dx)$ and $m^n(dy)$. 
Define Markov kernels
$q^n(x,dy)$ and $q^n(dx,y)$ via disintegration of $q^n(dx,dy)$ by
$$q^n(dx,dy)=q^n(x,dy)m(dx)=q^n(dx,y)m^n(dy).$$
These kernels later will be used 
to transport probability densities from $X^n$ to $X$ and vice versa.

Let probability measures $\mu^0$ and $\mu^1$ on $X$ be given as well as an interval $J\subset I$ and numbers $a\le \frac12$ and $N'\ge N$ satisfying
$N'\ge 2a
   \left[ |S_J(\mu^0)-S_t(\mu^1)|+  W_{\lambda,J}^2(\mu^0,\mu^1)/2\right]$.
Indeed, we may assume  strict inequality 
\begin{equation}\label{N-a-bound}
N'>2a
   \left[ |S_J(\mu^0)-S_t(\mu^1)|+  W_{\lambda,J}^2(\mu^0,\mu^1)/2\right].
\end{equation} (The limit case will follow by approximation.)

By means of the `transport kernel' $q^n (dx,y)$ we now transport the measures $\mu^0$ and $\mu^1$ from $X$ to $X^n$. That is, we define corresponding probability measures 
$\mu^{n,b}=\rho^{n,b}\cdot m^n$ for $b\in\{0,1\}$ where
$\rho^{n,b}(y)=\int_X \rho^{b}(x) \, q^n (dx,y)$.
%
Finally, these probability measures 
will be transported via the kernel $q^n (x,dy)$ onto probability measures 
$\tilde\mu_t^{n,b}=\tilde\rho_t^{n,b}\cdot m$ on $X$ for  $b\in \{0,1\}$ with
$\tilde\rho_t^{n,b}(x)=\int_{X^n} \rho_t^{n,b}(y) \, q^n (x,dy)$.

\medskip

(ii)
According to the previous Lemma, for each $t\in I$ as $n\to\infty$ the mm-spaces $(X^n,d^n_t,m^n)$ will converge to the mm-space $(X,d_t,m)$ w.r.t.\ the metric ${\mathbb D}$.
More precisely, we know that
$$\int \big(\hat d_t^n\big)^2\, d\hat m^n\to0$$
as $n\to\infty$ where $\hat m^n$ is a coupling of $m$ and $m^n$, (almost) optimal w.r.t.\ the defining functional for ${\mathbb D}_I$.

Thus according to \cite{St1}, Lemma 4.19, we 
can conclude that
$$\hat W_t^n\big(\mu^b,\mu^{n,b}\big)\to0$$
for $b=0,1$ and thus
$$W_t^n(\mu^{n.0},\mu^{n,1})\to W_t(\mu^0,\mu^1)$$
as $n\to\infty$. 
This property will be requested for the endpoints of the interval $J$, i.e. for $t=r$ and $t=s$ in case $J=(r,s]$.

\medskip

(iii)
Without restriction, we may assume that the given measures $\mu^0$ and $\mu^1$ have bounded densities w.r.t.\ $m$, say
$\rho^0\le C'$ and $\rho^1\le C'$.
This bound is preserved by the transport kernel. 
Thus 
$\rho^{n,b}\le C'$ for $b\in\{0,1\}$ and  all $n$. 
Therefore, for $b\in\{0,1\}$
\begin{eqnarray*}
\left| \int_J \int_X f_td\mu^bdt-  \int_J \int_{X^n} f^n_td\mu^{n,b}dt\right|
&=&\left|
 \int_J \int_X \int_{X^n} \left[ f_t(x)-  f^n_t(y)\right] \rho^b(x)  q^n(dx,dy)\,dt\right|\\
 &\le & C'
 \int_J \int_X \int_{X^n} \left| f_t(x)-  f^n_t(y)\right| q^n(dx,dy)\,dt\\
 &\le & C'\cdot |I| \cdot {\mathbb D}_I\Big( 
 (X^n,d^n_t,f^n_t,m^n), (X,d_t,f_t,m)\Big)
\end{eqnarray*}
and similarly
\begin{eqnarray*}
\left| \int_J \int_X f_td\tilde\mu^{n,b}dt-  \int_J \int_{X^n} f^n_td\mu^{n,b}dt\right|
&=&\left|
 \int_J \int_X \int_{X^n} \left[ f_t(x)-  f^n_t(y)\right]  \rho_t^{n,b}(y) q^n(dx,dy)\,dt\right|\\
 &\le & C'\cdot |I|\cdot {\mathbb D}_I\Big( 
 (X^n,d^n_t,f^n_t,m^n), (X,d_t,f_t,m)\Big).
\end{eqnarray*}
Let us consider the entropy w.r.t.\ the reference measures
$S_\diamond(\mu)=\Ent(\mu| m)$ and $S^n_\diamond(\mu)=\Ent(\mu| m^n)$.
Since entropy decreases under the transport kernels (\cite{St1}, Lemma 4.19), it follows that 
$$ S_\diamond(\mu^0)\ge S^n_\diamond(\mu^{n,0})\ge S_\diamond(\tilde\mu^{n,0})$$
and
$$ S_\diamond(\mu^1)\ge S^n_\diamond(\mu^{n,1})\ge S_\diamond(\tilde\mu^{n,1}).$$
Thus, in particular, $S_\diamond(\tilde\mu^{n,b})\to S_\diamond(\mu^b)$ as $n\to\infty$
by lower semicontinuity of $S_\diamond$ and the fact that $\tilde\mu^{n,b}\to \mu^b$ for $b\in\{0,1\}$ (again by adaption of the argument in \cite{St1}, Lemma 4.19: it yields convergence in transportation distance induced by the metric $(\int_J d_t^2\,dt)^{1/2}$ on $X$ and thus  weak convergence).
Therefore in turn also
$$S_\diamond^n(\mu^{n,b})\to S_\diamond(\mu^b).$$

\medskip

(iv)
Recall that
$$S_t(\mu)=S_\diamond(\mu)+\int f_td\mu \quad \mbox{and}\quad
S^n_t(\mu)=S^n_\diamond(\mu)+\int f^n_td\mu $$
for all $\mu$. Hence,
\begin{equation} S_J^n(\mu^{n,b}) \to  S_J(\mu^b) \
\mbox{ as well as } \
S_J(\tilde\mu^{n,b}) \to  S_J(\mu^b)
\end{equation}
as $n\to\infty$ for $b\in\{0,1\}$.
Moreover, by Lebesgue's dominated convergence theorem (since $\lambda$ is integrable and all the $W_t(.,.)$'s are uniformly bounded and converging for a.e.\ $t$)
\begin{equation} W^n_{\lambda,J}\big(\mu^{n,0},\mu^{n,1}\big)\to W_{\lambda,J}\big(\mu^{0}, \mu^1\big)
 \
\mbox{ and } \
W_{\lambda,J}\big(\tilde\mu^{n,0},\tilde\mu^{n,1}\big)\to W_{\lambda,J}\big(\mu^{0}, \mu^1\big)
\end{equation}
as $n\to\infty$ .

Thus for sufficiently large $n$ the constraints 
\begin{equation}\label{cons-n}
N'> 2a
   \left[ |S^n_J(\mu^{n,0})-S^n_J(\mu^{n,1})|+ W^n_{\lambda,J}(\mu^{n,0},\mu^{n,1})^2/2\right]
   \end{equation}
   as well as
   \begin{equation}\label{cons-til}
   N'> 2a
   \left[ |S_J(\tilde\mu^{n,0})-S_J(\tilde\mu^{n,1})|+  W_{\lambda,J}(\tilde\mu^{n,0},\tilde\mu^{n,1})^2/2\right]
   \end{equation}
   will be satisfied. Inequality \eqref{cons-n} allows to apply the defining property of averaged super-$N$-Ricci flows 
   which guarantees that
for a.e.\ $t$ under consideration there exists a
 $W_t^n$-geodesics $(\mu_t^{n,b})_{b\in[0,1]}$ of probability measures on $X^n$ connecting $\mu^{n,0}$ and $\mu^{n,1}$ such that
   \begin{eqnarray}\nonumber
\frac1b\Phi_{N'}\left(S^n_J(\mu^{n,0}))-S^n_J(\mu_J^{n,b})\right)
&+&\frac1b\Phi_{N'}\left(S^n_J(\mu^{n,1})-S^n_J(\mu_J^{n,1-b})\right)\\
&\ge&- \frac 1{2(s-r)}\left[W^n_s(\mu^{n,0},\mu^{n,1})^2-W^n_r(\mu^{n,0},\mu^{n,1})^2\right]
\\
&&-  a \cdot W^n_{\lambda,J}(\mu^{n,0},\mu^{n,1})^2
+\frac1{N'} \left| S^n_J(\mu^{n,0})-S^n_J(\mu^{n,1})\right|^2\nonumber
\end{eqnarray}
   for all $b\le a$.
   Our goal for the remaining part of the proof
   is to demonstrate that this inequality is preserved if we switch from $(\mu_t^{n,b})$
  to  $(\tilde \mu_t^{n,b})$ and also if we then pass to the limit $n\to\infty$.

\medskip

(v) For each $n$ and $t$, we know that the static space $(X^n,d_t^n, e^{-f_t^n}m^n)$ satisfies a synthetic lower Ricci bound
CD$(-K,\infty)$. Thus following the argumentation in \cite{Raj}, Thm.\ 1.3, the interpolating geodesics $(\mu_t^{n,b})_{b\in[0,1]}$ can be chosen with uniformly bounded densities. That is, such that
$$\rho^{n,b}_t\le C''$$
with $C''= C'\cdot \exp(K\, L^2/12)$ where $L$ is the uniform upper bound for the diameter of all the spaces under consideration. 

This bound on the densities is preserved by the transport map if we again transport these probability measures from $X^n$ to $X$. We define
 $\tilde\mu_t^{n,b}=\tilde\rho_t^{n,b}\cdot m$ for all $b\in [0,1]$ (previously we defined this only for the endpoints) with
$\tilde\rho_t^{n,b}(x)=\int_{X^n} \rho_t^{n,b}(y) \, q^n (x,dy)$. Then for all $b\in[0,1]$
\begin{eqnarray*}
\left| \int_J \int_X f_td\tilde\mu^{n,b}dt-  \int_J \int_{X^n} f^n_td\mu^{n,b}dt\right|
&=&\left|
 \int_J \int_X \int_{X^n} \left[ f_t(x)-  f^n_t(y)\right]  \rho_t^{n,b}(y) q^n(dx,dy)\,dt\right|\\
 &\le & C''
 \int_J \int_X \int_{X^n} \left| f_t(x)-  f^n_t(y)\right| q^n(dx,dy)\,dt\\
 &\le & C''\cdot |I|\cdot {\mathbb D}_I\Big( 
 (X^n,d^n_t,f^n_t,m^n), (X,d_t,f_t,m)\Big).
\end{eqnarray*}
Using again the argument that transport maps do not increase the entropy we get
$$S^n_\diamond(\mu^{n,b}_t)\ge S_\diamond(\tilde\mu^{n,b}_t)$$
for all $b\in[0,1]$ and thus (together with the previous convergence result for the weights)
\begin{equation}\label{+1}
\liminf_{n\to\infty}S^n_t(\mu^{n,b}_t)\ge\liminf_{n\to\infty} S_t(\tilde\mu^{n,b}_t)
\end{equation}
for all $t\in J$ as well as
$$
\liminf_{n\to\infty}S^n_J(\mu^{n,b}_J)\ge\liminf_{n\to\infty} S_J(\tilde\mu^{n,b}_J).$$

\medskip

(vi)
Due to the compactness of $X$, there exist measures $\tilde\mu_t^b$ such that (after passing to a suitable subsequence)
$\tilde\mu_t^{n,b} \to\tilde\mu_t^b$
as $n\to\infty$ for all $b\in [0,1]$. 
Following the argumentation in \cite{St1}, Thm.\ 4.20, one verifies that $(\tilde\mu^b_t)_{b\in[0,1]}$ is a $W_t$-geodesic of probability measures on $X$ with $\tilde\mu^0_t=\mu^0$ and $\tilde\mu^1_t=\mu^1$.
Let us briefly illustrate the basic idea in the case $b=1/2$. The family $(\mu^{n,1/2}_t)_{t\in J}$ is a midpoint of the 
families
$(\mu^{n,0}_t)_{t\in J}$  and $(\mu^{n,1}_t)_{t\in J}$ w.r.t.\ the distance $W^n_J$, i.e.\ 
$$\int_J W_t^n\big( \mu^{n,0}_t, \mu^{n,1/2}_t\big)^2dt+
\int_J W_t^n\big( \mu^{n,1/2}_t, \mu^{n,1}_t\big)^2dt=\frac12
\int_J W_t^n\big( \mu^{n,0}_t, \mu^{n,1}_t\big)^2dt.$$
For sufficiently large $n$, the families  $(\mu^{n,1/2}_t)_{t\in J}$ and
 $(\tilde\mu^{n,1/2}_t)_{t\in J}$ are arbitrarily close in $\hat W_J^n$, the $L^2$-Wasserstein distance on an ambient space $\hat X^n$ in which $X$ and $X^n$ are isometrically embedded in an optimal way.
 Similarly, the family $(\tilde\mu^{n,0}_t)_{t\in J}$ and the constant family $(\mu^0)_{t\in J}$ are arbitrarily close w.r.t.\ $\hat W_J^n$
 and so are the families  $(\tilde\mu^{n,1}_t)_{t\in J}$ and  $(\mu^1)_{t\in J}$ .
Hence, given $\epsilon>0$, for sufficiently large $n$ the family
$(\tilde\mu^{n,1/2}_t)_{t\in J}$ is an $\epsilon$-midpoint of  $(\mu^0)_{t\in J}$ and  $(\mu^1)_{t\in J}$ w.r.t.\ $W_J$.
Passing to the limit $n\to\infty$ this proves that
$(\tilde\mu^{1/2}_t)_{t\in J}$ is  a midpoint of  $(\mu^0)_{t\in J}$ and  $(\mu^1)_{t\in J}$ w.r.t.\ $W_J$.
This in turn implies that for a.e.\ $t\in J$ the measure $\tilde\mu^{1/2}_t$ is  a midpoint of  $\mu^0$ and  $\mu^1$ w.r.t.\ $W_t$.

\medskip

(vii) Lower semicontinuity of $S_\diamond$ implies
$$\liminf_{n\to\infty}S_\diamond(\tilde\mu_t^{n,b}) \ge S_\diamond(\tilde\mu_t^b)$$
and lower semicontinuity of $f_t$ (in $x$) implies
$$\liminf_{n\to\infty}\int_X f_t d\tilde\mu^{n,b} \ge\int_X f_t d \tilde\mu^b.$$
Thus
\begin{equation}\label{+2}
\liminf_{n\to\infty}S_t(\tilde\mu_t^{n,b}) \ge S_t(\tilde\mu_t^b)
\end{equation}
for all $t\in J$   and every $b\in [0,1]$ as well as 
   $$\liminf_{n\to\infty} S_J(\tilde\mu_J^{n,b}) \ge S_J(\tilde\mu_J^{b}).$$

Finally, note that 
$u\mapsto \Phi_{N'}(u):=u+\frac1{N'} u^2$ is increasing in $u\in [-\frac {N'}2,\infty)$
and recall from Lemma \ref{mono-conv} that inequality \eqref{cons-til} implies
$$S_J(\mu^0) -
S_J(\tilde\mu_J^{n,b})\ge -\frac {N'}2$$
for $b\le a$  and $n$ sufficiently large.
Hence,
 $$\limsup_{n\to\infty}\Phi_{N'}( S_J(\mu^0)-
 S_J(\tilde\mu_J^{n,b})) \le \Phi_{N'}(S_J(\mu^0)-S_J(\tilde\mu_J^b)
 )$$
for all  $b\le a$. 
Similarly,
 $$\limsup_{n\to\infty}\Phi_{N'}( S_J(\mu^1)-
 S_J(\tilde\mu_J^{n,1-b})) \le \Phi_{N'}(S_J(\mu^1)-S_J(\tilde\mu_J^{1-b})
 ).$$
This proves the claim.
\end{proof}

Given $K,L,\Phi$ and $\lambda$ as before, let 
${\mathbb X}_I(K, L,\Phi,\lambda)$ denote the subspace of ${\mathbb X}_I(K,L,\Phi)$ consisting of equivalence classes of time-dependent mm-spaces $(X,d_t,f_t,m)_{t\in I}$ 
 which in addition satisfy the upper log-Lipschitz bound   \eqref{low-lip} and for which a.e.\ of the static spaces 
$(X,d_t,m_t)$ is infinitesimally Hilbertian.

\begin{corollary}
 For each $N\in[1,\infty]$, the class of  super-$N$-Ricci flows within ${\mathbb X}_I(K, L,\Phi,\lambda)$ is closed w.r.t.\ ${\mathbb D}_I$-convergence.
 \end{corollary}
 That is, `uniformly bounded' ${\mathbb D}_I$-limits of  super-$N$-Ricci flows are   super-$N$-Ricci flows.

\begin{proof}
Within the class ${\mathbb X}_I(K, L,\Phi,\lambda)$, let  a sequence of time-dependent metric measure spaces
 $(X^n,d_t^n,f^n_t,m^n)_{t\in I}$ be given which ${\mathbb D}_I$-converges to 
  a time-dependent mm-space
 $(X,d_t,f_t,m)_{t\in I}$.
 Assume that all the $(X^n,d_t^n,f^n_t,m^n)_{t\in I}$ are super-$N$-Ricci flows. According to Corollary \ref{2.13} then they are also $\lambda$-moderate super-$N$-Ricci flows
 and in turn by Proposition \ref{aver-ae-point} they are $\lambda$-averaged super-$N$-Ricci flows. Thus by the previous Theorem also the limit space  $(X,d_t,f_t,m)_{t\in I}$ is a $\lambda$-averaged super-$N$-Ricci flow. Again by Proposition \ref{aver-ae-point} together with Corollary \ref{2.13}, it is then also a  $\lambda$-moderate super-$N$-Ricci flow as well as a super-$N$-Ricci flow.
\end{proof}

\subsection{The Compactness Result for Super-Ricci Flows}

 Given an interval $I$, constants $K,L\in(0,\infty)$, and a modulus of continuity $\Phi$, let ${\mathbb X}^\diamond_I(K,L,\Phi)$ denote the subspace of ${\mathbb X}_I(K,L,\Phi)$ consisting of equivalence classes of time-dependent mm-spaces $(X,d_t,f_t,m)_{t\in I}$ 
 which in addition
 satisfy
 \begin{equation}\label{cc-lip} d_t(x,y)/d_s(x,y) \le L,\qquad |f_t(x)|\le L,
  \end{equation} 
 \begin{equation} |f_t(x)-f_t(y)|\le \Phi\big(d_t(x,y)\big),\qquad |f_s(x)-f_t(x)|\le \Phi( |s-t|)
  \end{equation} 
  for all $s,t\in I$ and all $x,y\in X$.
 
 \begin{theorem}[Compactness]
 For each $N\in[1,\infty]$ and $\lambda\in L^1_{loc}(I)$, the class of $\lambda$-averaged  super-$N$-Ricci flows within ${\mathbb X}^\diamond_I(K,L,\Phi)$ is compact w.r.t.\ ${\mathbb D}_I$-convergence.
\end{theorem}

 That is, each `uniformly bounded' sequence of  $\lambda$-averaged  super-$N$-Ricci flows 
 has a ${\mathbb D}_I$-converging subsequence and a limit
which again is a $\lambda$-averaged  super-$N$-Ricci flow.

\begin{proof}
Let $T$ denote the midpoint of the interval $I$.
For each time-dependent metric measure space $(X,d_t,f_t,m)_{t\in I}$ under consideration, we  consider the static mm-space $(X,d_T,m)$ as `reference space', equipped with uniformly equicontinuous functions 
 $f:I\times X\to\R$ and $d:I\times X\times X\to\R$.
Indeed, for both functions, continuity in $t$ with modulus $\Phi$ was explicitly requested.  Uniform continuity in $x$ is expressed as follows:
$$|f_t(x)-f_t(y)|\le \Phi(d_t(x,y)\le \Phi(L\cdot d_T(x,y))$$
and
$$|d_t(x,y)-d_t(x'y')|\le d_t(x,x')+d_t(y,y')\le L\cdot [d_T(x,x')+d_T(y,y')].$$

Now let a ${\mathbb D}_I$-Cauchy sequence
$$\Big((X^n,d^n_t,f^n_t,m^n)_{t\in I}\Big)_{n\in\N}$$
of $\lambda$-averaged super-$N$-Ricci flows  in ${\mathbb X}_I(K,L,\Phi)$ be given.
Then according to the previous Lemma, the family $\big(X^n,d^n_T,m^n\big)_{n\in\N}$ of static mm-spaces is  a Cauchy sequence w.r.t.\ ${\mathbb D}$. Moreover, 
each of the static mm-spaces $(X^n,d^n_T,m^n)$ has diameter $\le L$ and satisfies  the curvature-dimension condition CD$(-K,\infty)$. (The passage from  $m_T^n$ to its normalization $m^n$ does not effect this.)
Thus by the compactness result from \cite{St1}, there exists a mm-space $(X,d_T,m)$ such that -- after passing to a suitable subsequence -- 
$$(X^n,d^n_T,m^n)\stackrel{\mathbb D}\longrightarrow (X,d_T,m).$$

\medskip

Now consider on these approximating spaces the sequence of uniformly equicontinuous functions $f^n: I\times X^n\to\R$ (w.r.t.\ to the metrics $d_T^n$). For a suitable subsequence we obtain convergence (along suitable $\epsilon_n$-isometries $\iota_n: X\to X^n$) towards a uniformly continuous functions $f:I\times X\to\R$.
Similarly, given the sequence of uniformly equicontinuous functions $d^n: I\times X^n\times X^n\to\R$. 
 For a suitable subsequence we obtain convergence towards a uniformly continuous function $d:I\times X\times X\to\R$ (`Arzela-Ascoli').
Symmetry and triangle inequality carry over from the approximating functions to the limit.
The fact that $d^n_T(.,.)=d^n(.,.,T)$ is preserved by the limit and yields 
$d_T(.,.)= d(.,.,T)$.

The convergence of $f^n$ to $f$  is uniform. (Indeed, $|f-f_n\circ\iota_n|\le 2\, \Phi(\epsilon)$ on $X$ as soon as $|f-f_n\circ\iota_n|\le  \epsilon$ on an $\epsilon$-net due to the uniform continuity).
Similarly, for the convergence of $d^n$ to $d$.
Thus it leads to convergence in ${\mathbb D}_I$.

It remains to prove that for each $t\in I$ the limit space $(X,d_t,m_t)$ satisfies the CD$(-K,\infty)$-condition.
This is clear for $t=T$ (by the stability result for lower Ricci bounds in the static case), but requires some care for $t\not= T$:
in general, the assertion is not a direct consequence of the stability result for lower Ricci bounds since optimal couplings for the ${\mathbb D}_I$-distance and associated transport kernels are not w.r.t.\ measure $m_t$ but w.r.t.\ $m_T$.

Our argument closely follows ths proof of the previous Theorem \ref{3.3}. Given $\mu^0,\mu^1$ on $X$, we define as before $\mu^{n,0}, \mu^{n,1}$ as their images on $X^n$ under the transport kernels (same as before).
These measures on $X^n$ will now be connected by a geodesic $(\mu_t^{n,b})_{b\in[0,1]}$ such that
$b\mapsto S^n_t(\mu_t^{n,b})$ is $(-K)$-convex.
By means of the same transport kernels as in the previous proof, the geodesic $(\mu_t^{n,b})_{b\in[0,1]}$ will be mapped onto a curve $(\tilde \mu_t^{n,b})_{b\in[0,1]}$ of probability measures on $X$ which, as $n\to\infty$, will converge to a $W_t$-geodesic
$(\tilde \mu_t^{b})_{b\in[0,1]}$.
With the estimates and arguments from the previous proof we conclude that
\begin{equation*}
\liminf_{n\to\infty}S^n_t(\mu^{n,b}_t)\ge S_t(\tilde\mu_t^b)
\end{equation*}
for all $t\in I$ and all $b\in[0,1]$,
see \eqref{+1}, \eqref{+2}.
(Indeed, this  first will be verified for a.e.\ $t$. Due to continuity in $t$ of both $f$ and $d$, it  thus holds for all $t$.)
Moreover,
\begin{equation*}
\lim_{n\to\infty}S^n_t(\mu^{n,b}_t)= S_t(\tilde\mu_t^b)
\end{equation*}
for $b\in\{0,1\}$.
Indeed, the convergence $S^n_\diamond(\mu^{n,b}_t)\to S_\diamond(\tilde\mu_t^b)$
 was already deduced before. The convergence $\int f_t^nd\mu_t^{n,b}\to \int fd\tilde\mu_t^b$ follows from
 \begin{eqnarray*}
 \left| \int_X f_td\tilde\mu^{n,b}dt-   \int_{X^n} f^n_td\mu^{n,b}\right|
&\le&
\left|\frac1{\delta} \int_{t-\delta}^t \int_X f_sd\tilde\mu^{n,b}ds-  \int_J \int_{X^n} f^n_sd\mu^{n,b}ds\right|
+2\Phi(\delta)
\\
 &\le & \frac{C'\, |I|}\delta {\mathbb D}_I\Big( 
 (X^n,d^n_s,f^n_s,m^n)_{s\in I}, (X,d_s,f_s,m)_{s\in I}\Big)+2\Phi(\delta)
 \\
 &= & \Phi_2\Big( |I|\cdot {\mathbb D}_I\Big( 
 (X^n,d^n_s,f^n_s,m^n)_{s\in I}, (X,d_s,f_s,m)_{s\in I}\Big)\Big)
\end{eqnarray*}
if we put $\Phi_2(r)=2\Phi(\sqrt r)+C'\sqrt r$ and choose $\delta=\sqrt{|I|\cdot{\mathbb D}_I(\ldots)}$.
Thus the $(-K)$-convexity of the entropy carries over from 
$(X^n,d^n_t,m_t^n)$
to $(X,d_t,m_t)$. This proves the claim.
\end{proof}

 Finally, let us introduce the  space  ${\mathbb X}_I^\diamond(K,L,\Phi,\lambda)$  of equivalence classes of time-dependent mm-spaces $(X,d_t,f_t,m)_{t\in I}$ in
 ${\mathbb X}^\diamond_I(K,L,\Phi)$
 which in addition satisfy the upper log-Lipschitz bound   \eqref{low-lip} and for which a.e.\ of the static spaces 
$(X,d_t,m_t)$ is infinitesimally Hilbertian.

 \begin{corollary}
 For each $N\in[1,\infty]$, the class of  super-$N$-Ricci flows within ${\mathbb X}^\diamond_I(K,L,\Phi,\lambda)$ is compact w.r.t.\ ${\mathbb D}_I$-convergence.
 \end{corollary}

 That is, each `uniformly bounded' sequence of   super-$N$-Ricci flows 
 has a ${\mathbb D}_I$-converging subsequence and a limit
which again is a   super-$N$-Ricci flow.

\begin{proof}
According to Corollary \ref{2.13} and Proposition \ref{aver-ae-point}, every 
super-$N$-Ricci flow within the given class  ${\mathbb X}^\diamond_I(K,L,\Phi,\lambda)$ is also a 
$\lambda$-averaged super-$N$-Ricci flow. Thus by the previous Compactness Theorem,
for each sequence 
$$\big((X^n,d^n_t,m_t^n)_{t\in I}\big)_{n\in\N}$$
of super-$N$-Ricci flow within this class
we get the existence of a converging subsequence towards a limit $(X,d_t,m_t)_{t\in I}$ which is a $\lambda$-averaged super-$N$-Ricci flow. Again by Proposition \ref{aver-ae-point} and  Corollary \ref{2.13}, this limit is indeed also a super-$N$-Ricci flow.

The upper log-Lipschitz bound   \eqref{low-lip} 
obviously carries over to the limit space. Moreover, we know that for each $n\in \N$ the reference space $(X^n,d^n_T,m^n)$ satisfies a RCD$(K,\infty)$-condition and indeed these spaces ${\mathbb D}$-convergence to $(X,d_T,m)$.
The RCD$(K,\infty)$-condition is preserved by the convergence \cite{AGS-Mms}.
Thus, in particular, $X$ is infinitesimally Hilbertian.
\end{proof}

\section{Appendix: Synthetic Upper Bounds for Ricci Curvature of MM-Spaces}

In this chapter, we  present a 
synthetic definition of generalized upper bounds for the Ricci curvature of mm-spaces. Before doing so, let us recall that more than a decade ago, the author \cite{St1,St2} -- and independently Lott \& Villani \cite{LV1} --  introduced a
 synthetic definition of generalized lower bounds for the Ricci curvature of mm-spaces. These definitions of lower Ricci bounds initiated a wave of investigations on analysis and geometry of mm-spaces and led to many new insights and results. Until now, however, no similar concept for upper bounds on the Ricci curvature was available. Surprisingly enough, such a definition can be given easily and very much in the same spirit as for the lower bounds.

For a completely different approach to \emph{both sided} bounds on the Ricci curvature -- also applicable to mm-spaces -- we refer to recent work of Naber \cite{Na}.

\begin{definition} We say that a
 mm-space $\big(X,d,m\big)$ has  \emph{Ricci curvature $\le K$} if   for any $K'>K$ there exists an open covering $X=\bigcup_i X_i$ of $X$ such that for each $i$ and any pair of nonempty open sets $U^0,U^1\subset X_i$    there exists
  a geodesic $(\mu^\tau)_{\tau\in[0,1]}$ with
   $\mathrm{supp}(\mu^\sigma)\subset U^\sigma$ for $\sigma=0,1$ such that
  $\tau\mapsto S(\mu^\tau)$ is   finite and $K'$-concave on $(0,1)$. That is, for all 
 $\rho,\tau,\sigma\in(0,1)$ with $\rho<\tau<\sigma$
\begin{equation}\label{concave}
S(\mu^\tau)\ge \frac{\sigma-\tau}{\sigma-\rho}S(\mu^\rho)+\frac{\tau-\rho}{\sigma-\rho}S(\mu^\sigma)-\frac{K'}2\frac{(\tau-\rho)(\sigma-\tau)}{(\sigma-\rho)^2}\cdot
W^2(\mu^\rho,\mu^\sigma).
\end{equation}
\end{definition}

Recall that $K'$-concavity of $\tau\mapsto S(\mu^\tau)$ is equivalent to absolute continuity and validity of the distributional inequality $\partial_\tau^2 S(\mu^\tau)\le K' \cdot W^2(\mu^0,\mu^1)$.
Moreover, it is equivalent to absolute continuity and the fact that 
$$\partial^-_\tau S(\mu^\tau)\big|_{\tau=\sigma-}-
\partial^+_\tau S(\mu^\tau)\big|_{\tau=\rho+}\le \frac{K'}{\sigma-\rho} \cdot W^2(\mu^\rho,\mu^\sigma)$$
for all $0<\rho<\sigma<1$.

The previous definition also allows for a straightforward extension (in the spirit of the variable lower bounds from \cite{St3}) to variable upper bounds $k$ for the Ricci curvature on $X$. 

\begin{remark} Given a continuous function $k:X\to \R$, we say 
 that the
 mm-space $\big(X,d,m\big)$ has  \emph{Ricci curvature $\le k$} if   for any continuous function $k'>k$ there exists an open covering $X=\bigcup_i X_i$ of $X$ such that for each $i$ and any pair of nonempty open sets $U^0,U^1\subset X_i$    there exists
  a geodesic $(\mu^\tau)_{\tau\in[0,1]}$ with
   $\mathrm{supp}(\mu^\sigma)\subset U^\sigma$ for $\sigma=0,1$ and a probability measure $\Theta\in \Gamma(X)$ such that $\mu^\tau = (e^\tau)_\sharp\Theta$  such that
 for all 
 $\rho,\tau,\sigma\in(0,1)$ with $\rho<\tau<\sigma$
\begin{equation}\label{concave-variable}
\infty>S(\mu^\tau)\ge \frac{\sigma-\tau}{\sigma-\rho}S(\mu^\rho)+\frac{\tau-\rho}{\sigma-\rho}S(\mu^\sigma)-{(\sigma-\rho)}\int_\rho^\sigma \chi^{\frac{\tau-\rho}{\sigma-\rho},\frac{a-\rho}{\sigma-\rho}
}k'(\gamma^a)\,| \dot\gamma|^2 da\,d \Theta(\gamma).
\end{equation}
 Here as usual   $\Gamma(X)$ denotes the space of all geodesics $\gamma:[0,1]\to X$
and   $e^\tau$ is the evaluation map $\gamma\mapsto\gamma^\tau$.
\end{remark}

\begin{theorem}\label{thm-riem-upper-ric}\label{upper-Ric}
Given a weighted Riemannian manifold  $(M, g,\tilde f)$ with a smooth
metric tensor $g$ and a smooth (`weight') function $\tilde f$. Let $d$ be the induced Riemannian distance, let $\mathrm{vol}$ be the induced Riemannian volume and put $m=e^{-\tilde f}\mathrm{vol}$. Then the   mm-space $(X,d,m)$ has Ricci curvature $\le K$ if and only if 
\begin{equation}\label{riem-sub-rf}
\Ric_{g}+\Hess_{g}\tilde f\le K\cdot g
\end{equation}
in the sense of inequalities between quadratic forms on the tangent bundle of $(M,g)$.
\end{theorem}

\begin{proof}
Assume that \eqref{riem-sub-rf} holds true and that $K'=K+\varepsilon$ for some $\varepsilon>0$. Let $M=\bigcup_i M_i$ be an open covering of $M$ by convex domains $M_i$ with sufficiently small diameter (where `sufficiently' will depend on bounds on the sectional curvature on $M_i$ and on $\varepsilon$).

   Given open sets $U^0,U^1\subset M_i$,
   choose a geodesic $\gamma$ in $M_i$ with endpoints $\gamma^\sigma\in U^\sigma$ for $\sigma=0,1$ and a smooth function $\varphi:M\to\R$ with
   $$\nabla\varphi\big|_{\gamma^0}=\dot\gamma^0,\qquad
   \Hess\,\varphi\big|_{\gamma^0}=0.$$
  
Put $F^\tau(x)=\exp_x(-\tau\nabla\varphi)$ and consider the geodesic in $\Pz$ induced by 
$\mu^\tau=(F^\tau)_\sharp\mu^0$ with bounded densities $u^\tau$ w.r.t.\ $m$ where $\mu^0$ is chosen to be concentrated in a small neighborhood of $\gamma^0$. For sufficiently small $M_i$ this indeed will be a geodesic and each $\mu^\tau$ will be supported in a small neighborhood of $\gamma^\tau$. Moreover, following \cite{St0,St1} put
$$\mathcal A^\tau(x)=dF^\tau(x), \quad
\mathcal U^\tau=\nabla^\tau\mathcal A^\tau\circ(\mathcal A^\tau)^{-1},\quad
y_t=\log\det \mathcal A^\tau.$$
It is well-known (e.g. \cite{St0}, equation (4.16)) that convexity and concavity properties of $\tau\mapsto\Ent(\mu^\tau|m)$ can be expressed in terms the above quantities due to the fact that
\begin{eqnarray}
\partial_\tau^2\Ent(\mu^\tau|m)&=&
\partial_\tau^2\left[\int u^0\log u^0\,dm-\int y^\tau\,u^0\,dm+\int\tilde f(F^\tau)u^0\,dm\right]\nonumber\\
&=&\int\left[ -\partial^2_\tau y^\tau+ \Hess\,\tilde f (\dot F^\tau,\dot F^\tau)
\right]u^0\,dm.
\label{ent-ableitung}
\end{eqnarray}
Since by construction $\mathcal U^0=0$, comparison geometry (`Rauch' and `Bishop-G\"unter') implies that $\mathcal U^\tau$ is `small' along $\gamma$ (in any given norm) provided the curve $\gamma$ is `short'. (The precise quantitative estimate will depend on bounds for the sectional curvature along $\gamma$.) 
Thus
\begin{eqnarray}
 -\partial^2_\tau y^\tau=
 -\mathrm{tr} (\nabla^\tau\mathcal U^\tau)
 &=&\mathrm{tr} \left((\mathcal U^\tau)^2\right)+\Ric(\dot F^\tau,\dot F^\tau)\nonumber\\
 &\le&\varepsilon\cdot |\dot F^\tau|^2+\Ric(\dot F^\tau,\dot F^\tau).
 \label{ent-ableitung-2}
\end{eqnarray}
By continuity, this estimate will extend to a suitable neighborhood of $\gamma$ and without restriction, we may assume that the curve $\tau\mapsto \mu^\tau$ is supported by such a small neighborhood of $\gamma$.
 Thus combining \eqref{ent-ableitung} and \eqref{ent-ableitung-2} then yields
      \begin{eqnarray*}
\partial_\tau^2\Ent(\mu^\tau|m)&=&\int\left[ -\partial^2_\tau y^\tau+ \Hess\,\tilde f (\dot F^\tau,\dot F^\tau)
\right]u^0\,dm\\
&\le&\int\left[
\varepsilon\cdot |\dot F^\tau|^2+\big(\Ric+ \Hess\,\tilde f \big)(\dot F^\tau,\dot F^\tau)\right]u^0\,dm\\
&\le&K'\cdot \int |\dot F^\tau|^2u^0\,dm=K'\cdot W^2(\mu^0,\mu^1)
\end{eqnarray*}
       which is the `if' implication of the claim.

The `only if'-implication is easy. Assume that \eqref{riem-sub-rf} is not satisfied. Then for some $\varepsilon>0$ in a suitable neighborhood of some $(s,v)\in I\times TM$ the inequality 
\begin{equation}
\Ric_{g}+\Hess_{g}\tilde f_t\ge (K+2\epsilon) \cdot g
\end{equation}
holds true. Thus all transports with velocity fields in this neighborhood will lead to  $(K+2\varepsilon)$-convexity of the entropy, being in contrast to the requested  $(K+\varepsilon)$-concavity. 
\end{proof}

\section{Appendix: Super-Ricci Flows for Diffusion Operators}

Let us now present another, completely different approach to super-Ricci flows. It is formulated in terms of the so-called $\Gamma$-calculus of Bakry, \'Emery and Ledoux \cite{BE,Le, BGL}.

\subsection{Time-dependent Diffusion Operators}

Throughout this chapter, we assume that we are given a 1-parameter family  $(\LL_t)_{t\in [0,T)}$  of  linear operators defined on an algebra $\A$ of functions on a set $X$ such that ${\LL_t}(\A)\subset\A$ for each $t$.
(The fact that we choose the time interval to start at 0 is just for convenience; here and in the sequel, we could choose any other time interval $[S,T]$.)
We assume that we are given a topology on $\A$ such that limits and derivatives make sense.
In terms of these data we define the square field operators
$\Gamma_t(f,g)=\frac12[{\LL_t}(fg)-f{\LL_t}g-f{\LL_t}g]$.
We assume that  ${\LL_t}$ is a \emph{diffusion operator}  in the sense that
\begin{itemize}
\item $\Gamma_t(u,u)\ge0$ for all $u\in\A$
\item $\psi(u_1,\ldots,u_k)\in\A$ for every $k$-tuple of functions $u_1,\ldots,u_k$ in $\A$ and every $C^\infty$-function $\psi:\R^k\to\R$ vanishing at the origin and
$
{\LL}_t\psi(u_1,\ldots,u_k)=\sum_{i=1}^k \psi_i(u_1,\ldots,u_k)\cdot {\LL}_tu_i + \sum_{i,j=1}^k \psi_{ij}(u_1,\ldots,u_k)\cdot \Gamma_t(u_i,u_j)
$
where $\psi_i:=\frac{\partial}{\partial y_i}\psi$ and $\psi_{ij}:=\frac{\partial^2}{\partial y_i\,\partial y_j}\psi$.
\end{itemize}
The Hessian of $u$ at time $t$ and a point $x\in X$ is the bilinear form on $\A$ given by
$${\HH}_tu(v,w)(x)=\frac12\Big[\Gamma_t\big(v,\Gamma_t(u,w)\big)+\Gamma_t\big(w,\Gamma_t(u,v)\big)-\Gamma_t\big(u,\Gamma_t(v,w)\big)\Big](x)$$
for $u,v,w\in\A$.
Similarly, 
 the $\Gamma_2$-operator is defined via iteration of the square field operator as
$$\Gamma_{2,t}(u,v)(x)=\frac12[{\LL_t}\Gamma_t(u,v)-\Gamma_t(u,{\LL_t}v)(x)-\Gamma_t(u,{\LL_t}v)](x).$$
 We put $\Gamma_t(u)=\Gamma_t(u,u)$ and $ \Gamma_{2,t}(u)=\Gamma_{2,t}(u,u)$.
In terms of the $\Gamma_2$-operator we define the {Ricci tensor} at the space-time point $(t,x)\in [0,T)\times X$ by
\begin{eqnarray*}
\RR_t(u)(x)=\inf\big\{ \Gamma_{2,t}(u+v)(x): \ v \in\A_x^0\big\}
\end{eqnarray*}
for $u\in\A$ where
$$\A_x^0=\{v=\psi(v_1,\ldots,v_k): \ k\in\N, v_1,\ldots,v_k\in\A, \psi \mbox{ smooth with } \psi_i(v_1,\ldots,v_k)(x)=0 \  \forall i\}.$$
Note that the definition  of $\RR$ here slightly differs from that in \cite{St4}.

We can  always extend the definition of ${\LL_t}$ and $\Gamma_t$ to the algebra  generated by the elements in $\A$ and the constant functions which leads to ${\LL_t}1=0$ and $\Gamma_t(1,f)=0$ for all $f$.

\medskip

For the sequel we assume in addition that we are given a 2-parameter family $(P^s_{t})_{0\le s\le t<T}$ of linear operators on $\A$ satisfying for all $s\le r\le t$ and all $u\in\A$
\begin{equation*}
P^t_{t}u=u, \quad P^r_{t}(P^s_{r}u)=P^s_{t}u\end{equation*}
\begin{equation*} (P^s_{t}u)^2\le P^s_{t}(u^2)\end{equation*}
\begin{equation*} s\mapsto P^s_{t}u \mbox{ and }t\mapsto P^s_{t}u \mbox{ continuous}
\end{equation*}
\begin{equation}\label{eins}
\partial_s P^s_{t}u=-P^s_{t}(\LL_s u)
\end{equation}
\begin{equation}\label{zwei}
\partial_t P^s_{t}u=\LL_t P^s_{t} u.
\end{equation}
Such a \emph{propagator} $(P^s_{t})$ for the given family of operators  $(\LL_t)$ will exist in quite general situations under mild assumptions. We also require that for each 1-parameter family $(u_r)_{r\in(s',t')}$ which is differentiable within $\A$ w.r.t.\ $r$
\begin{equation}\label{drei}
\partial_r P^s_{t}u_r= P^s_{t} (\partial_r u_r), \quad \partial_r \Gamma_t(u_r,v)= \Gamma_t(\partial_r u_r,v)
\end{equation}

\subsection{Super-Ricci Flows}

\begin{definition}
We say that  $(\LL_t)_{t\in [0,T)}$  is a \emph{super-Ricci flow} if
\begin{equation}\label{vier}
\partial_t \Gamma_t\le 2\RR_{t}.
\end{equation}
It is called Ricci flow if $\partial_t \Gamma_t= 2\RR_{t}$.
\end{definition}

\begin{lemma}\label{gamma-ric}  $(\LL_t)_{t\in [0,T)}$  is a super-Ricci flow if and only if
\begin{equation}\label{fuenf}
\partial_t \Gamma_t\le 2\Gamma_{2,t}.
\end{equation}
\end{lemma}

\begin{proof} Obviously  (\ref{vier}) implies (\ref{fuenf}) since $\Gamma_{2,t}\ge\RR_t$. For the converse, fix $t,x$, $u$, $\epsilon>0$ and choose suitable $v=\psi(v_1,\ldots,v_k)\in\A_x^0$ such that
$\RR_t(u)(x)\ge \Gamma_{2,t}(u+v)(x)-\epsilon$.
The fact that $v\in\A_x^0$ implies by chain rule that $\Gamma_s(u+v)(x)=\Gamma_s(u)(x)$ for all $s$. 
Thus in particular
$$\partial_t\Gamma_t(u+v)(x)=\partial_t\Gamma_t(u)(x).$$
This proves the claim. 
\end{proof}

\begin{remark}
 $(\LL_t)_{t\in [0,T)}$  is a Ricci flow if and only if in addition to \eqref{fuenf} for each $x$, each $\epsilon>0$ and each $u\in\A$ there exists  $v\in\A_x^0$ such that
\begin{equation}
\partial_t \Gamma_t(u)(x)+\epsilon\ge 2\Gamma_{2,t}(u+v)(x).
\end{equation}
\end{remark}

\begin{corollary}\label{ric-gam}
Let  $X=M$ be  a  manifold $M$ with a smooth 1-parameter family of Riemannian metrics  $(g_t)_{t\in[0,T]}$ and let $(L_t)_{t\in[0,T]}$ be the associated Laplace-Beltrami operators acting on the algebra $\A$ of smooth functions which vanish at infinity. Then for each $x\in X$ the following are equivalent:
\begin{itemize}
\item[(i)]
$-\partial_t g_t\le 2\Ric_{t}$
\item[(ii)]
$\partial_t \Gamma_t\le 2\RR_{t}$
\item[(iii)]
$\partial_t \Gamma_t\le 2\Gamma_{2,t}$.
\end{itemize}
\end{corollary}

\begin{proof}
It remains to prove the equivalence of (i) and (ii). For this, note that
$g_t(\nabla_tu,\nabla_tu)=g^*_t(Du,Du)=\Gamma_t(u,u)$ and thus 
$$\partial_t g_t(\nabla_su,\nabla_su)\Big|_{s=t}=-\partial_t \Gamma_t(u,u).$$
Moreover, note that according to \cite{St4}, $\Ric_{t}(\nabla_t u,\nabla_t u)=\RR_{t}(u,u)$ for each $u\in \A$.
\end{proof}

\begin{theorem}\label{ric-heat} Under appropriate regularity assumptions on $(P^s_t)_{s\le t}$, the following are equivalent
\begin{itemize}
\item[(i)] $\partial_t \Gamma_t(u)\le 2\Gamma_{2,t}(u)\qquad (\forall u\in\A)$
\item[(ii)]
$\Gamma_t(P^s_{t}u)\le P^s_{t}(\Gamma_s(u))\qquad (\forall u\in\A)$
\end{itemize}
\end{theorem}

\begin{proof}
For $r\in [s,t]$ consider
the function $q_r:=P^r_{t}\Gamma_r(P^s_{r}u)$. Differentiating w.r.t.\ $r$ yields, due to properties (\ref{eins}), (\ref{zwei}) and (\ref{drei}),
\begin{eqnarray}
\partial_r q_r&=&
P^r_{t}\Big(-\LL_r\Gamma_r(P^s_{r}u)+(\partial_r \Gamma_r)(P^s_{r}u)+2\Gamma_r(\partial_r P^s_{r}u, P^s_{r}u)\Big)\nonumber\\
&=&
P^r_{t}\Big(-\LL_r\Gamma_r(v)+\partial_r \Gamma_r(v)+2\Gamma_r(\LL_r v,v)\Big)\nonumber\\
&=&
P^r_{t}\Big(-2\Gamma_{2,r}(v)+\partial_r \Gamma_r(v)\Big)\label{forw-backw}
\end{eqnarray}
where we have put $v= P^s_{r}u$.
Thus assertion (i) implies $\partial_r q_r\le0$ for all $r\in [s,t]$ which in turn yields $q_s\ge q_t$. This is assertion (ii).

Conversely, assertion (ii) implies 
$$0\ge \frac1{t-s}(q_t-q_s)=\frac1{t-s}\int_s^t \partial_r q_r\,dr$$
and in the limit $s\to t$ thus $\partial_t q_t\le0$. According to the previous calculation, however,
$$\partial_t q_t=-2\Gamma_{2,t}(u)+\partial_t \Gamma_t(u).$$
This proves (i)
\end{proof}

\begin{remark}
The proof shows  also the equivalence of 
\begin{itemize}
\item[(i')] $\partial_t \Gamma_t(u)\ge 2\Gamma_{2,t}(u)\qquad (\forall u\in\A)$
\item[(ii')]
$\Gamma_t(P_{s,t}u)\ge P_{s,t}(\Gamma_s(u))\qquad (\forall u\in\A)$
\end{itemize}
as well as the equivalence of the two statements with "$=$" instead of "$\ge$".
Note, however, that in the Riemannian case the assertions  $\partial_t \Gamma_t(u)\ge 2\Gamma_{2,t}(u)$
and  $\partial_t \Gamma_t(u)= 2\Gamma_{2,t}(u)$ are no proper differential equations on the tangent bundle since the RHS involves not only the vector fields $\nabla u$ but also higher derivatives of $u$. In contrast to that, (i) of Theorem \ref{ric-heat} allows a reformulation (see Corollary \ref{ric-gam}) as an inequality on the tangent bundle.
\end{remark}

\subsection{Super-$N$-Ricci Flows}
Given any extended number $N\in[1,\infty]$ we define the {$N$-Ricci tensor} at $(t,x)$ by
\begin{eqnarray*}
\RR_{N,t}(u)(x)=\inf\big\{ \Gamma_{2,t}( u+v)(x)-\frac1N(\LL_t(u+v)^2(x): \  v\in\A_x^0\big\}.
\end{eqnarray*}
(Again recall that the definition  of $\RR_N$ here slightly differs from that in [St14].)

\begin{definition}\label{super-N-gam}  We say that $(\LL_t)_{t\in [0,T)}$  is a super-$N$-Ricci flow if
\begin{equation}
\partial_t \Gamma_t\le 2\RR_{N,t}.
\end{equation}
If equality holds then 
$(\LL_t)_{t\in [0,T)}$  is called $N$-Ricci flow. 
\end{definition}
A Ricci flow is a $N$-Ricci flow for the particular choice $N=\infty$, i.e.\ a solution to $
\partial_t \Gamma_t= 2\RR_{t}$.

\begin{theorem}  Under appropriate regularity assumptions on $(P^s_t)_{s\le t}$, the following are equivalent
\begin{itemize}
\item[(i)] $\partial_t \Gamma_t(u)\le 2\RR_{N,t}(u)\qquad (\forall u\in\A, \forall t)$
\item[(ii)] $\partial_t \Gamma_t(u)\le 2\Gamma_{2,t}(u)-\frac2N (\LL_t u)^2\qquad (\forall u\in\A, \forall t)$
\item[(iii)]
$\Gamma_t(P^s_{t}u)+\frac2N \int_s^t \left(P^r_t \LL_r P^s_{r}u\right)^2\,dr \le P^s_{t}(\Gamma_s(u))\qquad (\forall u\in\A, \forall s\le t)$
\end{itemize}
\end{theorem}

\begin{proof} The equivalence of (i) and (ii) is proven in the same manner as  Lemma \ref{gamma-ric}.
For the implication (ii)$\Rightarrow$(iii) we follow the argumentation of the proof for the previous theorem. Using assertion (ii) we obtain for $q_r:=P^r_{t}\Gamma_r(P^s_{r}u)$ as above from \eqref{forw-backw}
\begin{eqnarray*}
\partial_r q_r&=&
P^r_{t}\Big(-2\Gamma_{2,r}(v)+\partial_r \Gamma_r(v)\Big)\\
&\le& -\frac2N P^r_{t}\Big((\LL_r v)^2\Big)\le -\frac2N \Big(P^r_{t}\LL_r v\Big)^2
\end{eqnarray*}
with $v=P^s_r u$ as before.
Integrating this estimate from $s$  to $t$ yields
$$\Gamma_t(P^s_{t}u)- P^s_{t}(\Gamma_s(u))\le -\frac2N \int_s^t\Big( P^r_t \LL_r P^s_{r}u\Big)^2\,dr$$
which is the claimed assertion (iii). Conversely, differentiating the latter estimate  as in the previous proof yields (ii).
\end{proof}

\begin{remark}
 If we start with the operators $\LL_t$, then soon or later the question will come up to choose measures $m_t$ such that for given $t$ the operator $L_t$ is symmetric (or, more precisely, essentially self-adjoint) on $L^2(X,m_t)$. In non degenerate finite-dimensional cases, such a measure $m_t$ will be uniquely determined (up to multiplicative constants) by the operator $\LL_t$, provided there exists one.
 
 It should be pointed out, however, that the approach via $\Gamma$-calculus also applies to non-reversible operators -- and in this respect is more general than the approach via mm-spaces. 
\end{remark}

\subsection{Starting with the Gradients}
Instead of starting with a 1-parameter family of operators, alternatively we may start with a 1-parameter family of square field operators.
To do so, assume that $X$ is a Polish space with a
 1-parameter family $(m_t)_{t\in[0,T]}$ of measures on it. 
Let $\A$ be an algebra of measurable functions on $X$ such that $\A$ is dense in $L^2(X,m_t)$.
The basic ingredient will be a 1-parameter family $(\Gamma_t)_{t\in[0,T]}$ of 
\begin{itemize}
\item
symmetric, positive semidefinite bilinear forms $\Gamma_t$ on $\A$
\item
each of which has the diffusion property
$$\Gamma_t(\Psi(u_1,\ldots,u_k),v)=\sum_{i=1}^k \Psi_i(u_1,\ldots,u_k)\Gamma_t(u_i,v).$$
\end{itemize}
The operators $\LL_t$ then are defined through the property
$$\int \Gamma_t(u,v)\,dm_t=-\int \LL_tu \, v\, dm_t\qquad(\forall u,v\in\A).$$
We say that the family $(\Gamma_t,m_t)_{t\in(0,T)}$ is a super-Ricci flow if property \eqref{fuenf} is satisfied where $\Gamma_{2,t}$ is defined in terms of $\LL_t$.
 
 \medskip
 
 \begin{remark}
 In cases of applications, the measures typically are given as $m_t=e^{-f_t}m_t^\diamond$ for some $f_t\in \A$ and some 
  (`reference') measures $m_t^\diamond$ on $X$.
Often we will first consider $\Gamma_t$ with respect to the  measure $m_t^\diamond$ and then in a second step consider the effect of the time-dependent weights $f_t$.
 Define $\LL_t^\diamond$ as an operator on $\A$ by
 $$\int \LL^\diamond_tu \, v\, dm_t^\diamond=-\int_X \Gamma_t(u,v)\,dm_t^\diamond\qquad (\forall u,v\in \A)$$
 and define similarly $L_t$ by replacing all $m_t^\diamond$ by $m_t$.
 Then $\LL_t=\LL_t^\diamond-\Gamma_t(.,f_t)$ and thus 
$$\Gamma_{2,t}=\Gamma_{2,t}^\diamond+\HH_tf_t, \qquad \RR_t=\RR_t^\diamond+\HH_tf_t.$$
In particular, the family $(L_t)_{t\in(0,T)}$ defined by the family
$(\Gamma_t,f_t)_{t\in(0,T)}$ is a super-Ricci flow if and only if
$$\partial_t \Gamma_t\le \Gamma_{2,t}^\diamond+\HH_tf_t$$
which imposes no restriction on the evolution of the weights $f_t$. Each family of weight functions  $(f_t)_{t\in(0,T)}$ provides a differential inequality for  square field operators.
\end{remark}

\begin{example}
(i) Let $\psi: \R_+\to \R_+$ be any smooth positive function and $z: \R_+\to \R^n$ an \emph{arbitrary} curve in $\R^n$.
Put $$f_t(x)=\dot\psi_t\cdot \|x-z_t\|^2/2,\qquad \Gamma_t(u)={\psi_t}\cdot\|\nabla u\|^2$$
and let $m_t^\diamond$ (for all $t$) be the Lebesgue measure in $\R^n$.
Then 
$(L_t)_{t\in(0,T)}$ is a super-Ricci flow. Indeed, 
$$\partial_t \Gamma_t= \Gamma_{2,t}^\diamond+\HH_tf_t.$$

(ii) 
Let us consider a time-dependent  weighted, $n$-dimensional  Riemannian space $(X, g_t)$ with measures $m_t=e^{-f_t}vol_{t}$ and operators $L_t=\Delta_t-\nabla_t f_t \nabla_t$. Being a super-$N$-Ricci flow for $N=n$ implies that $f_t$ is constant in $x$ (without restriction $\equiv 1$).
\end{example}

\subsection{Outlook}
In a forthcoming paper \cite{KoS}, accompanying this here, we will define and analyze the heat propagator 
associated with the canonical family of Laplacians $(L_t)_{t\in [0,T,)}$ on time-dependent metric measure spaces $(X,d_t,m_t)_{t\in [0,T,)}$. It will turn out that in great generality the latter is a super-$N$-Ricci flow in the sense of Definition \ref{super-N-mms} if and only if $(L_t)_{t\in [0,T,)}$ is a super-$N$-Ricci flow in the sense of Definition \ref{super-N-gam}.

Moreover, these properties will be proven to be equivalent to the gradient estimate (iii) in Theorem 5.7 for the forward heat flow $P^s_t$ (acting on functions) as well as to certain transport inequalities for the dual flow $\hat P^s_t$ (acting on probability measures). In the case $N=\infty$, the latter will
coincide with the monotonicity property w.r.t.\ time-dependent Wasserstein distances for the dual heat flow  
$$W_s(\hat P^s_t\mu,\hat P^s_t\nu)\le W_t(\mu,\nu)$$
proposed by   McCann/Topping \cite{MT} as a characterization of super-Ricci flows.


\begin{thebibliography}{9}


\bibitem[AGS1]{AGS-Calc}
\textsc{L.Ambrosio, N.Gigli, G.Savar\'{e}},
Calculus and heat flow in metric measure spaces and applications to spaces with {R}icci bounds from below,
Invent.\,Math.\,195: 289-391, 2013

\bibitem[AGS2]{AGS-Mms}
\textsc{L.Ambrosio, N.Gigli, G.Savar\'{e}},
Metric measure spaces with {R}iemannian {R}icci curvature bounded from below,
Duke Math.\,J.\,163: 1405-1490, 2014
 

\bibitem[AGS3]{AGS-BE}
\textsc{L.Ambrosio, N.Gigli, G.Savar\'{e}},
Bakry-\'{E}mery curvature-dimension condition and {R}iemannian {R}icci curvature bounds,
Annals of Probab.\,43: 339-404, 2015




\bibitem[BS]{BaS}
\textsc{K.Bacher, K.T.Sturm},
{Ricci bounds for euclidean and spherical cones}. In `Singular Phenomena and Scaling in Mathematical Models' (ed. M. Griebel), pp 3-23, Springer, 2014


\bibitem[BE]{BE} \textsc{D.Bakry, M.\'Emery}, Diffusions hypercontractives, S\'em.\,prob\,XIX, LNM1123: 177–206,  1985.


\bibitem[BGL]{BGL}
\textsc{D.Bakry, I.Gentil, M.Ledoux},
Analysis and Geometry of Markov Diffusion Operators, Springer Grundlehren, 2014


\bibitem[CZ]{CZ}
\textsc{H-D.Cao; X-P.Zhu} Hamilton-Perelman's proof of the Poincar\'e conjecture and the Geometrization conjecture, Asian J.\,Math.\,10: 165--492, 2006








\bibitem[CMS]{CMS} \textsc{D.Cordero-Erausquin, R.J.McCann, M.Schmuckenschl\"ager, M.}, A Riemannian interpolation inequality a la Borell, Brascamp and Lieb, Invent. Math. 146: 219-257, 2001






\bibitem[EKS]{EKS} 
\textsc{M.Erbar, K.Kuwada, K.T.Sturm},
{On the equivalence of the entropic curvature-dimension condition and Bochner's inequality on metric measure spaces},  Inventiones Math.\, 201: 993-1071, 2015



\bibitem[ErS]{ErS} 
\textsc{M.Erbar, K.T.Sturm},
{Rigidity of Ricci-flat cones},  In preparation.




\bibitem[GiT]{GiT} 
\textsc{G.Giesen,P.Topping}, Ricci flows with unbounded curvature, 
Math.\,Z.\,273: 449-460, 2013



 \bibitem[GiM]{GiM} \textsc{N.Gigli, C.Mantegazza},
A flow tangent to the Ricci flow via heat kernels and mass transport,
Advances in Math.\,250, 74-104, 2014




\bibitem[HaN]{HN}\textsc{R.Haslhofer, A.Naber},
Weak solutions for the Ricci flow I,
arXiv:1504.00911 

 
 

\bibitem[KL1]{KL} \textsc{B.Kleiner, J.Lott},
 Notes on Perelman's papers, Geometry and Topology 12: 2587–2855, 2008


\bibitem[KL2]{KL1} \textsc{B.Kleiner, J.Lott},
Singular Ricci flows I,
arXiv:1408.2271

\bibitem[KoL]{KoL} \textsc{H.Koch, T.Lamm},
Geometric flows with rough initial data, Asian J.\,Math.\,16: 209-235, 2012


\bibitem[Ko1]{Ko1} \textsc{E.Kopfer},
Gradient flow for the Boltzmann entropy and Cheeger's energy on time-dependent metric measure spaces,
arXiv:1611.09522

\bibitem[Ko2]{Ko2} \textsc{E.Kopfer},
Super-Ricci flows and improved gradient and transport estimates,
arXiv:1704.04177


\bibitem[KoS1]{KoS} \textsc{E.Kopfer, K.T.Sturm},
Heat flows on time-dependent metric measure spaces and super-Ricci flows,
arXiv:1611.02570


\bibitem[KoS2]{KoS2} \textsc{E.Kopfer, K.T.Sturm},
Super Ricci flows and functional inequalities,
in preparation

\bibitem[La]{Lak} \textsc{S.Lakzian},
Differential Harnack estimates for heat equation under Finsler-Ricci flow,
arXiv:1404.0190


 
\bibitem[LM]{LaM} \textsc{S.Lakzian, M.Munn},
Super Ricci flow for disjoint unions,
 arXiv:1211.2792 

\bibitem[Le]{Le} \textsc{M.Ledoux}, The geometry of Markov diffusion generators. ETHZ,  1998





\bibitem[Lo1]{Lo1}
\textsc{J.Lott},
Optimal transport and Perelman's reduced volume,
Calc. Var. PDEs 36: 49–84, 2009



\bibitem[Lo2]{Lo2}
\textsc{J.Lott},
Ricci measure for some singular Riemannian metrics,
arXiv:1503.04725 



\bibitem[LV1]{LV1}
\textsc{J.Lott, C.Villani},
Ricci curvature for metric-measure spaces via optimal transport,  Annals of Math.\,169: 903-991, 2009

\bibitem[LV2]{LV2}
\textsc{J.Lott, C.Villani},
Weak curvature conditions and functional inequalities,
J.\ Funct.\,Anal.\,245: 311–333, 2007

\bibitem[MRS]{MRS}
\textsc{R.Mazzeo, Y.Rubinstein, N.Sesum}, 
Ricci flow on surfaces with conic
singularities, 
Analysis \& PDE,  2015+


\bibitem[McT]{MT}
\textsc{R.McCann, P.Topping},
Ricci flow, entropy and optimal transportation,
American J.\,Math.132: 711-730, 2010



\bibitem[MoT]{MoTian} \textsc{J.W.Morgan, G.Tian},
Ricci flow and the Poincar\'e conjecture, arXiv:math/0607607


 \bibitem[Na]{Na} \textsc{A.Naber},
Characterizations of bounded Ricci curvature on smooth and nonsmooth spaces,
 arXiv:1306.6512 


\bibitem[OS]{OS1}
\textsc{S.Ohta, K.T.Sturm},
{Heat flow on Finsler manifolds}, Comm.\,Pure Appl.\,Math.\,62:\,1386-1433, 2009


\bibitem[Ot]{Ot} \textsc{F.Otto}, The geometry of dissipative evolution equations: the porous medium equation, Comm. PDEs 26: 101-174, 2001

\bibitem[OV]{OV} \textsc{F.Otto, C.Villani}, Generalization of an inequality by Talagrand and links with the logarithmic Sobolev inequality, J. Funct. Anal. 173:  361-400, 2000

\bibitem[Pe1]{Pe1} \textsc{G.Perelman}, The entropy formula for the Ricci flow and its geometric applications, arXiv:math/0211159 

\bibitem[Pe2]{Pe2} \textsc{G.Perelman}, Ricci flow with surgery on three-manifolds, arXiv:math/0303109


\bibitem[Pe3]{Pe3} \textsc{G.Perelman},  
Finite extinction time for the solutions to the Ricci flow on certain three-manifolds,
arXiv:math/0307245 





\bibitem[PSSW]{PSSW} \textsc{D.Phong, J.Song, J.Sturm, X.Wang}, 
The Ricci flow on the sphere with
marked points, arXiv:1407.1118 

\bibitem[Raj]{Raj} \textsc{T.Rajala},
 Interpolated measures with bounded density in metric spaces satisfying the curvature-dimension conditions of Sturm,
J.\,Funct.\,Anal.\,263: 896-924, 2012

\bibitem[vRS]{vRS}
\textsc{M.K.v.Renesse., K.T.Sturm}, {Transport inequalities, gradient estimates,
entropy and Ricci curvature.} Comm.\,Pure Appl.\,Math.\,58:
923-940, 2005



\bibitem[Si1]{Si1}
\textsc{M.Simon}, 
Deformation of $C^0$ Riemannian metrics in the direction of their
Ricci curvature, Comm.\,Anal.\,Geom.\,10: 1033-1074, 2002

 \bibitem[Si2]{Si2}
\textsc{M.Simon}, 
Ricci flow of non-collapsed three manifolds whose Ricci curvature
is bounded from below, J.\,Reine Angew.\,Math.\,662: 59-94, 2012


\bibitem[St1]{St0}
\textsc{K.T.Sturm},
Convex functionals of probability measures and nonlinear diffusions on manifolds, J.\,Math.\,Pures.\,Appl\, 84: 149-168, 2005

\bibitem[St2]{St1}
\textsc{K.T.Sturm},
{On the geometry of metric measure spaces. I}, Acta Math.\,96: 65-131 2006



\bibitem[St3]{St2}
\textsc{K.T.Sturm},
{On the geometry of metric measure spaces. II}, Acta Math.\,196: 133-177 2006









\bibitem[St5]{St3}\textsc{K.T.Sturm},
{Metric measure spaces with variable Ricci bounds and couplings of Brownian motions}.
In `Festschrift Masatoshi Fukushima'
(edt.  Z.-Q. Chen et al.), 
 World Scientific 2014
 
\bibitem[St5]{St4}\textsc{K.T.Sturm},
	{Ricci tensor for diffusion operators and curvature-dimension inequalities under conformal transformations and time changes}, arXiv:1401.0687 




\bibitem[St6]{St5}
\textsc{K.T.Sturm},
Synthetic upper Ricci bounds for metric measure spaces,
in preparation.


	 
\bibitem[To1]{To1}
\textsc{P.Topping},
L-optimal transportation for Ricci flow, J.\,Reine Angew.\,Math.\,636: 93--122, 2009
	 
\bibitem[To2]{To2}
\textsc{P.Topping},
Ricci flow: The foundations via optimal transportation. In `Optimal Transportation, Theory and Applications.' LMS lecture notes series, vol. 413, 2014



\bibitem[Yi]{Yin}
\textsc{H.Yin}, 
Ricci flow on surfaces with conical singularities, I \& II, J.\,Geom.\,Anal.\,20: 970-995 \&
arXiv:1305.4355 









\end{thebibliography}
\end{document}